\documentclass[a4paper, 10pt]{amsart}
\usepackage{amsmath}
\usepackage{amssymb}

\theoremstyle{plain}
\newtheorem{theorem}{\bf Theorem}[section]
\newtheorem{proposition}[theorem]{\bf Proposition}
\newtheorem{lemma}[theorem]{\bf Lemma}
\newtheorem{corollary}[theorem]{\bf Corollary}

\theoremstyle{definition}
\newtheorem{example}[theorem]{\bf Example}
\newtheorem{examples}[theorem]{\bf Examples}
\newtheorem{definition}[theorem]{\bf Definition}

\newcommand{\N}{\mathbb N}
\newcommand{\Z}{\mathbb Z}
\newcommand{\R}{\mathbb R}
\newcommand{\Q}{\mathbb Q}

\newcommand{\Fc}{\mathcal F}
\newcommand{\vp}{\mathsf v}
\newcommand{\ZZ}{\mathsf Z}
\newcommand{\und}{\;\mbox{ and } \;}
\newcommand{\be}{\begin{equation}}
\newcommand{\ee}{\end{equation}}
\newcommand{\ber}{\begin{eqnarray}}
\newcommand{\eer}{\end{eqnarray}}
\newcommand{\nn}{\nonumber}

\newcommand{\B}{\mathcal B}
\newcommand{\dd}{\mathsf d}
\newcommand{\red}{{\text{\rm red}}}
\newcommand{\mon}{{\text{\rm mon}}}

 \DeclareMathOperator{\lcm}{lcm}
 
 \DeclareMathOperator{\supp}{supp}

\newcommand{\BF}{\text{\rm BF}}

\newcommand{\DP}{\negthinspace : \negthinspace}

\renewcommand{\t}{\, | \,}
\renewcommand{\time}{\negthinspace \times \negthinspace}

\numberwithin{equation}{section}

\begin{document}

\title[On the arithmetic of Krull monoids with infinite cyclic class group]
{On the arithmetic of Krull monoids \\ with infinite cyclic class group}

\address{Institut f\"ur Mathematik und Wissenschaftliches Rechnen\\
Karl--Fran\-zens--Universit\"at Graz\\
Heinrichstra{\ss}e 36\\
8010 Graz, Austria}

\email{alfred.geroldinger@uni-graz.at, diambri@hotmail.com,
wolfgang.schmid@uni-graz.at}

\address{University of California at Berkeley, Department of Mathematics,
Berkeley, California 94720, USA}

\email{gschaeff@math.berkeley.edu}

\author{A.~Geroldinger \and D.~J.~Grynkiewicz \and  G.~J.~Schaeffer \and W.~A.~Schmid}

\thanks{This work was supported by the Austrian Science Fund FWF
(Project Numbers )}



\begin{abstract}
Let $H$ be a Krull monoid with infinite cyclic class group $G$ and
let $G_P \subset G$ denote the set of classes containing prime
divisors. We study under which conditions on $G_P$  some of the
main finiteness properties of factorization theory---such as local
tameness, the finiteness and rationality of the elasticity, the structure theorem for sets of lengths, the finiteness of the
catenary degree, and the existence of monotone and of near monotone chains of factorizations---hold in
$H$. In many cases, we derive explicit characterizations.
\end{abstract}

\maketitle


\section{Introduction}

By an atomic monoid, we mean a commutative cancellative semigroup
with unit element such that every non-unit has a factorization as a
finite product of atoms (irreducible elements). The multiplicative
monoid consisting of the nonzero elements from a noetherian domain
is such a monoid. Let $H$ be an atomic monoid. Then $H$ is factorial
(that is, every non-unit has a unique factorization into atoms) if
and only if $H$ is a Krull monoid with trivial class group. The
first objective of factorization theory is to describe the various
phenomena related to the non-uniqueness of factorizations. This is
done by a variety of arithmetical invariants such as sets of lengths
(including all invariants derived from them, such as the elasticity
and the set of distances) and by the catenary and tame degrees of
the monoids. The second main objective is to then characterize the
finiteness (or even to find the precise value) of these arithmetical
invariants in terms of classical algebraic invariants of the objects
under investigation. To illustrate this, we mention some results of
this type (a few classical ones and some very recent). The following
result by Carlitz (achieved in 1960) is considered as a starting
point of factorization theory: the ring of integers $\mathfrak o_K$
of an algebraic number field has elasticity $\rho (\mathfrak o_K) =
1$ if and only if its class group has at most two elements (recall
that, by definition, $H$ is half-factorial if and only if its
elasticity $\rho (H) = 1$). A non-principal order $\mathfrak o$ in
an algebraic number field has finite elasticity if and only if, for
every prime ideal $\mathfrak p$ containing the conductor, there is
precisely one prime ideal $\overline{\mathfrak p}$ in the principal
order $\overline{ \mathfrak o}$ such that $\overline{\mathfrak p}
\cap \mathfrak o = \mathfrak p$. This result (achieved by
Halter-Koch in 1995) has far reaching generalizations (achieved by
Kainrath) to finitely generated domains and to various classes of
Mori domains satisfying natural finiteness conditions (for all this,
see \cite{Ca60, HK95b, Ka05a,Ka10a}).

An integral domain is a Krull domain if and only if its
multiplicative monoid of nonzero elements is a Krull monoid, and a
noetherian domain is Krull if and only if it is integrally closed. A
reduced Krull monoid is uniquely determined by its class group and
by the distribution of prime divisors in the classes (see Lemma
\ref{3.3} for a precise statement). Suppose $H$ is a Krull monoid
with class group $G$ and let $G_P \subset G$ denote the set of
classes containing prime divisors. Suppose that $G_P = G$. In that
case, it is comparatively easy to show that any of  the arithmetical
invariants under discussion is finite if and only if $G$ is finite
(the precise values of arithmetical invariants---when $G$
is finite---are studied by methods of Additive and Combinatorial
Number Theory; see \cite[Chapter 6]{Ge-HK06a} or \cite{Ge09a} for a
survey on this direction). However, only very little is known so far
on the arithmetic of $H$ when $G$ is infinite and $G_P$ is
a proper subset of $G$.

The present paper provides an in-depth study of the arithmetic of
Krull monoids having an infinite cyclic class group. This situation
was studied first by Anderson, Chapman and Smith in 1994
\cite{An-Ch-Sm94c}, then by Hassler \cite{Ha02c}, and the most
recent progress (again due to Chapman et al.) was achieved in
\cite{B-C-R-S-S10}. We continue this work. The arithmetical
properties under investigation are discussed in Section \ref{2} and
at the beginning of Section \ref{5}. The required material on Krull
monoids, together with a list of relevant examples, is summarized in
Section \ref{3}. Our main results are Theorems \ref{main-theorem-I},
\ref{main-theorem-II}, \ref{STSL_thm} and Corollary \ref{final-cor}.
Along the way, we introduce new methods (see the proofs of
Proposition \ref{rel_prop} and of Theorem \ref{prop-chain}) and
solve an old problem proposed in 1994 (see the equivalence of $(a)$
and $(e)$ in Theorem \ref{main-theorem-I}). A more detailed
discussion of the main results is shifted to the relevant sections
where we have the required terminology at our disposal.

\section{Preliminaries} \label{2}

Our notation and terminology are consistent with \cite{Ge-HK06a}. We
briefly gather some key notions. We denote by $\N$ the set of
positive integers, and we put $\N_0 = \N \cup \{0\}$. For real
numbers $a, b \in \R$, we set $[a, b] = \{ x \in \Z \mid a \le x \le
b \}$. For a subset $X$ of (possibly negative) integers, we use $\gcd X$ and $\lcm X$ to denote the greatest common divisor and least common multiple, respectively, and their values are always chosen to be nonnegative regardless of the sign of the input.

\medskip
Let $L, L' \subset \Z$. We set $-L = \{-a \mid a \in L \}$, $L^+ = L
\cap  \N$ and $L^- = L \cap (-\N)$. We denote by $L+L' = \{a+b \mid
a \in L,\, b \in L' \}$ their {\it sumset}. If $\emptyset \ne L
\subset \N$, we call
\[
\rho (L) = \sup \Bigl\{ \frac{m}{n} \; \Bigm| \; m, n \in L
\Bigr\} = \frac{\sup L}{\min L} \, \in \Q_{\ge 1} \cup \{ \infty
\}
\]
the \ {\it elasticity} \ of $L$, and we set
$\rho (\{0\}) = 1$.
Distinct elements $k, l \in L$ are called \emph{adjacent} if $L \cap [\min\{k,l\},\max\{k,l\}]=\{k,l\}$.
A positive integer $d \in \N$ is called a \ {\it distance} \ of $L$ \ if there exist adjacent elements $k,l \in L$ with $d=|k-l|$.  We denote by \ $\Delta (L)$ \ the {\it set of
      distances} of $L$. Note that $\Delta (L) = \emptyset$ if and only if $|L| \le 1$,
and that $L$ is an arithmetical progression with difference $d \in \N$ if and only if $\Delta (L)
\subset \{d\}$. We need the following generalization of an arithmetical progression.

Let $d \in \N$, \ $M \in \N_0$ \ and \ $\{0,d\} \subset \mathcal D
\subset [0,d]$. Then $L$ is called an {\it almost
arithmetical multiprogression} \ ({\rm AAMP} \ for
      short) \ with \ {\it difference} \ $d$, \ {\it period} \ $\mathcal D$,
      \  and \ {\it bound} \ $M$, \ if
\[
L = y + (L' \cup L^* \cup L'') \, \subset \, y + \mathcal D + d \Z
\]
where
\begin{itemize}
\item  $L^*$ is finite and nonempty with $\min L^* = 0$ and $L^* =
       (\mathcal D + d \Z) \cap [0, \max L^*]$

\item  $L' \subset [-M, -1]$ \ and \ $L'' \subset \max L^* + [1,M]$
\item      $y \in \Z$.
\end{itemize}
Note that an AAMP is finite and nonempty. An AAMP with period $\{0,d\}$ is called an \emph{almost arithmetical progression} (AAP for short).

\medskip
By a {\it monoid}, we mean a commutative, cancellative semigroup with
unit element; we denote the unit element by $\boldsymbol{1}$. Let $H$ be a monoid. We denote by $\mathcal A (H)$ the
set of atoms (irreducible elements) of $H$, by $H^{\times}$ the
group of invertible elements, and by $H_{\red} = \{ a H^{\times} \mid
a \in H \}$ the associated reduced monoid of $H$. We call elements
$a,b \in H$ associated (in symbols $a\simeq b$) if
$aH^{\times}=bH^{\times}$. We say that $H$ is reduced if
$|H^{\times}| = 1$. We denote by $\mathsf q (H)$ a quotient group of
$H$ with $H \subset \mathsf q (H)$, and for a prime element $p \in
H$, let $\mathsf v_p \colon \mathsf q (H) \to \Z$ be the $p$-adic
valuation. For a subset $H_0 \subset H$, we denote by $[H_0] \subset
H$ the submonoid generated by $H_0$ and by  $\langle H_0 \rangle
\subset \mathsf q (H)$ the subgroup generated by $H_0$.
For elements $a, b \in H$, we frequently use, in case  $a \t b$, the notation $a^{-1}b$ to denote
the element $c \in H$ with $ac=b$;
yet, we mention explicitly if we shift our investigations from $H$ to the quotient group of $H$.

For a set $P$, we denote by $\mathcal F (P)$ the \ {\it free
$($abelian$)$ monoid} \ with basis $P$. Then every $a \in \mathcal F
(P)$ has a unique representation in the form
\[
a = \prod_{p \in P} p^{\mathsf v_p(a) } \quad \text{with} \quad
\mathsf v_p(a) \in \N_0 \ \text{ and } \ \mathsf v_p(a) = 0 \ \text{
for almost all } \ p \in P \,.
\]
We call $|a|= \sum_{p \in P}\mathsf v_p(a)$ the \emph{length} of $a$.

The free monoid \ $\mathsf Z (H) = \mathcal F \bigl( \mathcal
A(H_\red)\bigr)$ \ is called the \ {\it factorization monoid} \ of
$H$, and the unique homomorphism
\[
\pi \colon \mathsf Z (H) \to H_{\red} \quad \text{satisfying} \quad
\pi (u) = u \quad \text{for each} \quad u \in \mathcal A(H_\red)
\]
is called the \ {\it factorization homomorphism} \ of $H$. For $a
\in H$ and $k \in \N$, the set
\[
\begin{aligned}
\mathsf Z_H (a) = \mathsf Z (a)  & = \pi^{-1} (aH^\times) \subset
\mathsf Z (H) \quad
\text{is the \ {\it set of factorizations} \ of \ $a$} \,, \\
\mathsf Z_k (a) & = \{ z \in \mathsf Z (a) \mid |z| = k \} \quad
\text{is the \ {\it set of factorizations} \ of \ $a$ of length \
$k$}, \quad \text{and}
\\
\mathsf L_H (a) = \mathsf L (a) & = \bigl\{ |z| \, \bigm| \, z \in
\mathsf Z (a) \bigr\} \subset \N_0 \quad \text{is the \ {\it set of
lengths} \ of $a$}  \,.
\end{aligned}
\]
By definition, we have \ $\mathsf Z(a) = \{\boldsymbol{1}\}$ and $\mathsf L (a) =
\{0\}$ for all $a \in H^\times$. The monoid $H$ is called
\begin{itemize}
\item  {\it atomic} \  if  \ $\mathsf Z(a) \ne \emptyset$ \
      for all \ $a \in H$.

\item  a {\it \BF-monoid} (a bounded factorization monoid) \ if \
      $\mathsf L (a)$ is finite and nonempty for all \ $a \in H$.

\item {\it half-factorial} \ if \ $|\mathsf L (a)| = 1$ \ for all $a \in H$.
\end{itemize}

\smallskip
We repeat the arithmetical concepts which are used throughout the
whole paper. Some more specific notions will be recalled at the
beginning of Section \ref{5}. Let $H$ be atomic and $a \in H$. Then
\ $\rho (a) = \rho \bigl( \mathsf L (a) \bigr)$ \ is called the {\it
elasticity} of $a$, and the {\it elasticity} of $H$ is defined as
\[
\rho (H) = \sup \{ \rho (b) \mid b \in H \} \in \R_{\ge
1} \cup \{\infty\} \,.
\]
We say that $H$ has \ {\it accepted elasticity } \ if there exists some $b \in H$ with
$\rho(b)=\rho(H)$.

Let $k \in \mathbb N$. If $H \ne H^{\times}$, then
\[
      \mathcal V_k (H) \ = \ \bigcup_{k \in \mathsf{L}(a),  a \in H}
      \mathsf{L}(a)
\]
is the union of all sets of lengths containing $k$. When $
H^\times=H$, we set $\mathcal V_k (H) =\{k\}$. In both cases, we define
$\rho_k (H) = \sup \mathcal V_k (H)$  and $\lambda_k (H)= \min
\mathcal V_k (H)$. Clearly, we have $\mathcal V_1 (H) =\{1\}$ and $k
\in \mathcal V_k (H)$. By its definition, $H$ is half-factorial if and
only if $\mathcal V_k (H) = \{k\}$ for each $k \in \N$.

We denote by
\[
\Delta (H) \ = \ \bigcup_{b \in H } \Delta \bigl( \mathsf L (b) \bigr) \ \subset \N
\]
the {\it set of distances} of $H$, and by
\(
\mathcal{L}(H)  = \{\mathsf{L} (b) \mid b \in H \}
\)
the {\it system of sets of lengths} of $H$.

\smallskip
Let $z,\, z' \in \mathsf Z (H)$. Then we can write
\[
z = u_1 \cdot \ldots \cdot u_lv_1 \cdot \ldots \cdot v_m \quad
\text{and} \quad z' = u_1 \cdot \ldots \cdot u_lw_1 \cdot \ldots
\cdot w_n\,,
\]
where  $l,\,m,\, n\in \N_0$ and $u_1, \ldots, u_l,\,v_1, \ldots,v_m,\,
w_1, \ldots, w_n \in \mathcal A(H_\red)$ are such that
\[
\{v_1 ,\ldots, v_m \} \cap \{w_1, \ldots, w_n \} = \emptyset\,.
\]
Then $\gcd(z,z')=u_1\cdot\ldots\cdot u_l$, and we call
\[
\mathsf d (z, z') = \max \{m,\, n\} = \max \{ |z \gcd (z, z')^{-1}|,
|z' \gcd (z, z')^{-1}| \} \in \N_0
\]
the {\it distance} between $z$ and $z'$. If $\pi (z) = \pi (z')$ and $z
\ne z'$, then
\begin{equation}\label{E:Dist}
2 +
      \bigl| |z |-|z'| \bigr| \le \mathsf d (z, z')
\end{equation}
by \cite[Lemma 1.6.2]{Ge-HK06a}. For subsets $X, Y \subset \mathsf Z (H)$,
we set
\[
\mathsf d (X, Y) = \min \{ \mathsf d (x, y ) \mid x \in X, \, y \in
Y \} \,,
\]
and thus $X \cap Y \ne \emptyset$ if and only if $\mathsf d (X, Y) =
0$.

\medskip
We recall the concepts of the (monotone) catenary and tame degrees
(see also the beginning of Section \ref{7}). The \ {\it
 catenary degree} \ $\mathsf c(a)$ \ of the element
$a$ is the smallest $N \in \N_0 \cup \{ \infty\}$ \ such that, for
      any two factorizations \ $z,\, z'$ \ of \ $a$,  there exists a finite
sequence \ $z = z_0\,, \, z_1\,, \ldots, z_k = z'$ \ of
factorizations of \ $a$ \ such that \ $\mathsf d (z_{i-1}, z_i) \le
N \quad \text{for all} \quad i \in [1,k]$. The \ {\it monotone
catenary degree} \ $\mathsf{c}_{\mon} (a)$ is defined in the same
way with the additional restriction  that $|z_0| \le \ldots \le
|z_k|$ or $|z_0| \ge \ldots \ge |z_k|$. We say that the two
factorizations $z$ and $z'$ can be concatenated by a (monotone)
$N$-chain if a sequence fulfilling the above conditions exists.
Moreover,
\[
\mathsf c(H) = \sup \{ \mathsf c(b) \mid b \in H\} \in \N_0 \cup
\{\infty\} \quad \text{and} \quad \mathsf c_{\mon} (H) = \sup \{
\mathsf c_{\mon} (b) \mid b \in H\} \in \N_0 \cup \{\infty\} \quad
\,
\]
denote  the \ {\it catenary degree} \ and the \ {\it monotone
catenary degree} of $H$. Clearly, we have $\mathsf c (a) \le \mathsf
c_{\mon} (a)$ for all $a \in H$, as well as $\mathsf c (H) \le
\mathsf c_{\mon} (H)$, and \eqref{E:Dist} implies that $2 + \sup
\Delta (H) \le \mathsf c (H)$.

For $x \in \mathsf Z (H)$, let  $\mathsf t
      (a,x) \in \N_0 \cup \{\infty\}$ denote the
      smallest $N\in \N_0 \cup \{\infty\}$ with the following property{\rm \,:}
      \begin{enumerate}
      \smallskip
      \item[] If  $\mathsf Z(a) \cap x\mathsf Z(H) \ne \emptyset$ and
              $z \in \mathsf Z(a)$, then there exists $z' \in
              \mathsf Z(a) \cap x\mathsf Z(H)$ such that  $\mathsf d (z,z') \le
              N$.
      \end{enumerate}
      For  subsets $H' \subset H$ and $X \subset \mathsf Z(H)$, we
      define
      \[
      \mathsf t (H',X) = \sup \big\{ \mathsf t (b,x) \, \big| \, b \in H',  x \in
      X\big\} \in \N_0 \cup \{\infty\}\,.
      \]
      $H$ is called \ {\it locally tame} \ if  \ $\mathsf t (H,u) <
      \infty$ \ for all $u \in \mathcal A(H_{\red})$ (see the
      beginning of Section \ref{4} and Definition \ref{def-tamely-gen}).

\section{Krull monoids: Basic Properties and Examples} \label{3}

The theory of Krull monoids is presented in detail in the monographs
\cite{HK98, Gr01, Ge-HK06a}. Here we first gather the required
terminology. After that, we  recall some facts concerning transfer
homomorphisms, since the arithmetic of Krull monoids is studied via
such homomorphisms. In particular, we  deal with  block
homomorphisms (which are transfer homomorphisms) from Krull monoids
into the associated block monoids. At the end of this section, we
discuss examples of Krull monoids with infinite cyclic class group.

\smallskip
\noindent {\bf Krull monoids.} Let \ $H$ \ and \ $D$ \ be monoids. A
monoid homomorphism $\varphi \colon H \to D$ \  is called
\begin{itemize}
\smallskip
\item a  {\it divisor homomorphism} if $\varphi(a)\mid\varphi(b)$ implies that $a \t b$  for all $a,b \in H$.

\smallskip
\item  {\it cofinal} \ if for every $a \in D$ there exists some $u
      \in H$ such that $a \t \varphi(u)$.

\smallskip
\item  a {\it divisor theory} (for $H$) if $D = \mathcal F (P)$
for some set $P$, $\varphi$ is a divisor homomorphism, and for every
$p \in P$ (equivalently for every $a \in \mathcal{F}(P)$), there exists a finite
subset $\emptyset \ne X \subset H$ satisfying $p = \gcd \bigl(
\varphi(X) \bigr)$.
\end{itemize}
Note that, by definition, every divisor theory is cofinal. We call
$\mathcal{C}(\varphi)=\mathsf q (D)/ \mathsf q (\varphi(H))$ the
class group of $\varphi $ and use additive notation for this group.
For \ $a \in \mathsf q(D)$, we denote by \ $[a] = [a]_{\varphi} = a
\,\mathsf q(\varphi(H)) \in \mathsf q (D)/ \mathsf q (\varphi(H))$ \
the class containing \ $a$. We recall that $\varphi$ is cofinal if
and only if $\mathcal{C}(\varphi) = \{[a]\mid a \in D \}$, and if
$\varphi$ is a divisor homomorphism, then $\varphi(H)= \{a \in D
\mid [a]=[1]\}$. If $\varphi \colon H \to \mathcal F (P)$ is a
cofinal divisor homomorphism, then
\[
G_P = \{[p] = p \mathsf q (\varphi(H)) \mid p \in P \} \subset
\mathcal{C}(\varphi)
\]
is called the \ {\it  set of classes containing prime divisors}, and
we have $[G_P] = \mathcal{C}(\varphi)$ (for a converse, see
Lemma \ref{lem_char}). If $H \subset D$ is a submonoid, then $H$ is called
\emph{cofinal} (\emph{saturated}, resp.) in $D$ if the imbedding $H
\hookrightarrow D$ is cofinal (a divisor homomorphism, resp.).

The monoid $H$ is called a {\it Krull monoid} if it satisfies one of
the following equivalent conditions (\cite[Theorem 2.4.8]{Ge-HK06a};
see \cite{Ki-Ki-Pa07a} for recent progress){\rm \,:}
\begin{itemize}
\item $H$ is $v$-noetherian and completely integrally closed.

\smallskip
\item $H$ has a divisor theory.

\smallskip
\item $H_{\red}$ is a saturated submonoid of a free monoid.
\end{itemize}
In particular, $H$ is a Krull monoid if and only if $H_{\red}$ is a
Krull monoid. Let $H$ be a Krull monoid. Then a divisor theory
$\varphi \colon H \to \mathcal F (P)$ is unique up to unique
isomorphism. In particular, the class group $\mathcal C ( \varphi)$ defined via a divisor theory of $H$
and the subset of classes containing prime divisors depend only on
$H$. Thus it is called the {\it class group} of $H$ and is denoted
by $\mathcal C (H)$. In fact, for every Krull monoid the map, defined via assigning to each $a\in H$ the principal ideal it generates, from $H$ to $\mathcal{I}_v^{*}(H)$---the monoid of $v$-invertible $v$-ideals of $H$, which is a free monoid with basis $\mathfrak X (H)$---is a divisor theory,
and thus $\mathcal{C}(H)$ is the $v$-class group of $H$ (up to isomorphism).

\smallskip
\noindent {\bf Transfer homomorphisms.} We recall some of the main
properties which are needed in the sequel (details can be found in
\cite[Section 3.2]{Ge-HK06a}).

\begin{definition} \label{3.1}
A monoid homomorphism \ $\theta \colon H \to B$ is called a \ {\it
transfer homomorphism} \ if it has the following properties:

\smallskip

\begin{enumerate}
\item[]
\begin{enumerate}
\item[{\bf (T\,1)\,}] $B = \theta(H) B^\times$ \ and \ $\theta
^{-1} (B^\times) = H^\times$.

\smallskip

\item[{\bf (T\,2)\,}] If $u \in H$, \ $b,\,c \in B$ \ and \ $\theta
(u) = bc$, then there exist \ $v,\,w \in H$ \ such that \ $u = vw$,
\ $\theta (v) \simeq b$ \ and \ $\theta (w) \simeq c$.
\end{enumerate}\end{enumerate}
\end{definition}

\medskip
\noindent
Every transfer homomorphism $\theta$ gives rise to a
unique extension $\overline \theta \colon \mathsf Z(H) \to \mathsf
Z(B)$ satisfying
\[\qquad \quad
\overline \theta (uH^\times) = \theta (u)B^\times \quad \text{for
each} \quad u \in \mathcal A(H)\,.
\]
For $a \in H$, we denote by \ $\mathsf c (a, \theta)$ \ the smallest
$N \in \N_0 \cup \{\infty\}$ with the following property:

\smallskip

\begin{enumerate}
\item[]
If $z,\, z' \in \mathsf Z_H (a)$ and $\overline \theta (z) =
\overline \theta (z')$, then there exist some $k \in \N_0$ and
factorizations $z=z_0, \ldots, z_k=z' \in \mathsf Z_H (a)$ such that
\ $\overline \theta (z_i) = \overline \theta (z)$ and \ $\mathsf d
(z_{i-1}, z_i) \le N$ for all $i \in [1,k]$ \ (that is, $z$ and $z'$
can be concatenated by an $N$-chain in the fiber \ $\mathsf Z_H (a)
\cap \overline \theta ^{-1} (\overline \theta (z)$)\,).
\end{enumerate}

\smallskip\noindent
Then
\[
\mathsf c (H, \theta) = \sup \{\mathsf c (a, \theta) \mid a \in H \}
\in \N_0 \cup \{\infty\}
\]
denotes the {\it catenary degree in the fibres}.

\begin{lemma} \label{3.2}
Let $\theta \colon H \to B$ and $\theta' \colon B \to B'$ be
transfer homomorphisms of atomic monoids.
\begin{enumerate}
\item For every $a \in H$, we have $\overline \theta(\mathsf Z_H(a)) = \mathsf Z_B(\theta(a))$
and \ $\mathsf L_H(a) = \mathsf L_B(\theta(a))$.

\smallskip
\item $\mathsf c (B) \le \mathsf c (H) \le \max \{ \mathsf c (B),
      \mathsf c (H, \theta) \}$, \ $\mathsf c_{\mon} (B) \le \mathsf c_{\mon} (H) \le \max \{ \mathsf c_{\mon} (B),
      \mathsf c (H, \theta) \}$ and $\delta (B) = \delta (H)$.

\smallskip
\item For every \ $a \in H$ \ and all \ $k, l \in \mathsf L (a)$, \ we have \
      $\mathsf d \bigl( \mathsf Z_k (a), \mathsf Z_l (a) \bigr) =
      \mathsf d \bigl( \mathsf Z_k \bigl( \theta(a)\bigr), \mathsf Z_l \bigl( \theta (a) \bigr)
      \bigr)$.

\smallskip
\item For every $a\in H$, we have $\mathsf{c}(a, \theta' \circ \theta) \le
 \max\{\mathsf{c}(a,\theta), \mathsf{c}(\theta(a),\theta')\}$.
 \newline
In particular,  $\mathsf{c}(H, \theta' \circ \theta) \le
\max\{\mathsf{c}(H, \theta), \mathsf{c}(B, \theta')\}$.
\end{enumerate}
\end{lemma}

\begin{proof}
1.  This follows from \cite[Proposition 3.2.3]{Ge-HK06a}.

\smallskip
2. The first statement follows from Theorem 3.2.5.4, the second from
Lemma 3.2.6 in \cite{Ge-HK06a}, and the third from \cite[Theorem
3.14]{Ge-Gr09b}.

\smallskip
3. Let $a \in H$ and $k, l \in \mathsf L (a)$. If $z, z' \in \mathsf
Z (a)$ with $|z| = k$ and $|z'| = l$, then $|\overline \theta (z)| =
k$, $|\overline \theta (z')| = l$ and $\mathsf d \bigl(\overline
\theta (z),\, \overline \theta(z') \bigr) \le \mathsf d(z,z')$, which
implies that $\mathsf d \bigl( \mathsf Z_k \bigl( \theta(a)\bigr),
\mathsf Z_l \bigl( \theta (a) \bigr)
      \bigr) \le \mathsf d \bigl( \mathsf Z_k (a), \mathsf Z_l (a)
      \bigr)$. To verify the reverse inequality, let
      $\overline z_1, \overline z_2 \in \mathsf Z ( \theta (a))$
      be given. We pick any $z_1 \in \mathsf Z (a)$ with $\overline
      \theta (z_1) = \overline z_1$. By \cite[Proposition
      3.2.3.3.(c)]{Ge-HK06a}, there exists a factorization $z_2 \in
      \mathsf Z (a)$ such that $\overline{\theta} (z_2) =
      \overline z_2$ and $\mathsf d (z_1, z_2) = \mathsf d (
      \overline z_1, \overline z_2)$. Since $|z_i| = |\overline
      z_i|$ for $i \in \{1,2\}$, it follows that $\mathsf d
\bigl( \mathsf Z_k (a), \mathsf Z_l (a) \bigr) \le
      \mathsf d \bigl( \mathsf Z_k \bigl( \theta(a)\bigr), \mathsf Z_l \bigl( \theta (a) \bigr)
      \bigr)$.

\smallskip
4. We recall that $\theta' \circ \theta$ is a transfer homomorphism
(see the paragraph after \cite[Definition 3.2.1]{Ge-HK06a}). Let
$a\in H$. Let $z,z'\in \mathsf{Z}_H(a)$ with $\overline{\theta'
\circ \theta}(z)=\overline{\theta' \circ \theta}(z')$. Let
$\overline{z}= \overline{\theta}(z)$ and $\overline{z'}=
\overline{\theta}(z')$. We have $\overline{z},\overline{z'}\in
\mathsf{Z}_{B}(\theta(a))$ and
$\overline{\theta'}(\overline{z})=\overline{\theta'}(\overline{z'})$.
Thus, by the definition of $\mathsf{c}(\theta(a),\theta')$, there
exist some $k\in \mathbb{N}_0$ and $\overline{z}=\overline{z_0},
\ldots, \overline{z_k}=\overline{z'} \in \mathsf Z_B (\theta(a))$
such that \ $\overline{\theta'} (\overline{z_i}) =
\overline{\theta'} (\overline{z})$ and \ $\mathsf d
(\overline{z_{i-1}}, \overline{z_i}) \le
\mathsf{c}(\theta(a),\theta')$ for each $i \in [1,k]$. Let $z_0=z$.
Again, by \cite[Proposition 3.2.3.3.(c)]{Ge-HK06a}, for each $i<k$,
there exists some factorization $z_{i+1} \in \mathsf{Z}_H (a)$ such
that $\overline{\theta} (z_{i+1}) = \overline{z_{i+1}}$ and $\mathsf
d (z_i, z_{i+1}) = \mathsf d (\overline{z_{i}},
\overline{z_{i+1}})$.

Now, we have $\overline{\theta}(z_k)= \overline{z'}=
\overline{\theta}(z')$. Thus, by the definition of $\mathsf{c}(a,
\theta)$, there exist some $l \in \mathbb{N}_0$ and $z_k=y_0,
\ldots, y_l=z' \in \mathsf Z_H (a)$ such that \ $\overline \theta
(y_i) = \overline \theta (z')$ and \ $\mathsf d (y_{i-1}, y_i) \le
\mathsf{c}(a, \theta)$ for each $i \in [1,l]$. Since $\overline
\theta (y_i) = \overline \theta (z')$ clearly implies
$\overline{\theta'\circ \theta} (y_i) = \overline{\theta'\circ
\theta} (z')$, we get that the $\max \{\mathsf{c}(\theta(a),
\theta'), \mathsf{c}(a, \theta)\}$-chain $z=z_0, \dots, z_k=y_0,
\dots, y_l=z'$ has the required properties.
\end{proof}

\smallskip
\noindent {\bf Monoids of zero-sum sequences.} Let \ $G$ \ be an
additive abelian group, \ $G_0 \subset G$ \ a subset and $\mathcal F
(G_0)$ the free monoid with basis $G_0$. According to the tradition
of combinatorial number theory,  the elements of $\mathcal F(G_0)$
are called \ {\it sequences } over \ $G_0$. Thus a sequence $S \in
\mathcal F (G_0)$ will be written in the form
\[
S = g_1 \cdot \ldots \cdot g_l = \prod_{g \in G_0} g^{\mathsf v_g
(S)} \,,
\]
and we use all the notions (such as the length) as in general free
monoids. Again using traditional language,
 we refer to \ $\mathsf
v_g(S)$ \ as the {\it multiplicity} of $g$ in $S$ and refer to a
divisor of $S$ as a {\it subsequence}.  If $T|S$, then $T^{-1}S$
denotes the subsequence of $S$ obtained by removing the terms of
$T$. We call the set \ $\supp (S) = \{ g_1, \ldots, g_l\}\subset
G_0$ \ the \ {\it support} \ of $S$, \ $\sigma (S) = g_1+ \ldots +
g_l\in G$ \ the \ {\it sum} \ of $S$, \ and define \ber\nn \Sigma
(S) &=& \Bigl\{ \sum_{i \in I} g_i \mid \emptyset \ne I \subset
[1,l] \Bigr\} \ \subset \ G \quad \mbox { and, for } \mbox k \in \N
\,,
 \\ \nn \Sigma_k(S) &=& \Bigl\{ \sum_{i \in I} g_i \mid I \subset
[1,l] ,\,|I|=k\Bigr\} \ \subset \ G.\eer We set \ $-S = (-g_1) \cdot
\ldots \cdot (-g_l)$. If $G = \Z$, then we  define
\[
S^+ = \prod_{g \in G_0^+} g^{\mathsf v_g (S)} \quad \text{and} \quad
S^- = \prod_{g \in G_0^-} g^{\mathsf v_g (S)} \,,
\]
and thus we have $S = S^+ S^- 0^{\mathsf v_0 (S)}$. The monoid
\[
\mathcal B(G_0) = \{ S \in \mathcal F(G_0) \mid \sigma (S) =0\}
\]
is called the \  {\it monoid of zero-sum sequences} \ over \ $G_0$,
and its elements are called \ {\it zero-sum sequences} \ over \
$G_0$. A sequence $S\in \mathcal F(G_0)$ is zero-sum free if it has no proper, nontrivial zero-sum subsequence (note the trivial/empty sequence is defined to have sum zero). For every arithmetical invariant \ $*(H)$ \ defined for a
monoid $H$, we write $*(G_0)$ instead of $*(\mathcal B(G_0))$. In
particular, we set \ $\mathcal A (G_0) = \mathcal A (\mathcal B
(G_0))$. We define the \ {\it Davenport constant} \ of $G_0$ by
\[
\mathsf  D (G_0) = \sup \bigl\{ |U| \, \bigm| \; U \in \mathcal A
(G_0) \bigr\} \in \N_0 \cup \{\infty\} \,,
\]
which is a central invariant in zero-sum theory (see
\cite{Ga-Ge06b}, and also \cite{Ge09a} for its relevance in
factorization theory).

Clearly, $\mathcal B (G_0) \subset \mathcal F (G_0)$ is saturated,
and hence $\mathcal B (G_0)$ is a Krull monoid.
We note that $\mathcal B (G_0) \subset \mathcal F (G_0)$ is cofinal if and only if
for each $g \in G_0$ there is a $B \in \mathcal B (G_0)$ with $\mathsf v_g
(B) > 0$ (see \cite[Proposition 2.5.6]{Ge-HK06a}); if this is the case, then the set $G_0$ is called \emph{condensed}.
For a condensed set $G_0$, the class group of $\mathcal{B}(G_0) \hookrightarrow \mathcal{F}(G_0)$ is $\langle G_0 \rangle$, and the subset of classes containing prime divisors is $G_0$.

For $G_0 \subset \Z$, we have that $G_0$ is condensed if and only if either
 $G_0^+ \ne \emptyset$ and $G_0^- \ne \emptyset$ or $G_0 \subset \{0\}$.
The latter case, which in our context can be disregarded (see Lemma \ref{3.3}), is frequently automatically excluded by some of the conditions
we impose in our results; if not, we impose the extra condition $|G_0| \ge 2$ to this end.

\smallskip
\noindent {\bf Block monoids associated to Krull monoids.} We will
make substantial use of the following result (\cite[Section
3.4]{Ge-HK06a}).

\begin{lemma} \label{3.3}
Let $H$ be a Krull monoid, $\varphi \colon H \to F = \mathcal F (P)$
a cofinal divisor homomorphism, $G = \mathcal C (\varphi)$ its class
group, and $G_P \subset G$ the set of classes containing prime
divisors. Let $\widetilde{\boldsymbol \beta} \colon F \to \mathcal F
(G_P)$ denoted the unique homomorphism defined by
$\widetilde{\boldsymbol \beta} (p) = [p]$ for all $p \in P$.
\begin{enumerate}
\item The homomorphism $\boldsymbol \beta = \widetilde{\boldsymbol \beta} \circ \varphi \colon H \to \mathcal B
      (G_P)$ is a transfer homomorphism with $\mathsf c (H,
      \boldsymbol \beta) \le 2$. In particular, it has all the
      properties mentioned in Lemma \ref{3.2}.

\item $\mathcal B (G_P) \subset \mathcal F (G_P)$ is saturated and
      cofinal. If $G$ is infinite cyclic, then $G_P \subset G$ is a
      condensed set and $|G_P|\ge 2$.
\end{enumerate}
\end{lemma}

The homomorphism $\boldsymbol \beta$ is called the {\it block
homomorphism}, and $\mathcal B (G_P)$ is called the {\it block
monoid}  associated  to $\varphi$. If $\varphi$ is a divisor theory,
then $\mathcal B (G_P)$ is called the block monoid associated to
$H$.

\smallskip
\noindent {\bf One more  theorem and examples}. The following lemma
highlights the strong connection between the algebraic structure of
a Krull monoid and its class group and provides a realization
result (see \cite[Theorem 2.5.4]{Ge-HK06a}). Let $G$ be an abelian
group and $(m_g)_{g \in G}$ a family of cardinal numbers. We say $H$
has \ {\it characteristic \ $( G, (m_g)_{g \in G} )$} \ if there is
a group isomorphism \ $\Phi \colon G\tilde{\to}  \mathcal C(H)$ \
such that ${\rm card} ( P \cap \Phi(g)) = m_g$ for every $g \in G$.

\begin{lemma} \label{lem_char}
Let $G$ be an abelian group, $(m_g)_{g \in G}$ a
family of cardinal numbers and $G_0 = \{g \in G \mid m_g \ne 0\}$.
\begin{enumerate}
\item The following statements are equivalent{\rm \,:}
      \begin{enumerate}
      \item[(a)] There exists a Krull monoid $H$ and a group isomorphism \
                 $\Phi \colon G \to \mathcal C(H)$ \ such that \newline ${\rm card} ( P \cap \Phi(g)) = m_g$
                 for every $g \in G$.

      \item[(b)] $G= [G_0]$, and  $G = [G_0 \setminus \{g\}]$ for every $g
                \in G_0$ with $m_g =1$.
      \end{enumerate}

\smallskip
\item Two Krull monoids \ $H$ \ and \ $H'$ \ have the same characteristic if and only if \ $H_{\red} \cong
      H_{\red}'$.
\end{enumerate}
\end{lemma}

\medskip

\begin{examples}~

\smallskip
{\bf 1. Domains.} A domain $R$ is a Krull domain if and only if its
multiplicative monoid of nonzero elements is a Krull monoid. As a
special case of Claborn's Realization Theorem, there is the
following result: For every subset $G_0 \subset \Z$ with $[G_0] =
\Z$, there is a Dedekind domain $R$ and an isomorphism $\Phi \colon G
\to \mathcal C (R)$ such that $\Phi (G_0) = \{ g \in \mathcal C (R)
\mid g \cap \mathfrak X (R) \ne \emptyset \}$ (\cite[Theorem
3.7.8]{Ge-HK06a}. More results of this flavor are discussed in
\cite[Section 3.7]{Ge-HK06a} and \cite[Section 5]{Ge-HK92a}.

Let $R$ be a domain and $H$ a monoid such that $R[H]$ is a Krull
domain. There are a variety of results on the class group of $R[H]$,
which provide many explicit monoid domains having infinite cyclic
class group (\cite[\S 16]{Gi84}, see also  \cite{Ki01}). Generalized
power series domains that are Krull are studied in \cite{Ki-Pa01}.

\medskip
{\bf 2. Zero-sum sequences.} Let $G_0 \subset \Z$ be a subset such
that $[G_0 \setminus \{g\}] = \Z$ for all $g \in G_0$. Then the
monoid of zero-sum sequences $\mathcal B (G_0)$ is a Krull monoid
with class group isomorphic to $\Z$, and $G_0$ corresponds to the set
of classes containing prime divisors (\cite[Proposition
2.5.6]{Ge-HK06a}).

\medskip
{\bf 3. Module theory.} Let $R$ be a (not necessarily commutative)
ring and $\mathcal C$  a class of (right) $R$-modules---closed under
finite direct sums, direct summands and isomorphisms---such that
$\mathcal C$ has a set $V ( \mathcal C)$ of representatives (that
is, every module $M \in \mathcal C$ is isomorphic to a unique $[M]
\in V( \mathcal C))$.  Then $V ( \mathcal C)$ becomes a commutative
semigroup under the operation $[M] + [N] = [M \oplus N]$, which
carries detailed information about the direct-sum behavior of
modules in $\mathcal C$, e.g., whether or not the
Krull--Remak--Azumaya--Schmidt Theorem holds, and, when it does not,
how badly it fails. If every module $M \in \mathcal C$ has a
semilocal endomorphism ring, then $\mathcal V (C)$ is a Krull monoid
(\cite{Fa02}). For situations where this condition is satisfied and
when the class group of $\mathcal V (C)$ is cyclic, we refer to
recent work of Facchini, Hassler, Wiegand et al. (see, for example,
\cite{Wi01, F-H-K-W06, Fa06a, Fa-He06a}).

\medskip
{\bf 4. Diophantine monoids.} A Diophantine monoid is a monoid which
consists of the set of solutions in nonnegative integers to a system
of linear Diophantine equations. In more technical terms, if $m, n
\in \N$ and $A \in M_{m,n} ( \Z)$, then $H = \{ \boldsymbol x \in
\N_0^n \mid A \boldsymbol x = \boldsymbol 0 \}$ is a Diophantine
monoid. Moreover, $H$ is a Krull monoid, and if $m = 1$, then its
class group is cyclic and there is a characterization of when it is
infinite (\cite[Theorem 1.3]{Ch-Kr-Oe00}, \cite[Proposition
4.3]{Ch-Kr-Oe02}; see also \cite[Theorem 2.7.14]{Ge-HK06a} and
\cite[Chapter II.8]{Gr01}).
\end{examples}

\section{Arithmetical Properties Equivalent to the Finiteness of $G_P^+$ or $G_P^-$} \label{4}

Before we formulate our main characterization result, Theorem
\ref{main-theorem-I}, we recall a recent characterization of
tameness, which is in contrast with our present results. Let
$H$ be an atomic monoid. For an element $b \in H $, let \ $\omega
(H, b)$ \ denote the
      smallest \ $N \in \N_0 \cup \{\infty\} $ \ with the following
      property{\rm \,:}
      \begin{enumerate}
      \smallskip
      \item[] For all $n \in \N $ and $a_1, \ldots, a_n \in H $, if
              $b \t a_1 \cdot \ldots \cdot a_n $, then there exists a
              subset $\Omega \subset [1,n] $ such that $|\Omega | \le N $ and
              \[
              b \Bigm| \, \prod_{\nu \in \Omega} a_\nu \,.
              \]
      \end{enumerate}
Clearly, $b \in H$ is a prime if and only if $\omega (H, b) = 1$,
and so the $\omega (H, \cdot)$ values measure how far away atoms are
from primes. They are closely related to the local tame degrees
$\mathsf t (H, \cdot)$. A detailed study of their relationship can
be found in \cite[Section 3]{Ge-Ha08a}, but here we mention only two
simple facts (to simplify the formulation,  we suppose that $H$ is
reduced):
\begin{itemize}
\item $\omega (H, u) \le \mathsf t (H, u)$ for all
      $\boldsymbol{1} \ne u \in H$ which are not prime (this follows from the
      definition).

\item $\sup \{ \mathsf t (H, u) \mid u \in \mathcal A (H) \} <
      \infty$ if and only if $\sup \{ \omega (H, u) \mid u \in \mathcal A
      (H) \} < \infty$ (\cite[Proposition 3.5]{Ge-Ka10a}).
\end{itemize}
The monoid $H$ is said to be \ {\it tame} \ if the above suprema are
finite. Note that the finiteness in Proposition \ref{tame-charact}.1 holds
without any assumption on $G_P$ (indeed, it holds for all
$v$-noetherian monoids \cite[Theorem 4.2]{Ge-Ha08a}). In particular,
one should compare Propositions \ref{tame-charact}.1 and \ref{tame-charact}.2.(c) and Theorem
\ref{main-theorem-I}.(b).

\smallskip
\begin{proposition} \label{tame-charact}
Let $H$ be a Krull monoid and $\varphi \colon H\to \mathcal{F}(P)$ a
cofinal divisor homomorphism into a free monoid such that the class
group $G= \mathcal{C}(\varphi)$ is an  infinite cyclic group that we
identify with $\mathbb Z$. Let $G_P \subset G$ denote the set of
classes containing prime divisors.
\begin{enumerate}
\item $\omega (H, u) < \infty$ for all $u \in \mathcal A (H)$.

\smallskip
\item If $\varphi$ is a divisor theory, then the following statements are
      equivalent{\rm \,:}
      \begin{enumerate}
      \item[(a)] $G_P$ \ is finite.

      \item[(b)] $\mathsf D (G_P) < \infty$.

      \item[(c)] $H$ is tame.
      \end{enumerate}
\end{enumerate}
\end{proposition}

\smallskip
The equivalence of the three properties is a special case of
\cite[Theorem 4.2]{Ge-Ka10a}. It is essential that the imbedding is
a divisor theory and not only a cofinal divisor homomorphism.
Indeed, if $G_0 = \{-1\} \cup \N$, then $\mathcal B (G_0)
\hookrightarrow \mathcal F (G_0)$ is a cofinal divisor homomorphism,
$\mathsf D (G_0) = \infty$, but $\mathcal B (G_0)$ is factorial and
hence tame (see also Lemmas \ref{lem_char} and
\ref{transfer-to-finite}).

\medskip
\begin{theorem} \label{main-theorem-I}
Let $H$ be a Krull monoid and $\varphi \colon H\to \mathcal{F}(P)$ a
cofinal divisor homomorphism into a free monoid such that the class
group $G= \mathcal{C}(\varphi)$ is an  infinite cyclic group that we
identify with $\mathbb Z$. Let $G_P \subset G$ denote the set of
classes containing prime divisors. The following statements are
equivalent{\rm \,:}
\begin{enumerate}
\item[(a)] $G_P^+$ \ or \ $G_P^-$ \ is finite.

\smallskip
\item[(b)] $H$ is locally tame, i.e., $\mathsf t (H, u) < \infty$
            for all $u \in \mathcal A (H_{\red})$.

\smallskip
\item[(c)] The catenary degree $\mathsf c (H)$ is finite.

\smallskip
\item[(d)] The set of distances $\Delta (H)$ is finite.

\smallskip
\item[(e)] The elasticity $\rho(H)$ is a rational number.

\smallskip
\item[(f)] $\rho_2 (H)$ is finite.

\smallskip
\item[(g)] There exists some $M \in \mathbb{N}$ such that, for each $k \in \mathbb{N}$,
           we have $\rho_{k+1}(H) - \rho_k(H) \le M$.

\smallskip
\item[(h)] There exists some $M \in \mathbb{N}$ such that, for each $k \in \mathbb{N}$, the set $\mathcal{V}_k(H)$ is an
           {\rm AAP} with difference $\min \Delta(H)$ and bound $M$.
\end{enumerate}
\end{theorem}

\medskip
We point out the crucial implications in the above result. Suppose
that $(a)$ holds. Then $(b)$, $(c)$, $(e)$, $(g)$ and $(h)$ are
strong statements on the arithmetic of $H$. The conditions $(d)$ and
$(f)$ are very weak arithmetical statements (indeed, the
implications $(e) \Rightarrow (f)$, $(g) \Rightarrow (f)$ and $(h)
\Rightarrow (f)$ hold trivially in any atomic monoid). The crucial
point is that $(d)$ and $(f)$ both imply $(a)$. In
\cite{An-Ch-Sm94c}, it was first proved that (in the setting of
Krull domains) $(a)$ is equivalent to the finiteness of the
elasticity $\rho (H)$, and the problem was put forward whether or
not $\rho (H)$ would always be rational; part (e) shows that this is
indeed so. In \cite{B-C-R-S-S10}, it was recently shown that $(a)$
is equivalent to $(c)$ as well as to $(d)$ (also in the setting of
Krull monoids). We will give a complete proof of all
implications, not only because our setting is slightly more
general---being valid for any divisor \emph{homomorphism} rather
than divisor \emph{theory} (recall, as noted earlier, that
Proposition \ref{tame-charact}.2 does not hold in this slightly more
general setting, and so there is indeed sometimes a difference
between a divisor theory and homomorphism)---but also because we
need all the required tools regardless (in particular, for the
monotone catenary degree in Section \ref{5}), and thus little could
be saved by not doing so.

Note, if the equivalent conditions of Theorem \ref{main-theorem-I}
hold, then \cite[Theorem 4.2]{Ga-Ge09b} implies that
      \[
      \lim_{k \to \infty} \frac{|\mathcal V_k (H)|}{k} = \frac{1}{\min \Delta(H)}
      \Bigl( \rho (H) - \frac{1}{\rho (H)} \Bigr) \,.
      \]
Under a certain additional assumption, the sets $\mathcal V_k (H)$
are even arithmetical progressions and not only AAPs (\cite[Theorem
3.1]{Fr-Ge08}; for more on the sets $\mathcal V_k (H)$, see
(\cite[Theorem 3.1.3]{Ge09a}).

As mentioned in the introduction, there are characterizations of
arithmetical properties in various algebraic settings. In most of
them, the finiteness of the elasticity is equivalent to the
finiteness of all $\rho_k (H)$  (though this does not hold in all atomic monoids). But in none of these
settings is the finiteness of the elasticity equivalent to the
finiteness of the catenary degree. The reader may want to compare Proposition
\ref{tame-charact} and Theorem \ref{main-theorem-I} with \cite[Corollary
3.7.2]{Ge-HK06a}, \cite[Theorem 4.5]{Ka10a} or \cite[Theorem
4.4]{Ge-Ha08a}.

\smallskip
The remainder of this section is devoted to the proof of Theorem
\ref{main-theorem-I}. We start with the necessary preparations.

\begin{lemma} \label{Lambert}
Let $G_0 \subset \Z$ be a condensed subset. Then
\[
|U^+| \le |\inf G_0|  \qquad \text{for each atom} \qquad U \in
\mathcal A (G_0) \,.
\]
If in particular $G_0$ is finite, then $\mathsf{D}(G_0) \le \max G_0
+ |\min G_0|$.
\end{lemma}

\begin{proof}
This is due to Lambert (\cite{La87a}); for a proof in the present
terminology, see  \cite[Theorem 3.2]{B-C-R-S-S10}.
\end{proof}

\begin{lemma} \label{ex_of_atom}
Let $G_0 \subset \Z$ be a condensed subset such that $G_0^+$ is
infinite. For each $S \in \mathcal{F}(G_0^-)$, there exists some
$U\in \mathcal{A}(G_0)$ with $S \mid U$.
\end{lemma}

\begin{proof}
Let $d= \gcd (G_0^-)$. Then $[G_0^-]\subset -d \mathbb{N}$ and there exists some $g \in \mathbb{N}$ such that $-gd - d\mathbb{N}\subset [G_0^-]$. Since $G_0^+$ is infinite,
let $b \in G_0^+$ with $b > |\sigma(S)| + gd$, and let $\beta \in [1, d]$ be minimal such that $\beta b \in d \mathbb{N}$.
By the definition of $g$, there exists some $S'\in \mathcal{F}(G_0^-)$ such that $\sigma(S') = - (\beta b -|\sigma(S)|)=-(\beta b+\sigma(S))$.
Thus, $\sigma(b^{\beta}SS')=0$ and, by the minimality of $\beta$, it follows that $b^{\beta}SS'$ is an atom.
\end{proof}

The next lemma uses ideas from the proof of Theorem 3.1 in
\cite{B-C-R-S-S10}. It will be used for the investigation of the
catenary degree as well as for the monotone catenary degree
(Proposition \ref{6.4}).

\begin{lemma} \label{3.4}
Let $G_0 \subset \Z$ be a condensed subset such that $G_0^-$ is
finite and nonempty. Let $A \in \mathcal B (G_0)$ ne nontrivial  and $z,
\overline z \in \mathsf Z (A)$ with $|z| \le |\overline z|$. Then
there exists a $U \in \mathcal A (G_0)$ with $U \t \overline z$ and
a factorization $\widehat z \in \mathsf Z (A) \cap U \mathsf Z
(G_0)$ such that $\mathsf d (z, \widehat z) \le   \bigl(|\min G_0| +
|G_0^-|^2 \bigr) \ |\min G_0|$.
\end{lemma}

\begin{proof}
Let $\overline z = U_1 \cdot \ldots \cdot U_m$ and $z = V_1 \cdot
\ldots \cdot V_l$ where $l, m \in \N$ and $U_1, \ldots, U_m, V_1,
\ldots, V_l \in \mathcal A (G_0)$. We proceed in two steps. Note we may assume $0\nmid A$, else the lemma is trivial taking $U=0$ and $\hat z=z$.

\medskip
1. We assert that there is an $i \in [1, m]$ and a set $I
\subset [1, l]$ such that
\[
|I| \le |\min G_0 | + |G_0^-|^2 \ \quad \text{and} \quad
U_i \Bigm| \prod_{\nu \in I} V_{\nu} \,.
\]
We assume $l > |G_0^-|$, since otherwise the claim is obvious.
Since
\[
\sum_{i=1}^m \max \Bigl\{ \frac{\mathsf v_g (U_i)}{\mathsf v_g (A)}
\mid g \in G_0^- \Bigr\} \le \sum_{i=1}^m \sum_{g \in G_0^-}
\frac{\mathsf v_g (U_i)}{\mathsf v_g (A)} = \sum_{g \in G_0^-}
\Bigl( \frac{1}{\mathsf v_g (A)} \sum_{i=1}^m \mathsf v_g (U_i)
\Bigr) = |G_0^-| \,,
\]
there exists an $i \in [1, m]$ such that
\be\label{toto}
\frac{\mathsf v_g (U_i)}{\mathsf v_g (A)} \le \frac{|G_0^-|}{m} \,.
\ee
For each $g \in G_0^-$, there is an $I_g \subset [1, l]$ with $|I_g|
= |G_0^-|$ such that
\[
\mathsf v_g \Bigl( \prod_{\nu \in I_g} V_{\nu} \Bigr) \ge
\frac{|G_0^-| \mathsf v_g (A)}{l} \,.
\]
Hence, since $l \le m$, it follows by \eqref{toto} that
\[
\mathsf v_g \Bigl( \prod_{\nu \in I_g} V_{\nu} \Bigr) \ge
\frac{|G_0^-| \mathsf v_g (A)}{l} \ge \frac{m \mathsf v_g
(U_i)}{\mathsf v_g (A)} \frac{\mathsf v_g (A)}{l} = \mathsf v_g
(U_i) \,.
\]
Since by Lemma \ref{Lambert} we have $|U_i^+| \le | \min G_0 |$,
there is an $I_0 \subset [1, l]$ with $|I_0| \le |\min G_0|$ such
that
\[
\mathsf v_g (U_i) \le \mathsf v_g \Bigr( \prod_{\nu \in I_0} V_{\nu}
\Bigr) \quad \text{for all} \quad g \in G_0^+ \,.
\]
Then, for $I = I_0 \cup \bigcup_{g \in G_0^-} I_g$, we get  $\mathsf
v_g (U_i) \le \mathsf v_g \Bigl( \prod_{\nu \in I} V_{\nu} \Bigr)$
for each $g \in G_0$, i.e., $U_i \t \prod_{\nu \in I} V_{\nu}$.
Noting that $|I| \le |\min G_0| + |G_0^-|^2$, the argument is
complete.

\smallskip
2. By part 1, we may suppose without restriction
that $U_1 \t \prod_{\nu=1}^k V_{\nu}$ with $k \le \bigl( | \min G_0|
+ |G_0^-|^2 \bigr)$. We consider a factorization $V_1 \cdot \ldots
\cdot V_k = W_1 W_2 \cdot \ldots \cdot W_n$, where $U_1 = W_1, W_2,
\ldots , W_n \in \mathcal A (G_0)$, and by Lemma \ref{Lambert},
\[
\begin{aligned}
n & \le |(W_1 \cdot \ldots \cdot W_n)^+|
= |(V_1 \cdot \ldots \cdot V_k)^+| \\
 & \le k \ |\min G_0|   \le \bigl( |\min G_0| + |G_0^-|^2
 \bigr) \ |\min G_0|   \,.
\end{aligned}
\]
Now we set $\widehat z = W_1 \cdot \ldots \cdot W_n V_{k+1} \cdot
\ldots \cdot V_l$ and get
\[
\mathsf d (z, \widehat z) \le \max \{k, n\} \le \bigl( |\min G_0| +
|G_0^-|^2
 \bigr) \ |\min G_0| \,. \qedhere
\]
\end{proof}

\begin{lemma}
\label{lem_rhok}
Let $G_0\subset \mathbb{Z}$ be a condensed set such that $G_0^-$ is finite and nonempty.
\begin{enumerate}
\item There exists some $M \in \mathbb{N}$ such that $\rho_{k+1}(G)\le 1+kM$ for each $k \in \mathbb{N}_0$. More precisely,
\begin{enumerate}
\item if $G_0$ is infinite, then for each $k \in \mathbb{N}$,
\[1\le \rho_{k+1} (G_0) - \rho_{k} (G_0) \le 2 \ |\min G_0|.\]
\item if $G_0$ is finite, then for each $k \in \mathbb{N}$,
\[
1\le \rho_{k+1} (G_0) - \rho_{k} (G_0) \le  \mathsf D (G_0)-1 \,.
\]
\end{enumerate}
\item  For each $k \in \mathbb{N}$,
\[
-1 \le \lambda_{k}(G_0) - \lambda_{k+1}(G_0) < \bigl( |\min G_0| +
|G_0^-|^2 \bigr) \ |\min G_0| \,.
\]
\end{enumerate}
\end{lemma}

\begin{proof}
1. We recall that $\rho_1(G_0)=1$. It thus suffices to establish the bounds on $\rho_{k+1} (G_0) - \rho_{k} (G_0) $.
By Lemma \ref{Lambert}, we know $\rho_k(G_0)\leq k\cdot |\min G_0^-|<\infty$.

\noindent 1.(a) The left inequality is trivial and it remains to
verify the right inequality. Let $m = |\min G_0|$. Let $l \in
\mathbb{N}$, and let $A_1, \dots, A_{k+1}, U_{1}, \dots, U_{l} \in
\mathcal{A}(G_0)$ be such that
\[A_1 \cdot\ldots\cdot A_{k+1}= U_1\cdot \ldots \cdot U_{l} \, .\]
We have to show that $l \le \rho_k(G_0) +2m$.
 By Lemma \ref{Lambert}, we know that $|A^+| \le m$ for each $A \in \mathcal{A}(G_0)$.
Thus, we may assume that $(A_{k}A_{k+1})^+ \mid U_1 \cdot \ldots
\cdot U_{2m}$. Then $(\prod_{j=2m+1}^{l}U_j)^+ \mid
\prod_{i=1}^{k-1} A_i$. Let $S= (\prod_{j=2m+1}^{l}U_j)^-$. By Lemma
\ref{ex_of_atom}, there exists some $A_{k}'\in \mathcal{A}(G_P)$
with $S \mid A_{k}'$. We consider $B=(\prod_{i=1}^{k-1} A_i)A_{k}'$,
which is a product of $k$ atoms. We observe that
$\prod_{j=2m+1}^{l}U_j\mid B$. Thus, $\max \mathsf{L}(B)\ge l - 2m$,
establishing the claim.

\smallskip
\noindent 1.(b) This follows from \cite[Proposition 3.6]{Ge-Ka10a}
(see also Lemma 4.3 in that paper and note that $\mathsf D (G_0)
\ge 2$).

\smallskip
\noindent
2.
The left inequality is trivial and it remains to verify the right
inequality. Let $s = \lambda_{k+1} (G_0)$ and let $U_1, \ldots,
U_s, A_1, \ldots, A_{k+1} \in \mathcal A (G_0)$ be such that
\[
U_1 \cdot \ldots \cdot U_s = A_1 \cdot \ldots \cdot A_{k+1} \,.
\]
After renumbering if necessary, Lemma \ref{3.4} implies that $A_1 \t
U_1 \cdot \ldots \cdot U_j$ and $U_1 \cdot \ldots \cdot U_j = A_1
W_2 \cdot \ldots \cdot W_i$ with $W_1, \ldots , W_i \in \mathcal A
(G_0)$ and  $i
\le  \bigl( |\min G_0| + |G_0^-|^2 \bigr)   |\min G_0| =M_2$ (note
that, in order to apply Lemma \ref{3.4}, we used that $s \le k+1$).
Then
\[
W_2 \cdot \ldots \cdot W_i U_{j+1} \cdot \ldots \cdot U_s = A_2
\cdot \ldots \cdot A_{k+1},
\]
and hence
\[
\begin{aligned}
\lambda_{k} (G_0) & \le \min \mathsf L (A_2 \cdot \ldots \cdot A_{k+1})
\le \min \mathsf L (U_{j+1} \cdot \ldots \cdot U_s) + \min \mathsf L
(W_2 \cdot \ldots \cdot W_i) \\
 & \le s - j + i-1 \le \lambda_{k+1} (G_0) + (M_2 - 1) \,. \qedhere
\end{aligned}
\]
\end{proof}

We continue with a lemma that is used when investigating the  sets of distances and local tameness.
To simplify the formulation, we introduce the following notation. For $a \in -\mathbb{N}$ and $b \in \mathbb{N}$, let
$V_{a,b}$ denote the unique atom with support $\{a,b\}$, that is $V_{a,b}= a^{\alpha}b^{\beta}$ with $\alpha= \lcm(a,b)/|a|$ and $\beta= \lcm(a,b)/b$.

\begin{lemma} \label{lem_gap}
Let $G_0 \subset \mathbb{Z}$ and let $v \in \mathbb{N}$. Suppose there exist distinct $a,a_2 \in G_0^-$ and $b,b_1
\in G_0^+$ that satisfy $b_1 \ge b|a|$ and   $|a_2| \ge  (v b_1 +b)|a|$.
For a given $z \in \mathsf{Z}((V_{a, b_1} V_{a_2,b})^v)$, let $z_0$ be the (unique) minimal divisor of $z$ such that $\mathsf{v}_{a_2}(\pi(z_0^{-1}z))=0$, and let $t(z)= \mathsf{v}_{b_1}(\pi(z_0))$.
Then,
\[|z| \in  \left[\frac{b_1}{\lcm(a,b)} \  t(z) - D ,\frac{b_1}{\lcm(a,b)} \  t(z)+ D \right] \quad \text{where} \quad D =  v (b+|a|) \gcd(a,b) \, .\]
Moreover, if $t(z)=0$, then $z = V_{a,b_1}^v\cdot V_{a_2,b}^v$.
\end{lemma}
Since it is relevant in applications of this lemma, we point out that $D$ depends neither on $a_2$ nor on $b_1$.
\begin{proof}
To simplify notation without suppressing the information on the origin of certain quantities, we set $\alpha = \mathsf{v}_a(V_{a,b})$,
$\alpha_1= \mathsf{v}_a(V_{a,b_1})$, and $\alpha_2 = \mathsf{v}_{a_2}(V_{a_2,b}) $. Likewise, we set $\beta = \mathsf{v}_b(V_{a,b}) $,
$\beta_1= \mathsf{v}_{b_1}(V_{a,b_1})$, and $\beta_2 = \mathsf{v}_{b}(V_{a_2,b})$.

From the explicit description or applying Lemma \ref{Lambert}, we get $\beta , \beta_{1}\in [1,|a|]$ and  $\alpha, \alpha_2 \in [1, b]$.

Let $z = U_1 \cdot \ldots \cdot U_m$, where $U_1, \ldots, U_m \in \mathcal A (G_0)$, and $k, l \in [1,m]$ with $k \le l$ be such that
\begin{itemize}
\item $a_2 \t U_{\nu}$ \ for each $\nu \in [1, k]$,

\item $a_2 \nmid U_{\nu}$ and $b_1 \t U_{\nu}$ \ for each $\nu \in [k+1,
l]$, and

\item $a_2 \nmid U_{\nu}$ and $b_1 \nmid U_{\nu}$ \ for each $\nu \in
[l+1, m]$;
\end{itemize}
in particular, $z_0 =  U_{1}\cdot \ldots \cdot U_k \in \mathsf{Z}(G_0)$.
Also note that $U_{\nu}=V_{a,b}$ for each $\nu \in [l +1, m]$.

For $\nu \in [1, k]$, we have
\[
U_{\nu} = a_2^{ \alpha_{\nu, 2}} a^{ \alpha_{\nu, 1} }
b_1^{ \beta_{\nu,1} } b^{ \beta_{\nu, 2} } \,,
\]
where $ \alpha_{\nu,2}  \in \N$ and $ \alpha_{\nu,1} ,
 \beta_{\nu, 1} ,  \beta_{\nu, 2}  \in \N_0$.
By the assumption on $|a_2|$ and since $\beta, \beta_1 \in [1, |a|]$, we have  $|a_2| \ge v \beta_1 b_1 + \beta b$.
Thus, in view of $\vp_{b_1}(\pi(z))=\beta_1 v$, it follows that $\beta_{\nu, 2}  \ge \beta$. Hence $ \alpha_{\nu, 1}  \le \alpha-1$, since otherwise $V_{a,b} \mid U_{\nu}$, which is impossible (as $a_2|U_\nu$).

Let $\alpha_2' = \mathsf{v}_{a}( \pi(z_0))$ and $\beta_2' = \mathsf{v}_{b}( \pi(z_0))$.
In view of $ \alpha_{\nu, 1}  \le \alpha - 1$,  $k \le v \alpha_2$ and $\alpha, \alpha_2 \in [1, b]$, we have
$0 \le \alpha_2' \le   v b^2$.

We note that $\sigma( \pi(z_0)^-) = v \alpha_2 a_2 + \alpha_2'a$, and thus
\[t(z) b_1 +  \beta_2' b =  v \alpha_2 |a_2| + \alpha_2'|a|,\] i.e.,
$\beta_2' =  b^{-1}(v \alpha_2 |a_2| + \alpha_2'|a| - t(z)b_1)$.
In particular, note that if $t(z)=0$, then, since $$\sigma(b^{\vp_b(\pi(z))})=v\cdot \sigma(b^{\vp_b(V_{a_2,b})})=-v\cdot \sigma(a_2^{\vp_{a_2}(V_{a_2,b})})$$
implies $\mathsf{v}_b((V_{a,b_1}V_{a_2,b})^v) = b^{-1}(v \alpha_2 |a_2|)$, it follows that
$\alpha_2' = 0$ and $z_0= V_{a_2,b}^v$; this establishes the ``moreover''-statement.

Consequently,
\begin{equation}
\label{lem_gap_eq_1}
b^{-1}( v \alpha_2 |a_2|  - t(z) b_1 )\le  \beta_2' \le  b^{-1}(v \alpha_2 |a_2| + v b^2 |a| - t(z)b_1).
\end{equation}

For $\nu \in [k+1, l]$, we have
\[
U_{\nu} = b_1^{\beta_{\nu,1}''} b^{\beta_{\nu,2}''} a^{\alpha_{\nu,1}''} \,,
\]
with $\beta_{\nu,1}'' \in \N$ and $\alpha_{\nu,1}'',\beta_{\nu,2}'' \in \N_0$.
We have  $\alpha_{\nu,1}''|a| \ge b_1$.  Thus, by the assumption on $b_1$ and since $\alpha\in [1,b]$,
we get $\alpha_{\nu,1}'' \ge \alpha$, and hence $\beta_{\nu, 2}'' \le \beta-1$ (as otherwise $U_\nu=V_{a,b}$ with $b_1|U_\nu$ but $b_1\nmid V_{a,b}$, a contradiction).

Let $\beta_2'' = \mathsf{v}_{b}( \prod_{\nu= k+1}^l U_{\nu})$.
We  note that $l - k \le \mathsf v_{b_1} ((V_{a,b_1}V_{a_2,b})^v) - t(z)= v \beta_1 - t(z)\le v |a| - t(z) \le v|a|$.
Thus, we obtain that
\begin{equation}
\label{lem_gap_eq_2}
0 \le \beta_2'' \le  (l-k) (\beta-1) \le v|a| (\beta-1) \le  v |a|^2\,.
\end{equation}

Let $\beta_2''' = \mathsf{v}_b(\prod_{\nu = l +1}^mU_{\nu})$.
We have
\[\beta_2'''= \mathsf{v}_{b}((V_{a,b_1}V_{a_2,b})^v)- \beta_2' -\beta_2'' = v \beta_2 - \beta_2' - \beta_2''.\]
In combination with \eqref{lem_gap_eq_1} and \eqref{lem_gap_eq_2}, we get that
\[v \beta_2 -  b^{-1}\bigl(v \alpha_2 |a_2| + v b^2 |a| - t(z)b_1\bigr) - v |a|^2
   \le \beta_2'''
   \le  v \beta_2 - b^{-1}\bigl( v \alpha_2 |a_2|  - t(z) b_1 \bigr).\]
Thus, since  $\beta_2 = b^{-1} \alpha_2 |a_2|$ (in view of $V_{a_2,b}=a_2^{\alpha_2}b^{\beta_2}$), it follows that
\be\label{pillow}\beta_2''' \in   \frac{b_1}{b} t(z) + [ - vb|a| - v|a|^2 ,0 ].\ee

Since $U_{\nu}=V_{a,b}$ for each $\nu \in [l +1, m]$, it follows that $\beta_2''' = (m-l)\beta$.
Since $k \in [0, vb]$ and $l - k \in [0,v|a|]$, we get that
$m \in (m-l)+[0, v (b+|a|)]$.
Combining with $\beta_2''' = (m-l)\beta$ and \eqref{pillow} then yields
\[ m \in  \left[ \frac{b_1}{b\beta} t(z) - \frac{ vb|a| + v|a|^2}{\beta} , \frac{b_1}{b\beta} t(z) + v (b+|a|) \right]\, ,\]
and, since  $\beta \le |a|$, we have $v (b+|a|) \le  v (b+|a|)|a|/\beta$.
Substituting the explicit value of $\beta$, the claim follows.
\end{proof}

The following proposition is a major portion of Theorem \ref{main-theorem-I}.

\begin{proposition} \label{rel_prop}
Let $G_0 \subset \mathbb{Z}$ be a condensed set such that $G_0^-$ is finite and nonempty.
Then $\rho(G_0)$ is a rational number.
\end{proposition}

To prove this result, we need the concept of factorizations with
respect to a (not necessarily minimal) generating set. This idea is
also used in the recent paper \cite{C-K-D-H10}, where a generalized
set of distances is studied for numerical monoids.

Let $H$ be a monoid and $S \subset H_{\red}\setminus
\{\boldsymbol{1}\}$ a subset. We call $\mathsf{Z}^S (H) = \mathcal{F}
(S)$ the factorization monoid of $H$ with respect to $S$. The
homomorphism $\pi_{H}^S=\pi^S \colon \mathsf{Z}^S (H) \to H_{\red}$
defined by $\pi^{S}(z)= \prod_{u \in S} u^{\mathsf{v}_{u}(z)}$ is
called the factorization homomorphism of $H$ with respect to $S$.
For $a\in H$, we set $\mathsf{Z}^{S}_H(a)=\mathsf{Z}^{S}(a)=
(\pi^{S})^{-1}(aH^{\times})$; we call this the set of
factorizations in $S$ of $a$. The set $\mathsf{L}^{S}(a) = \{|z|
\mid z \in  \mathsf{Z}^{S}(a) \}$ is called the set of lengths of
$a$ with respect to $S$.

We note that $\mathsf{Z}^{S}(a)\neq \emptyset$ for each $a \in H$ if
and only if $S$ generates $H_{\red}$ (as a monoid). If $S$ generates
$H_{\red}$, then $\mathcal A (H_{\red}) \subset S$ by
\cite[Proposition 1.1.7]{Ge-HK06a}. If $S= \mathcal{A}(H_{\red})$,
then $\mathsf Z^S (a) = \mathsf Z (a)$, and all other notions
coincide with the usual ones. Suppose that  $S \subset H_{\red}$ is
a generating set. For $a \in H$, let $\rho^S(a)= \rho
(\mathsf{L}^S(a))$ denote the elasticity of $a$ with respect to $S$,
and $\rho^S(H)= \sup \{\rho^S (a) \mid a \in H\}$ the elasticity of
$H$ with respect to $S$; note that $0 \in \mathsf{L}^{S}(a)$ if and
only if $\mathsf{L}^{S}(a)=\{0\}$, i.e., $a \in H^{\times}$. We say
that the elasticity of $H$ with respect to $S$ is accepted if there
exists some $a \in H $ with $\rho^{S}(a)= \rho^{S}(H)$.

The proof of the following result is a direct modification of the
one for the (usual) elasticity of finitely generated monoids
(\cite[Theorem 3.1.4]{Ge-HK06a}) and contains it as the special case
$S= \mathcal{A}(H_{\red})$.

\begin{lemma} \label{rel_lem_taurat}
Let $H$ be a monoid and  $S \subset H_{\red}\setminus
\{\boldsymbol{1}\}$  a finite generating set of $H_{\red}$. Then
$\rho^S(H)$ is finite and accepted, in particular,
 rational.
\end{lemma}

\begin{proof}
By construction,   $\mathsf{Z}^S(H) \time \mathsf{Z}^S(H)$ is a
finitely generated free monoid. Obviously,  $Z=\{(x,y) \in
\mathsf{Z}^S(H) \time \mathsf{Z}^S(H) \mid \pi^S(x)=\pi^S(y)\}$ is a
saturated submonoid, thus finitely generated by
\cite[Proposition 2.7.5]{Ge-HK06a}. Let $Z^{\bullet}= Z \setminus
Z^{\times}$; clearly $|Z^{\times}|=1$ and, for each $(x,y)\in
Z^{\bullet}$, we have that both $|x|\neq 0$ and $|y|\neq 0$. We note
that $\rho^S(H) = \sup \{|x|/|y| \mid (x,y)\in Z^{\bullet}\}$. We
assert that $\sup \{|x|/|y| \mid (x,y)\in Z^{\bullet}\}= \sup
\{|x|/|y| \mid (x,y)\in \mathcal{A}(Z)\}$. Since $\mathcal{A}(Z)$ is
finite, this implies the result.

Let $s=(x_s,y_s)\in Z^{\bullet}$ and let $s= t_1\cdot\ldots\cdot t_l$ with $t_i=(x_i,y_i) \in \mathcal{A}(Z)$ be a factorization of $s$ in the monoid $Z$.
We have, using the standard inequality for the mediant,
\[\frac{|x_s|}{|y_s|}= \frac{\sum_{i=1}^{l}|x_i|}{\sum_{i=1}^{l}|y_i|}\le \max
\left\{\frac{|x_i|}{|y_i|} \mid i \in [1,l] \right\},\] showing that
$\sup \{|x|/|y| \mid (x,y)\in Z^{\bullet}\}\le \sup \{|x|/|y| \mid
(x,y)\in \mathcal{A}(Z)\}$. The other inequality being trivial, the
claim follows.
\end{proof}

For a condensed set $G_0 \subset \mathbb{Z}$ with $|G_0| \ge 2$, we define
\[
\mathcal{B}(G_0)^+ =\{B^+\mid B \in \mathcal{B}(G_0)\} \ \text{ and}
\ \mathcal{A}(G_0)^+ =\{A^+\mid A \in \mathcal{A}(G_0)\} \,.
\]

\begin{lemma} \label{rel_lembas}
Let $G_0 \subset \mathbb{Z}$ be a condensed set with $|G_0|\ge 2$.
\begin{enumerate}
\item $\mathcal{B}(G_0)^+ \subset \mathcal{F}(G_0^+)$ is a submonoid.

\smallskip
\item $\mathcal{A}(G_0)^+$ is a generating set of $\mathcal{B}(G_0)^+$.

\smallskip
\item $|F| \le |\inf G_0^-|$ for each $F \in \mathcal{A}(G_0)^+$.
\end{enumerate}
\end{lemma}

\begin{proof}
The first two claims are immediate, and the last one is a direct
consequence of Lemma \ref{Lambert}.
\end{proof}

Clearly, $\mathcal{A}(G_0)^+$ contains $\mathcal{A}(\mathcal{B}(G_0)^+)$, the set of atoms of $\mathcal{B}(G_0)^+$, yet
it is in general not equal to this set; by definition, we have that $F \in \mathcal{A}(G_0)^+$ if and only if there  exists  some $A \in \mathcal{A}(G_0)$ such that $F=A^+$, yet $F \in \mathcal{A}(\mathcal{B}(G_0)^+)$ if and only if, for each $B \in \mathcal{B}(G_0)$ with $F=B^+$, we have $B\in \mathcal{A}(G_0)$.
Moreover, $\mathcal{B}(G_0)^+ $ is in general not a saturated submonoid of $\mathcal{F}(G_0^+)$.

The following technical result is used to partition $\mathcal{A}(G_0)$ into finitely many classes (cf.~below).

\begin{lemma}
\label{rel_lemt} Let $G_0 \subset \mathbb{Z}$ be a condensed set
such that $G_0^-$ is finite and nonempty. Let $F \in \mathcal{F}(G_0^+)$, $g \in
\supp (F)$ with $g \ge |G_0^-| \, |\min G_0^-| \, \lcm(G_0^-)$, and
$k \in \mathbb{N}$ with $g'=g+k\lcm (G_0^- ) \in G_0^+$.
 Then
\ $F \in \mathcal{A}(G_0)^+$ \ if and only if \ $g'g^{-1}F \in
\mathcal{A}(G_0)^+$.
\end{lemma}

\begin{proof}
We set $T= g'g^{-1}F \in \mathcal{F}(G_0^+)$. Suppose $F \in
\mathcal{A}(G_0)^+$. Let $R \in \mathcal{F}(G_0^-)$ such that $FR
\in \mathcal{A}(G_0)$. Since $\sigma(F) \ge g \ge |G_0^-| \, |\min
G_0^-| \, \lcm(G_0^-)$, there exists some $a \in G_0^-$ such that
$\mathsf{v}_a(R)\ge \lcm (G_0^- )$. Let $R_1= Ra^{k\lcm (G_0^-
)/|a|}$. Then $TR_1 \in \mathcal{B}(G_0)$. Assume to the contrary
that $TR_1$ is not an atom, say $TR_1 = (T'R_1')(T''R_1'')$, where
$g'\mid T'$, $T=T'T''$ and $R_1=R_1'R_1''$. Let $l'\in \mathbb{N}_0$
be maximal such that $a^{l'\lcm (G_0^- )/|a|}\mid R_1'$ and let
$l=\min \{l',k\}$. We note that $a^{-l\lcm (G_0^- )/|a|}R_1'\mid R$.
Moreover, since

\ber\nn |\sigma(a^{-l\lcm (G_0^- )/|a|}R_1')|&\ge& g' - l\lcm (G_0^-
) \ge (k-l)\lcm (G_0^- ) + |G_0^-| \, |\min G_0^-| \,
\lcm(G_0^-)\\\nn &\geq& (k-l)\cdot \lcm(G_0^-)+\sum_{x\in
G_0^-}|x|\left(\frac{\lcm(G_0^-)}{|x|}-1\right),\eer there exists a
subsequence $R_2'\mid a^{-l\lcm (G_0^- )/|a|}R_1'$ such that
$\sigma(R_2')= -(k-l)\lcm (G_0^- )$. We set $R_0 = R_2'^{-1}
a^{-l\lcm (G_0^- )/|a|}R_1'$. Then $\sigma(R_0) = \sigma(R_1') +
k\lcm (G_0^- )$. Thus $\sigma(gg'^{-1}T'R_0)=0$, yet
$gg'^{-1}T'R_0\mid FR$, contradicting that $TR_1$ is not an atom.

Suppose $T \in \mathcal{A}(G_0)^+$. Let $R'\in \mathcal{F}(G_0^-)$
be such that $TR'\in \mathcal{A}(G_0)$. Since
$$-\sigma(R_1)=\sigma(T)\ge g'\ge k\cdot \lcm
(G_0^- )+ |G_0^-| \, |\min G_0^-| \, \lcm(G_0^-)\geq k\cdot
\lcm(G_0^-)+\sum_{x\in
G_0^-}|x|\left(\frac{\lcm(G_0^-)}{|x|}-1\right),$$ there exists a
subsequence $R_1'\mid R'$ with $\sigma(R_1') = -k\cdot \lcm (G_0^-
)$. Let $R= R_1'^{-1}R'$. Then $FR$ is a zero-sum sequence. Assume
$FR$ is not an atom, say $FR= (F'R_2')(F''R_2'')$, where $g \mid
F'$, $F=F'F''$ and $R=R_2'R_2''$. Then $g'g^{-1}F'R_2'R_1'\mid TR'$
and it is a zero-sum sequence, contradicting that $FR$ is not an
atom.
\end{proof}

Let $G_0 \subset \mathbb{Z} \setminus \{0\}$ be a condensed set
such that $G_0^-$ is finite and nonempty. In view of Lemma \ref{rel_lemt}, we
introduce the following relation on $G_0^+$. For $g,h \in G_0^+$, we
say that $g$ is equivalent to $h$ if $g= h$ or if $g,h \ge |G_0^-| \,
|\min G_0^-| \, \lcm(G_0^-)$ and $g \equiv h \mod \lcm (G_0^- )$.
This relation is an equivalence relation and it partitions $G_0^+$
into finitely many---namely, less than $|G_0^-| \, |\min G_0^-| \,
\lcm(G_0^-) + \lcm (G_0^- )$---equivalence classes; we denote the
equivalence class of $g$ by $\kappa(g)$ and also use $\kappa$ to
denote the extension of this map to $\mathcal{F}(G_0^+)$.

We note that $\kappa(\mathcal{A}(G_0)^+)$ is a finite set, since it
consists of sequences over the finite set $\kappa(G_0^+)$ and the
length of each sequence is at most $|\min G_0^-|$ by Lemma
\ref{rel_lembas}. Moreover, it is a generating set of the monoid
$\kappa(\mathcal{B}(G_0)^+)$.

In order to study factorizations, we  extend $\kappa$ to
$\mathsf{Z}(G_0)$ via
\[
\kappa(A_1 \cdot \ldots \cdot A_l)= \kappa(A_1^+)\cdot \ldots \cdot \kappa(A_l^+).
\]
This is an element of $\mathcal{F}(\kappa(\mathcal{A}(G_0)^+))$,
i.e.,
$\mathsf{Z}^{\kappa(\mathcal{A}(G_0)^+)}(\kappa(\mathcal{B}(G_0)^+))$;
for brevity, we denote this factorization monoid by $\mathsf{Z}^{\kappa}$.
Likewise, for $F \in \kappa (\mathcal{B}(G_0)^+)$, we denote $\mathsf{Z}^{\kappa(\mathcal{A}(G_0)^+)}(F)$ by
$\mathsf{Z}^{\kappa}(F)$;
$\pi^{\kappa(\mathcal{A}(G_0)^+)}$ by $\pi^{\kappa}$; and
$\rho^{\kappa(\mathcal{A}(G_0)^+)}$ by $\rho^{\kappa}$. The
homomorphism $\kappa \colon \mathsf{Z}(G_0)\to \mathsf{Z}^{\kappa}$
is epimorphic.

We note that, for $B \in \mathcal{B}(G_0)$, we have that
$\kappa(\mathsf{Z}(B)) \subset (\pi^{\kappa})^{-1}(\kappa(B^+))$, and in general, this is a proper inclusion.
However, we have, for each $F \in \mathcal{B}(G_0)^+$, by Lemma \ref{rel_lemt},
\begin{equation}
\label{rel_eq_1} (\pi^{\kappa})^{-1}(\kappa(F))\;=  \bigcup_{B \in
\mathcal{B}(G_0),\,  B^+=F} \kappa(\mathsf{Z}(B)) ,
\end{equation}
whenever $G_0 \subset \mathbb{Z} \setminus \{0\}$ is condensed with $G_0^-$ finite and nonempty.

\begin{lemma}
\label{rel_lem_eq} Let $G_0 \subset \mathbb{Z}\setminus \{0\}$ be a
condensed set such that $G_0^-$ is finite and nonempty.
\begin{enumerate}
\item For each $B \in \mathcal{B}(G_0)$, we have $\rho(B) \le
      \rho^{\kappa}(\kappa(B^+))$. In particular, $\rho(G_0)\le
      \rho^{\kappa}(\kappa(\mathcal{B}(G_0)^+))$.

\item If $G_0$ is infinite, then $\rho(G_0) = \rho^{\kappa}(\kappa(\mathcal{B}(G_0)^+))$.
\end{enumerate}
\end{lemma}

\begin{proof}
1. Let $B \in \mathcal{B}(G_0) \setminus \{1\}$,  $x,y\in
\mathsf{Z}(B)$ with $|x|=\max \mathsf{L}(B)$ and $|y|= \min
\mathsf{L}(B)$. Since $\kappa(x),\kappa(y)\in
\mathsf{Z}^{\kappa}(\kappa(B^+))$, we have that $\rho(B)=|x|/|y|=
|\kappa(x)|/|\kappa(y)|\le \rho^{\kappa}(\kappa(B^+))$. The
additional claim is clear.

2. By part 1, it remains to show that $\rho(G_0)\ge
\rho^{\kappa}(\kappa(\mathcal{B}(G_0)^+))$.

By Proposition \ref{rel_lem_taurat} and since
$\kappa(\mathcal{A}(G_0)^+)$ is finite, we know that
$\rho^{\kappa}(\kappa(\mathcal{B}(G_0)^+))$ is accepted. Let
$B_{\kappa}\in \kappa(\mathcal{B}(G_0)^+)$ be such that
$\rho^{\kappa}(B_{\kappa})=
\rho^{\kappa}(\kappa(\mathcal{B}(G_0)^+))$, and let $x_{\kappa},
y_{\kappa}\in \mathsf{Z}^{\kappa}(B_{\kappa})$ be such that
$|x_{\kappa}|/|y_{\kappa}| = \rho^{\kappa}(B_{\kappa})$. By
\eqref{rel_eq_1}, we know that there exist $B_x,B_y \in
\mathcal{B}(G_0)$ with $B_x^+ = B_y^+ $,  $x \in
\mathsf{Z}(B_x)$ with $\kappa(x)=x_{\kappa}$, and $y \in
\mathsf{Z}(B_y)$ with $\kappa(y)=y_{\kappa}$; in particular, we have
$\kappa(B_x^+)= \kappa(B_y^+)=B_{\kappa}$.

Let $n \in \mathbb{N}$. Since $G_0^+$ is infinite, Lemma
\ref{ex_of_atom} yields some $U_n \in \mathcal{A}(G_0)$ with
$(B_x^n)^- \mid U_n$ and $U_n^-=(B_x^n)^-$. We set $D_n= B_y^nU_n$ and note that, since
$(B_x^n)^+ = (B_y^n)^+$ and $(B_x^n)^-|U_n^-$, the sequence $B_x^n$
is a proper subsequence of $D_n$. Thus,
\[\min \mathsf{L}(D_n)\le |y^n|+1 = n|y_{\kappa}|+1 \quad\quad \text{and} \quad\quad \max \mathsf{L}(D_n)\ge |x^n|+1 = n|x_{\kappa}|+1.\]
So we get
\[\rho(D_n)\ge \frac{n|x_{\kappa}|+1}{n|y_{\kappa}|+1}.
\]
Thus, for each $n \in \mathbb{N}$,
\[\rho(G_0)\ge \frac{n|x_{\kappa}|+1}{n|y_{\kappa}|+1},\]
and letting $n\rightarrow \infty$, we have
\[
\rho(G_0) \ge \frac{|x_{\kappa}|}{|y_{\kappa}|}=
\rho^{\kappa}(\kappa(\mathcal{B}(G_0)^+)) \,. \qedhere
\]
\end{proof}

\begin{proof}[{\bf Proof of Proposition \ref{rel_prop}}]
Since $\rho(G_0)= \rho(G_0 \setminus \{0\})$, we may assume that $0
\notin G_0$.

If $G_0$ is finite, then $\mathcal{B}(G_0)$ is finitely generated \cite[Theorem 3.4.2.1]{Ge-HK06a},
and thus the elasticity  is rational by Lemma \ref{rel_lem_taurat}
(applied with $S= \mathcal{A}(H_{\red})$). Suppose $G_0$ is
infinite. By Lemma \ref{rel_lem_eq}, we have that $\rho(G_0)=
\rho^{\kappa}(\kappa(\mathcal{B}(G_0)^+))$, and by Lemma
\ref{rel_lem_taurat}, we know that
$\rho^{\kappa}(\kappa(\mathcal{B}(G_0)^+))$ is rational.
\end{proof}

\bigskip
\begin{proof}[{\bf Proof of Theorem \ref{main-theorem-I}}]~

\smallskip
(a) \,$\Rightarrow$\, (b) \ Without restriction, we may suppose that
$G_P^-$ is finite. Let $u \in \mathcal A (H_{\red})$. We have to
show that $\mathsf t (H, u) < \infty$. If $u$ is prime, then
$\mathsf t (H, u) = 0$. Suppose that $u$ is not prime. Let $a \in
H$ and $a'=aH^{\times}$ be such that $u \mid a'$.
Let $z = v_1 \cdot \ldots \cdot v_n \in \mathsf Z
(a)$.  There is a minimal subset $\Omega \subset [1, n]$, say
$\Omega = [1,k]$, such that $u \t v_1 \cdot \ldots \cdot v_k$ and $k
\le |\varphi_{\red}(u)|$. We consider any factorization of $v_1 \cdot \ldots \cdot
v_k$ containing $u$, say $v_1 \cdot \ldots \cdot v_k = u_1 \cdot
\ldots \cdot u_l$, where $u = u_1, \ldots , u_l \in \mathcal A (H_{\red})$.

For $i \in [1,k]$ and $j \in [1, l]$, we set $V_i = \boldsymbol \beta
(v_i)$ and $U_j = \boldsymbol \beta (u_j)$. Then $U_1, \ldots , U_l,
V_1 , \ldots , V_k \in \mathcal A (G_P)$. Since $u$ is not a prime
and $\Omega$ is minimal, it follows that $0 \nmid V_1 \cdot \ldots
\cdot V_k$. Hence, for every $j \in [1,l]$, $U_j$ contains an element
from $G_P^+$, and Lemma \ref{Lambert} implies that
\[
l \le |(U_1 \cdot \ldots \cdot U_l)^+| = |(V_1 \cdot \ldots \cdot
V_k)^+| \le k \ |\min G_P^-|  \le |\varphi_{\red}(u)|\ |\min G_P^-| \,.
\]
Setting $z' = u_1 \cdot \ldots \cdot u_l v_{k+1} \cdot \ldots \cdot
v_n$, we infer that $\mathsf d (z, z') \le \max \{k, l \} \le |\varphi_{\red}(u)| \
|\min G_P^-|$, and hence $\mathsf t (H, u) \le |\varphi_{\red}(u)| \ |\min G_P^-|$.

\smallskip
(a) \,$\Rightarrow$\, (c) \ Without restriction, we may suppose that
$G_P^-$ is finite. By Lemma \ref{3.3}, it suffices to show that
$\mathsf c (G_P) < \infty$. We set $M =\bigl( |\min G_P| + |G_P^-|^2
\bigr) \ |\min G_P|$, and assert that $\mathsf c (A) \le M$ for all
$A \in \mathcal B (G_P)$. To do so, we proceed by induction on $\max
\mathsf L (A)$. If $A \in \mathcal B (G_P)$ with $\max \mathsf L (A)
\le M$, then $\mathsf c (A) \le \max \mathsf L (A) \le M$. Let $A
\in \mathcal B (G_P)$, let $z, \overline z \in \mathsf Z (A)$ with $|z|
\le |\overline z|$, and suppose that $\mathsf c (B) \le M$ for all $B
\in \mathcal B (G_P)$ with $\max \mathsf L (B) < \max \mathsf L
(A)$. By Lemma \ref{3.4}, there is a $U \in \mathcal A (G_P)$ and a
factorization $\widehat z \in \mathsf Z (A) \cap U \mathsf Z (G_P)$
such that $U \t \overline z$ and $\mathsf d (z, \widehat z) \le M$,
say $\widehat z = U \widehat y$ and $\overline z = U \overline y$
with $\widehat y, \overline y \in \mathsf Z (B)$ and $B = U^{-1}A$.
Since $\max \mathsf L (B) < \max \mathsf L (A)$, there is an
$M$-chain $\widehat y = y_0, \ldots, y_k = \overline y$ of
factorizations of $B$, and hence $z, \widehat z=Uy_0, Uy_1, \ldots, Uy_k=U\overline y = \overline z$ is an $M$-chain of factorizations
concatenating $z$ and $\overline z$.

\smallskip
(a) \,$\Rightarrow$\, (e) Without restriction, we may suppose that
$G_P^-$ is finite. The claim follows by Proposition \ref{rel_prop}
and Lemma \ref{3.3}.

\smallskip
(c) \,$\Rightarrow$\, (d) \ and (e) \,$\Rightarrow$\, (f) \ hold for
all atomic monoids (\cite[Proposition 1.4.2 and Theorem
1.6.3]{Ge-HK06a}).

\smallskip
(b) \,$\Rightarrow$\, (a), (d) \,$\Rightarrow$\, (a), and  (f)
\,$\Rightarrow$\, (a) \ Assume to the contrary that $G_P^+$ and
$G_P^-$ are both  infinite. We show that $\mathcal B (G_P)$ is not
locally tame, which implies that $H$ is not locally tame (\cite[Theorem 3.4.10.6]{Ge-HK06a}). Along the way, we show that
$\rho_2 (G_P) = \infty$ and that $\Delta (G_P)$ is infinite, which
by Lemma \ref{3.3} implies the according statements for $H$.

We set $a = \max G_P^-$ and $b = \min G_P^+$.
Using the notation of Lemma \ref{lem_gap}, let $U = V_{a,b}= a^{\alpha}
b^{\beta} \in \mathcal A (G_P)$. We pick an arbitrary $N \in \N_{\ge 2}$ and
show that $\mathsf t (G_P, U) \ge N$, which implies the assertion.

We intend to apply Lemma \ref{lem_gap} with $v=1$.
Thus, let $D = |a|(b+|a|)\gcd(a,b) $, let $b_1 \in G_P^+$ be such that
  \[\frac{b_1}{\lcm(a,b)} \ge N + D, \]
  and let $a_2 \in G_P^-$ be such that $|a_2| \ge (b_1+b)|a|$.
Let $\alpha_1,\alpha_2,\beta_1,\beta_2 \in \N$  be such that
$V_{a,b_1} = a^{\alpha_1} b_1^{\beta_1}$ and $V_{a_2,b} = a_2^{\alpha_2} b^{\beta_2}$ are elements of $\mathcal{A}(G_P)$.

We note that all conditions of Lemma \ref{lem_gap} with $v=1$ are fulfilled.
Since $\alpha \le b \le \alpha_1$ and $\beta \le |a| \le \beta_2$, we have $U \t V_{a,b_1} V_{a_2,b}$, and
therefore $ \mathsf{Z}(V_{a,b_1}V_{a_2,b})\cap U \mathsf{Z}(G_P) \neq \emptyset$.
Let $z \in \mathsf{Z}(V_{a,b_1}V_{a_2,b})\setminus \{V_{a,b_1} \cdot V_{a_2,b_1} \}$, which exists in view of $U|V_{a,b_1}V_{a_2,b}$.
By Lemma \ref{lem_gap}, we get that $t(z)\neq 0$, and thus that
\[|z|\ge  \frac{b_1}{\lcm(a,b)} - D \ge N .\]
This shows that $\max \Delta \bigl( \mathsf L (V_{a,b_1}V_{a_2,b}) \bigr) \ge N-2$, $\mathsf t(G_p,U)\geq N$ and
\[
\rho_2 (G_P) \ge \max  \mathsf L (V_{a,b_1}V_{a_2,b}) \ge N  \,.
\]

\smallskip
(a) \,$\Rightarrow$\, (g) \ This follows from Lemma \ref{lem_rhok}.

\smallskip
(g) \,$\Rightarrow$\, (f) \ We have $\rho_2 (H) \le M + \rho_1 (H) =
                           M+1$, where $M$ is as given by $(g)$.

\smallskip
(a) \,$\Rightarrow$\, (h) \ If $(a)$ holds, then $(d)$ and $(g)$
hold. Thus all assumptions of \cite[Theorem 4.2]{Ga-Ge09b} are
fulfilled, and $(h)$ follows.

\smallskip
(h) \,$\Rightarrow$\, (f) \ We have $\rho_2 (H) = \sup \mathcal V_2 (H)
                    < \infty$.
\end{proof}

\section{Arithmetical Properties stronger than the Finiteness of $G_P^+$ or $G_P^-$} \label{5}

Let $H$ be a Krull monoid and  $G_P \subset G$ as always (see
Theorem \ref{main-theorem-II} below). In this section, we discuss
arithmetical properties which are finite if $G_P$ is finite or $\min
\{ |G_P^+|, |G_P^-| \} = 1$, and whose finiteness implies that
$G_P^+$ or $G_P^-$ is finite. However, it will turn out that none of
the implications can be reversed (with one possible exception for $(c)\Rightarrow (b4)$, which remains open), and that the finiteness of these
properties cannot be characterized by the size of $G_P^+$ and
$G_P^-$ but also depends on the structure of these sets. We start
with some definitions and then formulate the main result.

\begin{definition}
Let $H$ be an atomic monoid and $\pi \colon \mathsf Z(H) \to
H_{\red}$ the factorization homomorphism.

\smallskip

\begin{enumerate}

\item For $z \in \mathsf Z(H)$, we denote by \ $\delta (z)$ \ the smallest
      $N \in \N_0$ with the following property:
      if $k \in \N$ is such that $k$ and $|z|$ are adjacent lengths of $\mathsf L \bigl( \pi(z) \bigr)$, then
      \[
      \mathsf{d}(z, \mathsf{Z}_k(a) )\le N \,.
      \]
      Globally, we define
      \[
      \delta (H) = \sup \{\, \delta (z) \mid z \in \mathsf Z (H) \} \in
      \N_0 \cup \{\infty\} \,,
      \]
      and we call \ $\delta (H)$ \ the \ {\it successive distance} \ of $H$.

\smallskip
\item We say that the \ {\it Structure Theorem for Sets of Lengths}
      holds
      (for the monoid $H$) \ if $H$ is atomic and there exist some $M \in
      \N_0$ and a finite, nonempty set $\Delta^* \subset \N$ such that, for
      every $a \in H$, the set of lengths $\mathsf L (a)$ is an {\rm AAMP}
      with some difference $d \in \Delta^*$ and bound $M$. In that case, we
      say more precisely that the Structure Theorem for Sets of Lengths
      holds with set $\Delta^*$ and bound $M$.
\end{enumerate}
\end{definition}

\medskip
\begin{theorem}
\label{main-theorem-II} Let $H$ be a Krull monoid and $\varphi
\colon H\to \mathcal{F}(P)$ a cofinal divisor homomorphism into a
free monoid such that the class group $G = \mathcal{C}(\varphi)$ is
an infinite cyclic group that we identify with $\mathbb Z$. We
denote by $G_P \subset G$  the set of classes containing prime
divisors and consider the following conditions{\rm \,:}
\smallskip
\begin{enumerate}
\item[(a)] $G_P$ \ is finite or \ $\min \{ |G_P^+|, |G_P^-| \} = 1$.

\medskip

\item[(b1)] The Structure Theorem for Sets of Lengths holds for $H$ with set $\Delta(G_P)$.

\item[(b2)] The successive distance $\delta (H)$ is finite.

\item[(b3)] The monotone catenary degree $\mathsf c_{\mon} (H)$ is
            finite.

\item[(b4)] There is an $M \in \N$ such that, for all $a \in H$ and for each two adjacent lengths \ $k, \, l \in
           \mathsf L (a)
           \cap [\min \mathsf L (a) + M, \, \max \mathsf L (a) - M]$, \ we
           have
           $\mathsf d \bigl( \mathsf Z_k (a), \mathsf Z_l (a) \bigr) \le M$.

\medskip
\item[(c)] $G_P^+$ \ or \ $G_P^-$ is finite.
\end{enumerate}
Then we have
\begin{enumerate}
\item Condition $(a)$ implies each of the conditions $(b1)$ to $(b4)$.

\item Each of the conditions $(b1)$ to $(b4)$ implies $(c)$.

\item  $(b2) \Rightarrow (b3) \Rightarrow (b4)$.
\end{enumerate}
\end{theorem}

\medskip
We briefly discuss the newly introduced arithmetical properties and
point out the trivial implications in the above result. The
successive distance of $H$ was introduced by Foroutan in
\cite{Fo06a} in order to study the monotone catenary degree.  For
Krull monoids with finite class group, an explicit upper bound for
the successive distance was recently given in \cite[Theorem
6.5]{Fr-Sc10a}. Note that, by definition, $\delta (H) < \infty$
implies that $\Delta (H)$ is finite. The significance of the
Structure Theorem for Sets of Lengths will be discussed at the
beginning of Section \ref{6}. Note that, if it holds for a monoid
$H$, then $H$ is a \BF-monoid with finite set of distances $\Delta
(H)$. Moreover, if $G_P = \Z$, then the Structure Theorem badly
fails: indeed, then every finite subset $L \subset \N_{\ge 2}$
occurs as a set of lengths by Kainrath's Theorem \cite[Theorem
7.4.1]{Ge-HK06a}; for recent progress in this direction see
\cite{Ch-Sc10a}. The implications $(b2) \Rightarrow (b4)$ and $(b3)
\Rightarrow (b4)$ follow from the definitions. A condition implying
$(b1)$ as well as $(b4)$ is given in Proposition \ref{4.2}. The
bound $M$ in $(b4)$ reflects the fact that, in many settings,
factorizations $z$ of an element $a \in H$ show more unusual
phenomena if their length $|z|$ is close either to $\max \mathsf L
(a)$ or to $\min \mathsf L (a)$ (the reader may want to consult
\cite[Theorem 4.9.2]{Ge-HK06a}, \cite[Theorem 3.1]{Fo-Ha06b},
\cite[Theorem 3.1]{Fo-Ha06a} and the associated examples showing the
relevance of the bound $M$).

In Sections \ref{6} and \ref{7}, we obtain results  showing that,
even under the more restrictive assumption that $\varphi$ is a
divisor theory, the Conditions $(b1)$ to $(b4)$ do not imply $(a)$
(Proposition \ref{STSL_prop1}), and $(c)$ does not imply $(b1)$ to
$(b3)$ (Theorem \ref{STSL_thm}, Proposition \ref{STSL_prop1},
Proposition \ref{STSL_prop2} and Proposition \ref{7.1}). Proposition \ref{STSL_prop2} shows
that $(b3)$ does not imply $(b2)$. Moreover, $(b1)$, $(b2)$ and
$(b3)$ may hold as well as may fail even if $\min \{ |G_P^+|,
|G_P^-| \} = 2$. Most of the observed phenomena (around the
non-reversibility of implications) have not been pointed out before
in any $v$-noetherian monoid, and in particular not in a Krull
monoid. Finally, by Theorem \ref{main-theorem-II}, a Krull monoid
$H$ satisfies strong arithmetical properties both when
$G_P$ is finite as well as when $\min \{ |G_P^+|, |G_P^-| \} =
1$. Note that an arithmetical difference between these two cases was
pointed out in Proposition \ref{tame-charact}.

The remainder  of this section is devoted to the proof of Theorem
\ref{main-theorem-II}, which heavily uses Theorem
\ref{main-theorem-I}. We start with the necessary  preparations. To
show that $(a)$ implies each of the Conditions $(b1)$ to $(b4)$, we
will construct transfer homomorphisms to finitely generated monoids.

\begin{lemma} \label{transfer-to-finite}
Let $G_0 \subset \mathbb{Z}$ be a condensed set with $\min \{
|G_0^+|,\,|G_0^-|\}=1$, say $G_0^- = \{-n\}$. The map
\[
\varphi \colon
\begin{cases}
\mathcal{B}(G_0)& \to \mathcal{F}(G_0 \setminus \{-n\})\\
B & \mapsto (-n)^{-\mathsf{v}_{-n}(B)}B
\end{cases}
\]
is a cofinal divisor homomorphism. Its class group
$\mathcal{C}(\varphi)$ is isomorphic to a subgroup of $\mathbb{Z}/n
\mathbb{Z}$, and the set of classes containing prime divisors
corresponds to $\{b+ n\mathbb{Z}\mid b \in G_0 \setminus \{-n\}\}$.
In particular, the class group of the Krull monoid
$\mathcal{B}(G_0)$ is a finite cyclic group.
\end{lemma}

\begin{proof}
Clearly,  $\varphi$ is a cofinal monoid homomorphism. In order to
show that $\varphi$ is a divisor homomorphism, let $A,B \in
\mathcal{B}(G_0)$ be such that $\varphi(A)\mid \varphi(B)$. We have to
verify that $A \mid B$, and for that it suffices to check that
$\mathsf{v}_{-n}(A)\le \mathsf{v}_{-n}(B)$. For each $C \in
\mathcal{B}(G_0)$, we have $\mathsf{v}_{-n}(C)= \sigma(C^+)/n$ and
$\sigma(C^+)=\sigma(\varphi(C))$. Since $\varphi(A)\mid \varphi(B)$,
we have $\sigma(\varphi(A))\le \sigma(\varphi(B))$, and thus
$\mathsf{v}_{-n}(A)\le \mathsf{v}_{-n}(B)$ follows.

Now, we show that, for $F_1, F_2 \in \mathcal{F}(G_0 \setminus
\{-n\})$, we have $F_1  \in F_2
\mathsf{q}(\varphi(\mathcal{B}(G_0)))$ if and only if
$\sigma(F_1)\equiv \sigma(F_2) \mod n $. This establishes the
results regarding $\mathcal{C}(\varphi)$ and the set of classes
containing prime divisors.

First, suppose that \ $\sigma(F_1) \equiv \sigma(F_2) \mod n$. We
note that $F_iF_j^{n-1}(-n)^{(\sigma(F_i) + (n-1)\sigma(F_j))/n}\in
\mathcal{B}(G_0)$, for $i,j \in \{1,2\}$. Thus, $F_j^n$ and
$F_iF_j^{n-1}$ are elements of $ \varphi(\mathcal{B}(G_0))$ for $i,j
\in \{1,2\}$. Since $F_1= F_2 (F_1F_2^{n-1})(F_2^{-n})$, the claim
follows. Since \ $\sigma(\varphi(C))\equiv 0 \mod n$ \ for each $C
\in \mathcal{B}(G_0)$, the converse claim follows.

By \cite[Theorem 2.4.7]{Ge-HK06a}, the class group of
$\mathcal{B}(G_0)$ is an epimorphic image of a subgroup of
$\mathcal{C}(\varphi)$, and thus it is a finite cyclic group.
\end{proof}

The following example shows that $\mathcal{C}(\varphi)$ can be a
proper subgroup of $\mathbb{Z}/n \mathbb{Z}$ and that
$\mathcal{C}(\varphi)$ can be distinct from the class group of
$\mathcal{B}(G_0)$. However, if $[G_0]= \mathbb{Z}$,  then
$\mathcal{C}(\varphi)=\mathbb{Z}/n\mathbb{Z}$; and, applying
\cite[Theorem 3.1]{Sc09f}, there is a simple and explicit method to
determine the class group of $\mathcal{B}(G_0)$ from
$\mathcal{C}(\varphi)$ as well as the subset of classes containing
prime divisors (note that $\mathcal{C}(\varphi)$ is a torsion
group).

\begin{example}
Let $d_1,d_2 \in \mathbb{N}_{\ge 2}$, $n=d_1d_2$ and $G_0= \{-n,
d_1\}$. Then $G_0$ fulfils all assumptions of Lemma
\ref{transfer-to-finite}, and with $\varphi$ as in Lemma
\ref{transfer-to-finite}, we get that $\mathcal{C}(\varphi)= \langle
d_1 + n \mathbb{Z}\rangle \subsetneq \mathbb{Z}/n \mathbb{Z}$.
However, $\mathcal{B}(G_0)$ is factorial, and thus its class group
is trivial.
\end{example}

\begin{proposition}
\label{nice_prop} Let $H$ be a Krull monoid and $\varphi \colon H\to
\mathcal{F}(P)$ a cofinal divisor homomorphism into a free monoid
such that the class group $G= \mathcal{C}(\varphi)$ is an  infinite
cyclic group that we identify with $\mathbb Z$. Let $G_P \subset G$
denote the set of classes containing prime divisors. Suppose that
$G_P$ is finite or that $\min \{|G_P^+|, |G_P^-|\}=1$. Then there
exists  a transfer homomorphism $\theta \colon H \to H_0$ into a
finitely generated monoid $H_0$ such that $\mathsf c(H, \theta)\le
2$. Moreover, the following statements hold.
\begin{enumerate}
  \item $\mathcal{L}(H)= \mathcal{L}(H_0)$, in particular, the Structure Theorem for Sets of Lengths holds for $H$ with $\Delta (H)=\Delta(H_0)$ and some bound $M$, and $\rho(H)= \rho(H_0)$ is finite and accepted.
  \item $\delta(H)= \delta(H_0) < \infty$.
  \item $\mathsf c_{\mon} (H) \le \max \{\mathsf c_{\mon} (H_0),2\} < \infty$.
\end{enumerate}
\end{proposition}

\begin{proof}
First we show the existence of the required transfer homomorphism.
For this, we recall that a monoid of zero-sum sequences over a finite
set is finitely generated (\cite[Theorem 3.4.2]{Ge-HK06a}). If $G_P$
is finite, then $\boldsymbol \beta \colon H \to \mathcal B (G_P)$
has the desired properties by Lemma \ref{3.3}. Now suppose that $\min
\{|G_P^+|, |G_P^-|\} = 1$, say $G_P^-= \{-n\}$, and set  $G_0 = \{b
+ n \Z \mid b \in G_P^+ \} \subset \mathbb{Z}/n \mathbb{Z}$. Using
Lemmas \ref{3.3} and \ref{transfer-to-finite}, we have block
homomorphisms $\boldsymbol \beta \colon H \to \mathcal B (G_P)$ and
$\boldsymbol \beta' \colon \mathcal B (G_P) \to \mathcal B (G_0)$.
By Lemma \ref{3.2}, the composition $\theta = \boldsymbol \beta'
\circ \boldsymbol \beta \colon H \to \mathcal B (G_0)$ still has the
required properties.

Again, by  Lemmas \ref{3.2} and \ref{3.3}, it suffices to verify the
additional statements for finitely generated monoids: we refer to
\cite[Theorem 4.4.11]{Ge-HK06a} for the Structure Theorem, to
\cite[Theorem 3.1.4]{Ge-HK06a} for the elasticity and the successive
distance, and to \cite[Theorem 5.1]{Fo06a} for the monotone catenary
degree.
\end{proof}

\begin{lemma} \label{monotone-delta1}
Let $H$ be an  atomic monoid, $a \in H$ and $z, \,z' \in \mathsf Z
(a)$ and $l = \bigl| |z| - |z'| \bigr|$. Then there exists some $z''
\in \mathsf Z (a)$ such that $|z''| = |z'|$ and \ $\mathsf d (z,
z'') \le l \delta (H)$.
\end{lemma}

\begin{proof}
See \cite[Lemma 3.1.3]{Ge-HK06a}.
\end{proof}

\begin{lemma} \label{monotone-delta2}
Let $H$ be an atomic monoid with  $\delta (H) < \infty$. Let $M \in
\N$, $a \in H$, $u \in \mathcal A (H_{\red})$ and $z, \widehat z,
\overline z \in \mathsf Z (a)$ be such that
\[
|z| \le |\overline z|, \ u \t \overline z, \ u \t \widehat z \quad
\text{and} \quad \mathsf d (z, \widehat z) \le M \,.
\]
Then there is a $z' \in \mathsf Z (a) \cap u \mathsf Z (H)$ such
that $|z| \le |z'| \le |\overline z|$ and $\mathsf d (z, z') \le M +
\Bigl( M + \max
 \Delta (H) \Bigr) \delta
(H)$.
\end{lemma}

\begin{proof}
Let $v \in H$ be such that $vH^{\times}=u$.
We set $b = v^{-1}a$, $\overline z = u \overline y$ and $\widehat z
= u \widehat y$, where $\overline y, \widehat y \in \mathsf Z (b)$. If
$|z| \le |\widehat z| \le |\overline z|$, then  $z' = \widehat z$
fulfills the requirements. If not, then either $|\widehat z| < |z|$
or $|\overline z| < |\widehat z|$, and we decide these two cases separately.

\smallskip
\noindent {\sc Case 1:} \, $|\widehat z| < |z|$.

Since $|\widehat y| = |\widehat z|-1 \in \mathsf L (b)$ and
$|\overline y| = |\overline z| - 1 \in \mathsf L (b)$, there is a
\[
k \in \mathsf L (b) \cap [|z|-1, |\overline z|-1] \quad \text{with}
\quad k \le |z|-1+\max \Delta (H) \,.
\]
Let $y'' \in \mathsf Z (b)$ with $|y''| = k$. Then
\[
\begin{aligned}
|y''| - |\widehat y| & = k - |\widehat z|+ 1 \le |z|-1+\max \Delta
(H) - |\widehat z|+ 1 \\
 & \le \mathsf d ( z, \widehat z) + \max \Delta (H) \le M + \max
 \Delta (H) \,.
\end{aligned}
\]
Thus, by Lemma \ref{monotone-delta1}, there is a $y' \in \mathsf Z
(b)$ with $|y'| = |y''|$ and $\mathsf d ( \widehat y, y') \le \Bigl(
M + \max
 \Delta (H) \Bigr) \delta
(H)$. Then $z' = u y' \in \mathsf Z (a) \cap u \mathsf Z (H)$ with
$|z'| = 1 + k \in [|z|, |\overline z|]$ and
\[
\mathsf d (z, z') \le \mathsf d (z, \widehat z) + \mathsf d (u
\widehat y, u y') \le M + \Bigl( M + \max
 \Delta (H) \Bigr) \delta
(H) \,.
\]

\smallskip
\noindent {\sc Case 2:} \, $|\overline z| < |\widehat z|$.

By Lemma \ref{monotone-delta1}, there is a $y' \in \mathsf Z (b)$
with $|y'| = |\overline y|$ and
\[
\begin{aligned}
\mathsf d (\widehat y, y') & \le \Bigl( |\widehat y| - |\overline y|
\Bigr) \delta (H) = \Bigl( |\widehat z| - |\overline z| \Bigr)
\delta (H) \\
 & \le \Bigl( |\widehat z| - |z| \Bigr)\delta (H) \le \mathsf d
 (\widehat z, z) \delta (H) \le M \delta (H) \,.
\end{aligned}
\]
Then $z' = uy' \in \mathsf Z (a) \cap u \mathsf Z (H)$ with $|z'| =
|\overline z|$ and
\[
\mathsf d (z, z') \le \mathsf d (z, \widehat z) + \mathsf d (u
\widehat y, u y') \le M + M\delta (H) \,. \qedhere
\]
\end{proof}

\begin{proposition} \label{6.4}
Let $H$ be a Krull monoid and $\varphi \colon H\to \mathcal{F}(P)$ a
cofinal divisor homomorphism into a free monoid with  infinite
cyclic class group $\mathcal C(\varphi)$. If the successive distance
$\delta (H)$ is finite, then the monotone catenary degree $\mathsf
c_{\mon} (H)$ is finite.
\end{proposition}

\begin{proof}
We set $G = \mathcal C(\varphi)$, identify $G$ with $\Z$ and denote
by $G_P \subset G$  the set of classes containing prime divisors.
Suppose that $\delta (H) < \infty$. Lemma \ref{3.3} shows $\delta
(H) = \delta (G_P)$ and that it suffices to verify that $\mathsf
c_{\mon} (G_P) < \infty$. Note that $\Delta (G_P)$ is finite (since
$\delta(G_P)<\infty$), and thus by Theorem \ref{main-theorem-I} we
get that (say) $G_P^-$ is finite.

We set $M = \bigl( |\min G_P| + |G_P^-|^2
 \bigr) \ |\min G_P | $ and assert that
\[
 \mathsf c_{\mon} (G_P) \le M + \Bigl( M + \max
 \Delta (H) \Bigr) \delta
(H) = M^{\ast} \,.
\]
For this, we have to show that $\mathsf c_{\mon} (A) \le M^{\ast}$
for all $A \in \mathcal B (G_P)$, and we proceed by induction on
$\max \mathsf L (A)$.

If $A \in \mathcal B (G_P)$ with $\max \mathsf L (A) = 1$, then $A
\in \mathcal A (G_P)$ and $\mathsf c_{\mon} (A) = 0$. Now let $A \in
\mathcal B (G_P)$ with $\max \mathsf L (A) > 1$ and suppose that
$\mathsf c_{\mon} (B) \le M^{\ast}$ for all $B \in \mathcal B (G_P)$
with $\max \mathsf L (B) < \max \mathsf L (A)$.

\smallskip
We pick $z, \overline z \in \mathsf Z (A)$ with $|z| \le |\overline
z|$ and must find a monotone $M^{\ast}$-chain of factorizations
from $z$ to $\overline z$.

By Lemma \ref{3.4} there is a $U
\t \overline z$ with $U \in \mathcal A (G_P)$ and a
$\widehat z \in \mathsf Z (A) \cap U \mathsf Z (G_P)$ such that $\mathsf d (z, \widehat z) \le M$. By Lemma
\ref{monotone-delta2}, there is a $z' \in \mathsf Z (A) \cap U
\mathsf Z (G_P)$ such that $|z| \le |z'| \le |\overline z|$ and
$\mathsf d (z, z') \le M^{\ast}$. Now we set
\[
B = U^{-1}A,\quad \overline z = U \overline y \quad \text{and} \quad z' =
U y',
\]
where $\overline y, y' \in \mathsf Z (B)$. Since $\max \mathsf L (B)
< \max \mathsf L (A)$, the induction hypothesis gives a monotone
$M^*$-chain $y' = y_1, \ldots, y_k = \overline y$ of factorizations
of $B$ from $y'$ to $\overline y$. Therefore
\[
z, z' = Uy'=Uy_1, Uy_2, \ldots , Uy_k = U \overline y = \overline z
\]
is a monotone $M^{\ast}$-chain of factorizations of $A$ from $z$ to
$\overline z$.
\end{proof}

\bigskip
\begin{proof}[{\bf Proof of Theorem \ref{main-theorem-II}}]
3.  The implication $(b3) \Rightarrow (b4)$ follows since, for $a
\in H$ and each two adjacent lengths \ $k, \, l \in \mathsf L (a)$,
we have, by definition, $\mathsf d \bigl( \mathsf Z_k (a), \mathsf
Z_l (a) \bigr) \le \mathsf c_{\mon} (H)$. The implication $(b2)
\Rightarrow (b3)$ is Proposition \ref{6.4}.

\smallskip
1. By Proposition \ref{nice_prop}, we know that $(a)$ implies $(b1)$, $(b2)$, and $(b3)$; and,  by part 3, we know that
$(b3)$ implies $(b4)$.

\smallskip
2. By definition, each of $(b1)$, $(b2)$ and $(b3)$ implies the
finiteness of $\Delta (H)$. Thus, Theorem \ref{main-theorem-I}
implies the assertion. It remains to show that $(b4)$ implies $(c)$.

Suppose that $(b4)$ holds with some $M \in \mathbb{N}$ and assume
to the contrary that $(c)$ does not hold, i.e., $G_P^+$ and $G_P^-$
are both infinite. We proceed similarly to the proof of Theorem
\ref{main-theorem-I}, part $(b) \Rightarrow (a)$.

We set $a = \max G_P^-$ and $b = \min G_P^+$ and let $\alpha \in
[1,b]$  and $\beta \in [1, |a|]$ be such that $V_{a,b}= a^{\alpha}b^{\beta}
\in \mathcal{A}(G_P)$. We intend to apply Lemma \ref{lem_gap} with
$v=3$. Thus, let $D = 3|a|(b+|a|)\gcd(a,b)$, let $b_1 \in G_P^+$
with
\[\frac{b_1}{\lcm(a,b)} \ge  2 D + M, \]
and let $a_2 \in G_P^-$ with $|a_2| \ge (3 b_1 +b)|a| $.
Let $\alpha_1,\alpha_2,\beta_2,\beta_2 \in \mathbb{N}$ be such that
$V_{a,b_1} = a^{\alpha_1}b_1^{\beta_1}$ and  $V_{a_2,b} = a_2^{\alpha_2}b^{\beta_2}$ are elements of $\mathcal{A}(G_P)$.

First, we assert that there exist $z_0, z_1, z_2, z_3 \in \mathsf{Z}((V_{a,b_1}V_{a_2,b})^3)$ with, where $t(\cdot)$ is defined as in Lemma \ref{lem_gap},
\[t(z_0)< t(z_1)< t(z_2)< t(z_3).\]
We note that $V_{a,b}\t V_{a,b_1}V_{a_2,b}$ (by the same reasoning used in the proof of Theorem \ref{main-theorem-I}), and thus there exists some $y \in \mathsf{Z}(V_{a,b_1}V_{a_2,b})$ with $t(y)\neq 0$.
For $i \in [0,3]$, we  set $z_i= y^i(V_{a,b_1}\cdot V_{a_2,b})^{3-i}$.
Then we have $t(z_i)= i t(y)$, establishing the claim.

Let $z_0', z_1', z_2', z_3' \in \mathsf{Z}((V_{a,b_1}V_{a_2,b})^3)$ be such that $t(z_0')< t(z_1')< t(z_2')< t(z_3')$ and such that
there exists no $z \in \mathsf{Z}((V_{a,b_1}V_{a_2,b})^3)$ with
$t(z_1')< t(z) <  t(z_2')$.

By Lemma \ref{lem_gap}, we get, for $i\in [0,2]$,
that
\[|z_{i+1}'|-|z_i'| \ge  \frac{b_1}{\lcm(a,b)} \bigl( t(z_{i+1}') - t(z_i')\bigr)  -  2D  \ge M .
\]
Since
$\min \mathsf{L}((V_{a,b_1}V_{a_2,b})^3) \le |z_0'| < |z_1'| < |z_2'|< |z_3'| \le \max \mathsf{L}((V_{a,b_1}V_{a_2,b})^3)$,
we get  that
\[|z_1'|,|z_2'| \in  \bigl[\min\mathsf{L}\bigl((V_{a,b_1}V_{a_2,b})^3\bigr) + M,  \max \mathsf{L}\bigl((V_{a,b_1}V_{a_2,b})^3\bigr)-M \bigr].\]

Let
\[
k = \max \left( \mathsf{L}\bigl((V_{a,b_1}V_{a_2,b})^3\bigr)  \cap    \left[\frac{b_1}{\lcm(a,b)}
\ t(z_1') - D ,\frac{b_1}{\lcm(a,b)} \  t(z_1') + D \right]  \right)
\]
and
\[
l = \min \left( \mathsf{L}\bigl((V_{a,b_1}V_{a_2,b})^3\bigr)  \cap  \left[\frac{b_1}{\lcm(a,b)} \
t(z_2') - D ,\frac{b_1}{\lcm(a,b)} \  t(z_2') + D \right]  \right)\ ;
\]
note that, by Lemma \ref{lem_gap}, $|z_1'|$ and $|z_2'|$ are elements of the former and the latter set, respectively,
and also note that the two intervals are disjoint. In particular, we have $|z_1'|\le k < l \le |z_2'|$.
Since there exists no  $z \in \mathsf{Z}((V_{a,b_1}V_{a_2,b})^3)$ with
$t(z_1')< t(z) <  t(z_2')$, it follows by Lemma \ref{lem_gap} that $k$ and $l$ are adjacent lengths.
Since $k - l \ge \frac{b_1}{\lcm(a,b)}   -  2D  \ge M  $ and by \eqref{E:Dist},
 we have $\mathsf d \bigl( \mathsf Z_k (a), \mathsf Z_l (a) \bigr) \ge M + 2$, a contradiction to the assumption that $(b4)$ holds with $M$.
\end{proof}

\section{The Structure Theorem for Sets of Lengths} \label{6}

The Structure Theorem for Sets of Lengths is a central finiteness
result in factorization theory.  Apart from Krull monoids---which
will be discussed below---the Structure Theorem holds, among others,
for weakly Krull domains with finite $v$-class group and for Mori
domains $A$ with complete integral closure $\widehat A = R$ for
which the conductor $\mathfrak f = (A \DP R) \ne \{0\}$ and
$\mathcal C (R)$ and $R/ \mathfrak f$ are both finite (see
\cite[Section 4.7]{Ge-HK06a} for an overview, and \cite{Ge-Gr09b,
Ge-Ka10a} for recent progress). Moreover, it was recently shown that
the Structure Theorem is sharp for Krull monoids with finite class
group \cite{Sc09a}.

Let $H$ be a Krull monoid and  $G_P \subset G$ as always. By Theorem
\ref{main-theorem-II}, it suffices to consider the situation when
$G_P^+$ is finite and $2 \le |G_P^-| < \infty$. Essentially, all
results so far which establish the Structure Theorem  for some class
of monoids  use the machinery of pattern ideals and tame generating
sets (presented in detail in \cite[Section 4.3]{Ge-HK06a}). First,
we repeat these concepts and outline their significance for the
Structure Theorem.  However, Proposition \ref{4.6}  shows that in
our situation this approach is not applicable in general. The main
result of this section, Theorem \ref{STSL_thm}, provides a full
characterization of when the Structure Theorem holds. Although the
setting is special, it shows that, in Theorem \ref{main-theorem-II},
condition $(b1)$ does not imply condition $(a)$, and it
provides---together with Proposition \ref{4.6}---the first example
of any Krull monoid for which the Structure Theorem holds without
tame generation of  pattern ideals. Furthermore, note by Lemma
\ref{lem_char} that, for  the sets $G_P$ considered in Theorem
\ref{main-theorem-II}, there actually exists a Krull monoid such
that $G_P$ is the set of classes containing prime divisors with
respect to a divisor theory of $H$.

Likewise, all previous examples of monoids $H$ with finite monotone
catenary degree $\mathsf c_{\mon}(H)$ have been achieved by using
that $\delta(H)$ is finite. However, in Proposition
\ref{STSL_prop2}, we give the first example of a monoid $H$ with
$\mathsf c_{\mon}(H)<\infty$ but $\delta(H)=\infty$.

\begin{definition} \label{def-tamely-gen}
Let $H$ be  an atomic monoid, let $\mathfrak a \subset H$ and let $A \subset
\Z$ be a finite nonempty subset.

\begin{enumerate}
\item We say that a subset \ $L \subset \Z$ \  {\it contains the
      pattern \ $A$} \ if there exists some $y \in \Z$ such that $y
      +A \subset L$. We denote by \ $\Phi(A)  = \Phi_H (A)$ \ the set of all $a \in H$ for which \ $\mathsf L(a)$
      contains the pattern $A$.

\smallskip

\item Now $\mathfrak a$ is called a \ {\it
      pattern ideal} \ if $\mathfrak a = \Phi (B)$ for some finite nonempty subset
      $B \subset \Z$.

\smallskip

\item A subset $E \subset H$ is called a \ {\it tame generating
      set} \ of $\mathfrak a$ \ if $E \subset \mathfrak a$ and there
      exists some $N \in \N$ with the following property:
      for every $a \in \mathfrak a$, \ there exists some \ $e \in E$ \ such that
      \[
      e \t a\,, \quad \sup \mathsf L(e) \le N \quad \text{and} \quad \mathsf t(a, \mathsf Z(e)) \le N\,.
      \]
      In this case, we call $E$ a \ {\it tame generating set with bound \ $N$}, and we say that
      $\mathfrak a$ \ is \ {\it tamely generated}.
\end{enumerate}
\end{definition}

The significance of tamely generated pattern ideals stems from the
following result.

\medskip
\begin{proposition} \label{4.2}
Let $H$ be a \BF-monoid with finite nonempty set of distances
$\Delta (H)$ and suppose that  all pattern ideals of $H$ are tamely
generated. Then there exists a constant $M \in \mathbb N_0$ such
that  the following properties are satisfied{\rm \,:}
\begin{enumerate}
\item[(a)] The Structure Theorem for Sets of Lengths holds with
           $\Delta (H)$ and bound $M$.

\smallskip
\item[(b)] For all $a \in H$ and for each two adjacent lengths \ $k, \, l \in
           \mathsf L (a)
           \cap [\min \mathsf L (a) + M, \, \max \mathsf L (a) - M]$, \ we
           have
           $\mathsf d \bigl( \mathsf Z_k (a), \mathsf Z_l (a) \bigr) \le M$.
\end{enumerate}
\end{proposition}

\begin{proof}
The first statement follows from \cite[Theorem 4.3.11]{Ge-HK06a} and
the second from \cite[Proposition 5.4]{Ge-Ka10a}.
\end{proof}

\medskip
\begin{proposition} \label{4.6}
Let $H$ be a Krull monoid and $\varphi \colon H\to \mathcal{F}(P)$ a
cofinal divisor homomorphism into a free monoid such that the class
group $G= \mathcal{C}(\varphi)$ is an  infinite cyclic group that we
identify with $\mathbb Z$. Let $G_P \subset G$ denote the set of
classes containing prime divisors. Suppose that
\begin{itemize}
\item $G_P^+$ is infinite and

\item there are \ $a_1, \, a_2 \in G_P^-$ \ and \ $b \in G_P^+$  \ such
that
      \[
       a_1 \frac{\gcd (a_2,b)}{\gcd (a_1, a_2, b)} \equiv a_2 \frac{\gcd (a_1,b)}{\gcd (a_1, a_2, b)}  \mod b     \quad {but} \quad  a_1 \frac{\gcd (a_2,b)}{\gcd (a_1, a_2, b)} \neq a_2 \frac{\gcd (a_1,b)}{\gcd (a_1, a_2, b)}  \,.
      \]
\end{itemize}
Then both $H$ and $\mathcal B (G_P)$ have a pattern ideal which is
not tamely generated.
\end{proposition}

\begin{proof}
By \cite[Proposition 3.14]{Ge-Gr09b}, it suffices to show that
$\mathcal B (G_P)$ has a pattern ideal which is not tamely
generated.

First we show that $\mathcal B (\{a_1, a_2, b\})$  is
half-factorial. By Lemma \ref{transfer-to-finite}, it suffices to
show that $\mathcal B ( \{ a_1 + b \Z, a_2 + b \Z \})$ is
half-factorial. By \cite[Proposition 5]{Ge90d}, this follows by
(indeed, it is equivalent to) the congruence that $a_1$, $a_2$, and
$b$ fulfil by assumption.

We set $\alpha_1 = b/\gcd(a_1,b)$, $\beta_1= |a_1|/\gcd (a_1,b)$,
$\alpha_2 = b/\gcd( a_2, b)$, $\beta_2 = |a_2| /\gcd (a_2, b)$ and
observe that, by rearranging our assumption $a_1 \frac{\gcd (a_2,b)}{\gcd (a_1, a_2, b)} \neq a_2 \frac{\gcd (a_1,b)}{\gcd (a_1, a_2, b)}$,  we have $d=a_1 \alpha_1- a_2
\alpha_2\neq 0$, say $d>0$. Noting that $\alpha_1a_1=\lcm(a,b)$ and $\alpha_2a_2=\lcm(a_2,b)$, we can consider the two atoms
\[
U_1 = a_1^{\alpha_1} b^{\beta_1} \quad \text{and} \quad U_2 =
{a_2}^{\alpha_2} b^{\beta_2} \in \mathcal A (G_P) \,.
\]
Since $G_P^+$ is infinite, it contains arbitrarily large elements.
Let $N \in G_P^+\setminus \{b\}$.
We define
\[
\gamma = \min \{\mathsf v_N (U) \mid U \in \mathcal A ( \{a_1, a_2,
b, N\}) \ \text{with} \ N \t U \} \,.
\]
Since $N^{|a_1|} a_1^N \in \mathcal B (G_P)$, it follows that $\gamma
\in [1,  |a_1|]$. Now we pick an atom $U_N \in \mathcal A ( \{ a_1,
a_2, b, N \})$ with $\gamma = \mathsf v_N (U_N)$ for which $\mathsf
v_{b} (U_N)$ is minimal, say
\[
U_N = N^{\gamma} b^{\beta} a_1^{M_1} {a_2}^{M_2} \in \mathcal A
(G_P), \quad \text{where} \quad \beta, \gamma, M_1, M_2 \in \N_0 \
\text{depend on}\  N \,.
\]
If $M_2 \ge |a_1|$, then $U'_N=U_N a_1^{|a_2|} {a_2}^{a_1}$ has sum zero,
and by the minimality of $\mathsf v_N (U_N)$ and $\mathsf
v_{b} (U_N)$, it is an atom (as each atom must have at least one positive element). Thus, we may additionally choose $U_N$ such that $M_2 < |a_1|$, which implies (recall $a_2<0$)
\begin{equation}
\label{eq_M1}
M_1 = \frac{1}{|a_1|} \Bigl( \gamma N + \beta b + a_2M_2 \Bigr)  \ge
\frac{1}{|a_1|} \Bigl( \gamma N + a_2|a_1| \Bigr) \ge \frac{N}{|a_1|}+a_2 \,.
\end{equation}
In view of this inequality, we may suppose that $N$ is sufficiently large to guarantee that $M_1\ge |a_2| \alpha_1 \alpha_2 $.
Note that, since $U_N$ is an atom and $M_1\ge |a_2| \alpha_1 \alpha_2 \geq \alpha_1$, we have $\beta < \beta_1$. We consider the element
\[
A_N = U_N {U_2}^{M_1} \in \mathcal B (G_P) \,.
\]
Let $k \in \Bigl[ 0, \bigl\lfloor \frac{M_1}{|a_2| \alpha_1 \alpha_2}
\bigr\rfloor \Bigr]$. Then we have
\[
U_{N,k} =  N^{\gamma} b^{\beta} a_1^{M_1 + (a_2 \alpha_1 \alpha_2)k}
{a_2}^{M_2 + (|a_1| \alpha_1 \alpha_2)k} \in \mathcal B (G_P) \,,
\]
and by the minimality of $\gamma$ and $\beta$, it follows that $U_{N,k} \in
\mathcal A (G_P)$. Clearly, we get
\[
z_{N,k} = U_{N,k} U_1^{-a_2 \alpha_2 k} {U_2}^{M_1 + a_1 \alpha_1 k}
\in \mathsf Z (A_N) \,.
\]
This shows that
\be\label{spiffyspaff}
\mathsf L (A_N) \supset   \left\{ M_1 + 1 + d k \Bigm| k \in \left[ 0, \left\lfloor \frac{M_1}{|a_2| \alpha_1
\alpha_2} \right\rfloor \right] \right\} \,.
\ee
Thus, we have $A_N \in \Phi ( \{0, d \})$ for each sufficiently large $N \in G_P^+$.

\medskip
Let $E_N \in \Phi ( \{0, d \})$ with $E_N \t A_N$. Since $\{a_1,
a_2, b \}$, is half-factorial, it follows that $N \t E_N$. By the
definition of $\gamma$, there is a  $U_N' \in \mathcal
A (G_P)$ with $N^{\gamma} \t U_N' \t E_N$. Note that \cite[Lemma
1.6.5.6]{Ge-HK06a} shows that $\mathsf t (A_N, U_N') \le \mathsf t
(A_N, \mathsf Z (E_N) )$.

Let $A_N = U_N' W_N$ with $W_N \in \mathcal B (G_P)$. Then $\supp
(W_N) = \{a_1, a_2, b \}$ and hence $|\mathsf L (W_N)| = 1$. Thus
all factorizations in $\mathsf Z (A_N) \cap U_N' \mathsf Z (G_P)$
have the same length. We pick some factorization $z_N \in \mathsf Z
(A_N)  \cap U_N' \mathsf Z (G_P)$. Clearly, there is a factorization
$z_N^* \in \mathsf Z (A_N)$ such that (in view of \eqref{spiffyspaff})
\[
\begin{aligned}
\bigl| |z_N| - |z_N^*| \bigr|  \ge \frac{\max \mathsf L (A_N) - \min
\mathsf L (A_N)}{2}  \ge \frac{d}{2} \left\lfloor \frac{M_1}{|a_2|
\alpha_1 \alpha_2} \right\rfloor \,.
\end{aligned}
\]
This implies that
\[
\begin{aligned}
\mathsf t (A_N, \mathsf Z (E_N) ) \ge \mathsf t (A_N, U_N') & \ge
\min \{ \mathsf d (z_N^*, y_N) \mid y_N
\in \mathsf Z (A_N) \cap U_N' \mathsf Z (G_P) \} \\
 & \ge \min \{ \bigl| |z_N^*| - |y_N| \bigr| \mid y_N
\in \mathsf Z (A_N) \cap U_N' \mathsf Z (G_P) \} \\
 & \ge \bigl| |z_N| - |z_N^*| \bigr| \ge \frac{d}{2} \left\lfloor \frac{M_1}{|a_2| \alpha_1 \alpha_2} \right\rfloor \,.
\end{aligned}
\]
Since $N$ can be arbitrarily large and by \eqref{eq_M1}, we get that $\Phi ( \{0, d \})$ is not tamely generated.
\end{proof}

We will frequently make use of the following  simple observation.
Let $G$ be an abelian group and $G_1 \subset G_0 \subset G$
subsets. Then $\mathcal B (G_1) \subset \mathcal B (G_0)$ is a
divisor-closed submonoid, and hence $\mathcal{L}(G_1) \subset
\mathcal{L}(G_0)$. Therefore, if the Structure Theorem holds for
$\mathcal B (G_0)$, then it holds for $\mathcal B (G_1)$. In
particular, if condition $(b)$ holds, then
the Structure Theorem holds for all $\mathcal B (G_0)$ with $G_0
\subset G_P$, and if $(b)$ fails, then the Structure Theorem fails
for all $\mathcal B (G_0)$ with $G_P \subset G_0$---where $G_P$ is as below.

\medskip
\begin{theorem} \label{STSL_thm}
Let $H$ be a Krull monoid and $\varphi \colon H\to \mathcal{F}(P)$ a
cofinal divisor homomorphism into a free monoid such that the class
group $G= \mathcal{C}(\varphi)$ is an  infinite cyclic group that we
identify with $\mathbb Z$. Let $G_P \subset G$ denote the set of
classes containing prime divisors. Suppose that $1 \in G_P^+$ and
$G_P^- = \{-d,-1\}$ for some $d \in \N$. Then the following
statements are equivalent{\rm \,:}
\begin{enumerate}
\item[(a)] The Structure Theorem for Sets of Lengths holds for $H$.

\smallskip
\item[(b)] $G_P^+ \setminus d \mathbb{Z}$ \ is finite  or  a subset of \ $1+ d\mathbb{Z}$.
\end{enumerate}
\end{theorem}

\medskip
The remainder  of this section is devoted to the proof of Theorem
\ref{STSL_thm}.

\begin{lemma} \label{STSL_lemt}
Let $H$ be an atomic monoid. Suppose that there
exists some $e \in \mathbb{N}$ such that, for each $N \in \mathbb{N}$,
there exists some $a \in H$ such that  $\mathsf{L}(a)\cap [\min
\mathsf{L}(a),\min \mathsf{L}(a)+N]\subset \min \mathsf{L}(a)+e
\mathbb{Z}$, yet $\mathsf{L}(a)\not\subset \min \mathsf{L}(a)+ e
\mathbb{Z}$. Then the Structure Theorem does not hold for $H$.
\end{lemma}

\begin{proof}
We assume to the contrary that there exists some finite  nonempty
set $\Delta^{\ast}\subset \mathbb{N}$ and some $M\in \mathbb{N}$
such that, for each $b \in H$, the set $\mathsf{L}(b)$ is an AAMP with
difference  $d\in \Delta^{\ast}$ and bound $M$.

Let $D= 2\ \lcm ( \Delta^{\ast})$. Let $N \ge 2M + D$ and let $a \in H$ with the properties from the statement of the lemma. Let $l_1 =\min \mathsf{L}(a)$
and $l_2= \max \mathsf{L}(a)$. Note that $l_2 \ge l_1+N$ (by the property assumed for $a$).
By assumption, we get that $\mathsf{L}(a)$ is an AAMP, i.e.,
\[
\mathsf  L(a) = y + (L' \cup L^* \cup L'') \, \subset \, y +
\mathcal D + d \Z
\]
where $d \in \Delta^{\ast}$, $\{0,d\}\subset \mathcal{D}\subset [0,d]$, $L^*$ is finite nonempty with $\min L^* = 0$ and
$L^* =(\mathcal D + d \Z) \cap [0, \max L^*]$,
$L' \subset [-M, -1]$ \ and \ $L'' \subset \max L^* + [1,M]$, and $y \in \mathbb N$.

Since $$l_2\geq l_1+N\geq l_1+2M+D\geq l_1-\min L'+M+D,$$ it follows that $[l_1 - \min L', l_1- \min L'+D-1]\cap \mathsf{L}(a)\subset L^{\ast}$, and thus
\[
[l_1 - \min L', l_1- \min L'+D-1]\cap \mathsf{L}(a)= [l_1 - \min L', l_1- \min L'+D-1]\cap (y + \mathcal{D}+d\mathbb{Z}).
\]
On the other hand, by the property assume for $a$, and since $N\geq 2M+D\geq -\min L'+D$, we have
\[
[l_1 - \min L', l_1- \min L'+D-1]\cap \mathsf{L}(a)\subset l_1+e\mathbb{Z}.
\]
Thus
\[A=[ - \min L', - \min L'+D-1]\cap (y-l_1 + \mathcal{D}+d\mathbb{Z}) \subset e\mathbb{Z}.\]
Since $D \ge 2d$, it follows that for each $d'\in \mathcal{D}$ there exists some $k\in \mathbb{Z}$ and $\epsilon \in \{-1,1\}$
such that $y-l_1+d'+kd,y-l_1+d'+(k+\epsilon) d \in A $. Thus $e \mid d$ and, furthermore, $e \mid y-l_1+d'$.
Consequently,  $y + \mathcal{D}+d\mathbb{Z} \subset l_1+e \mathbb{Z}$. This yields a contradiction, since $\mathsf{L}(a) \subset y + \mathcal{D}+d\mathbb{Z}$, yet $\mathsf{L}(a)\not \subset l_1 +e \mathbb{Z}$ by hypothesis.
\end{proof}

\begin{lemma}
\label{STSL_lem1} Let $d \in \mathbb{N}$, $e \in [2,d-1]$ with
$\gcd(e,d)>1$ and $G_0 \subset \Z$. If $\{-d,-1,1\}\subset G_0$ and
$G_0^+ \cap (e+d\mathbb{Z})$ is infinite, then the Structure Theorem
does not hold for $\mathcal B(G_0)$.
\end{lemma}

\begin{proof}
We may assume $d \ge 4$, since otherwise there exists no $e
\in [2,d-1]$ with $\gcd(e,d)>1$.
Let $k \in \mathbb{N}$ such that $e+dk \in G_0$; by assumption, we
know that arbitrarily large $k$ with this property exist, and we thus
may impose that $k \ge 10$. Let $f \in \mathbb{N}$ be minimal such that $ef
\in d \mathbb{N}$, say $ef= du$. Since $\gcd(e,d)>1$, we see that $f \in [2, d/2]$ and $u
\le e/2 \leq d/2$. We consider the sequence
\[B= (e+dk)^f(-d)^{u+fk}(-1)^{d(u+fk)}1^{d(u+fk)}.\]
Since $ef=du$, we have $B \in \mathcal{B}(G_0)$. First, we consider two
specific factorizations of $B$. Then, we investigate the length of
all factorizations of $B$ of small length. Let
\[z_1 = ((e+dk)^f(-d)^{u+fk})\cdot ((-1)1)^{d(u+fk)}\]
and
\[z_2 = ((e+dk)(-1)^{e+dk})^f\cdot ((-d)1^d)^{u+fk}.\]
We note that $z_1, z_2 \in \mathsf{Z}(B)$ and that $|z_1|=
1+d(u+fk)$ and $|z_2|= f+(u+fk)$. Since $f-1 \notin
(d-1)\mathbb{Z}$ (as $f\in [2,d/2]$), we have $|z_1|-|z_2|\notin (d-1)\mathbb{Z}$.

We claim that there exists an absolute positive constant $c$ such
that, for each
 $z \in \mathsf{Z}(B)$ with
\[|z| \le |z_2| + c(d-1)k,\]
we have
\[|z| - |z_2| \in (d-1)\mathbb{N}_0.\]
By Lemma \ref{STSL_lemt} and since $k$ can be arbitrarily large, this implies that
the Structure Theorem does not hold. Thus, it suffices to establish
this claim. For definiteness, we set $c=1/6$ (it is apparent from the
subsequent argument that it only has to be less than $1/2$). Let
\[z= A_1 \cdot \ldots \cdot A_s U_1 \cdot \ldots \cdot U_t\in \mathsf{Z}(B)\]
with $A_i, U_j\in \mathcal{A}(G_0)$, and $(e+dk)\mid A_i$ and
$(e+dk)\nmid U_j$ for all $i,j$. We proceed to show that
$\mathsf{v}_{e+dk}(A_i)= 1$ for each $i$, i.e., $s=f$. Clearly,
$\mathsf{v}_{(-1)1}(z)\le |z|$, and thus we have
\[
\begin{split}
\mathsf{v}_{-1}(\pi(A_1 \cdot \ldots \cdot A_s))& \ge d(u+fk)-|z| \\
& \ge d(u + fk) - (f+ u+ fk + c(d-1)k) \\
& = (f-2)(e+dk) + 2(e+dk) - ( f+ u+fk + c(d-1)k ) \\
& \ge  (f-2)(e+dk) + dk - ( d/2 + d + dk/2 + cdk) \\
& >  (f-2)(e+dk) + d(k - 3/2 - k/2 - ck).
\end{split}
\]
Since $c=1/6$ and $k\ge 10$, we have $k(1/2 - c)- 3/2 \ge 1$. So
we have
\be\label{ladyloo}\mathsf{v}_{-1}(\pi(A_1 \cdot \ldots \cdot A_s))\ge (f-2)(e+dk) + d.\ee
If $s\le  f-1$, then, since
$\mathsf{v}_{-1}(A_i)\le e+dk$ for each $i$, we conclude from \eqref{ladyloo} that
$\mathsf{v}_{-1}(A_i)\ge d$ for each $i$, implying (since $\supp(A_i^-)\subset \{-1,-d\}$) that
$\mathsf{v}_{e+dk}(A_i)= 1$ for each $i$, contradicting $s\leq f-1$. Thus
$s=f$. We have  $U_j \in \{(-1)1, ((-d)1^d)\}$ for each $j$. Thus
\[
z=A_1 \cdot \ldots \cdot A_f ((-1)1)^a ((-d)1^d)^b
\]
where $a = d(u+fk) - \mathsf{v}_{-1}(\pi(A_1 \cdot \ldots \cdot
A_f)) $ and $b = u+fk - \mathsf{v}_{-d}(\pi(A_1 \cdot \ldots \cdot
A_f))$. We have
\[|z|= f + (u+fk)(d+1)- (\mathsf{v}_{-1}(\pi(A_1 \cdot \ldots \cdot A_f)) + \mathsf{v}_{-d}(\pi(A_1 \cdot \ldots \cdot A_f)))\]
and, since
\[ d\cdot \mathsf{v}_{-d}(\pi(A_1 \cdot \ldots \cdot A_f)) + \mathsf{v}_{-1}(\pi(A_1 \cdot \ldots \cdot A_f))= (u+fk)d,\]
this implies
\[|z|= f + u + fk + (d-1)\mathsf{v}_{-d}(\pi(A_1 \cdot \ldots \cdot A_f)),\]
establishing $|z|-|z_2| \in (d-1)\mathbb{N}_0$.
\end{proof}

\begin{lemma} \label{STSL_lem2}
Let $d \in \mathbb{N}$, $e \in [1,d-1]$ with
$\gcd(e,d)=1$ and $G_0 \subset \Z$. If $\{-d,-1,1\}\subset G_0$,
$G_0^+ \cap (e+d\mathbb{Z})$ is infinite and $G_0^+ \setminus
((e+d\mathbb{Z}) \cup d \mathbb{Z})$ is nonempty, then the Structure
Theorem does not hold for $\mathcal{B}(G_0)$.
\end{lemma}

\begin{proof}
We may assume $d \ge 3$, as the hypotheses are null otherwise.
Since $G_0^+ \setminus
((e+d\mathbb{Z}) \cup d \mathbb{Z})$ is nonempty, let $f \in [1,d-1]\setminus \{e\}$ and $\ell \in \mathbb{N}_0$ be such
that $f+d\ell\in G_0^+$. Since $\{-d,-1,1\}\subset G_0$,
$G_0^+ \cap (e+d\mathbb{Z})$ is infinite, let $k \in \mathbb{N}$ be such that $e+dk \in
G_0^+$ and $e+dk \ge f + d \ell$. Since $\gcd(e,d)=1$, let $x \in [1,d-1]$ be the integer such that
$f+xe\in d \mathbb{Z}$, say $f+xe=ud$. Since  $f\in [1,d-1]\setminus \{e\}$, we have $x \neq d-1$ and $u\leq d-1$.

We proceed similarly to Lemma \ref{STSL_lem2}. We consider the
following element of $\mathcal{B}(G_0)$:
\[B = (f+d\ell)(e+dk)^x (-d)^{u+xk+\ell}(-1)^{d(u+xk+\ell)}1^{d(u+xk+\ell)}.\]
Again, we first consider two specific factorizations of $B$, namely
\[z_1 = ((f+d\ell)(e+dk)^x (-d)^{u+xk+\ell})\cdot ((-1)1)^{d(u+xk+\ell)}\]
and
\[z_2 = ((f+d\ell)(-1)^{f+d\ell})\cdot((e+dk)(-1)^{e+dk})^x\cdot ((-d)1^d)^{u+xk+\ell}.\]
The respective lengths of these factorizations are $1+ d(u+xk+\ell)$
and $1+x+ (u+xk+\ell)$. Thus, $|z_1|-|z_2| \notin (d-1)\mathbb{Z}$.

As in Lemma \ref{STSL_lem1}, we show that there exists a positive $c$, now depending on $d$ (but
not on $k$), such that, for each
 $z \in \mathsf{Z}(B)$ with
\[|z| \le |z_2| + c(d-1)k,\]
we have
\[|z| - |z_2| \in (d-1)\mathbb{N}_0,\]
which again completes the proof  by Lemma \ref{STSL_lemt}.
We set $c= 1/(d-1)$ (this choice is not optimal). Let
\[z= A_1 \cdot \ldots \cdot A_s((-1)1)^a ((-d)1^d)^b\]
where $A_i \notin \{(-1)1, (-d)1^d\}$. We proceed to show that $|A_i^+|=1$ for each $i$. From the definition of $B$, we have $s \le x+1$. Again,
$\mathsf{v}_{(-1)1}(z)\le |z|$, and thus
\[
\begin{split}
\mathsf{v}_{-1}(\pi(A_1 \cdot \ldots \cdot A_s))& \ge d(u+xk+\ell)-|z|\\
&  \ge d(u+xk+\ell) - (1+x+ (u+xk+\ell) + c(d-1)k) \\
& = (x-1)(e+dk) + (f+d\ell) + (e+dk)  - ( 1 + x + (u+xk+\ell) + c(d-1)k) \\
& \ge (x-1)(e+dk) + (f+d\ell) + (e+dk)  - (d-1 + (d-1+(d-2)k+\ell) + c(d-1)k)\\
& \ge (x-1)(e+dk)+ d + 2k  - 3d - c(d-1)k.
\end{split}\]
Since $c= 1/(d-1)$, we have, for $k \ge 3d$,
\[\mathsf{v}_{-1}(\pi(A_1 \cdot \ldots \cdot A_s))\ge (x-1)(e+dk)+ d.\]
If $s=x+1$, the claim is obvious. Thus, assume $s\le x$. Since
$\mathsf{v}_{-1}(A_i)\le e+dk$ for each $i$ (recall that $e+dk \ge
f+d\ell$), we get that $\mathsf{v}_{-1}(A_i)\ge d$ for each $i$,
establishing the claim (since $\supp(A_i^-)\subset \{-1,-d\}$).

Thus
\[z=A_1 \cdot \ldots \cdot A_s ((-1)1)^a ((-d)1^d)^b,\]
where $a =  d(u+xk+\ell) - \mathsf{v}_{-1}(\pi(A_1 \cdot \ldots
\cdot A_s)) $ and $b = (u+xk+\ell) - \mathsf{v}_{-d}(\pi(A_1 \cdot
\ldots \cdot A_s))$. We have
\[|z|= s +  (d+1)(u+xk+\ell) - (\mathsf{v}_{-1}(\pi(A_1\cdot \ldots \cdot A_s)) + \mathsf{v}_{-d}(\pi(A_1\cdot \ldots \cdot A_s))).\]
We note that if $f+\ell d\neq 1$, then $s=1+x$, and if $f+\ell d=
1$, then $s=x$. Moreover, if the former holds true, then
\[ d\cdot  \mathsf{v}_{-d}(\pi(A_1\cdot \ldots \cdot A_f)) + \mathsf{v}_{-1}(\pi(A_1\cdot \ldots \cdot A_f))= d(u+xk+\ell),\]
whereas if the latter holds true, then
\[ d \cdot \mathsf{v}_{-d}(\pi(A_1\cdot \ldots \cdot A_f)) + \mathsf{v}_{-1}(\pi(A_1\cdot \ldots \cdot A_f))= d(u+xk+\ell)-1.\]
In both cases, this implies
\[|z|= 1+x +   (u+xk+\ell) +(d-1) \mathsf{v}_{-d}(\pi(A_1\cdot \ldots \cdot A_s)).\]
establishing $|z|-|z_2| \in (d-1)\mathbb{N}_0$, as claimed.
\end{proof}

\begin{proposition} \label{STSL_prop}
Let $\{-1,1\} \subset G_0 \subset \mathbb{Z}$ with $G_0^-$ finite
such that the Structure Theorem holds for $\mathcal B (G_0)$. For
each $-d\in G_0^-$, at least one of the following statements
holds{\rm \,:}
\begin{enumerate}
  \item[(a)] $|G_0^+ \setminus d \mathbb{Z}| < \infty$.
  \item[(b)] $G_0^+ \setminus d \mathbb{Z} \subset 1 + d \mathbb{Z}$.
\end{enumerate}
\end{proposition}

\begin{proof}
The claim is trivial for $d\le 2$. Suppose $d \ge 3$. Let $E\subset
[0,d-1]$ be such that $G_0^{+}\cap (e+d \mathbb{Z})$ is infinite for each
$e \in E$. If there exists some $e\in E \setminus \{0\}$ with $\gcd(e,d)> 1$,  Lemma \ref{STSL_lem1}
yields a contradiction. Thus, $\gcd(e,d)=1$ for each $e \in E
\setminus \{0\}$. By Lemma \ref{STSL_lem2} we get that if
$\gcd(e,d)=1$, then $e=1$ (note that $1\in G_0^+$), and moreover, in
this case we have $G_0^+ \subset ((1+ d \mathbb{Z}) \cup
d\mathbb{Z})$.
\end{proof}

Now, we show that the Structure Theorem indeed holds for  monoids of
zero-sum sequences over  sets of the form considered in Theorem
\ref{STSL_thm} not covered by the above results. Moreover, we
investigate the finiteness of the successive distance for these
sets. Again, note that the set $F_0 \cup d\mathbb{N}$ in the result
below  does not fulfil condition $(a)$ of Theorem
\ref{main-theorem-II}, yet by Lemma \ref{lem_char} it can occur as
the subset of classes containing prime divisors of a Krull monoid,
even with respect to a divisor theory, showing that the conditions
$(b1)$, $(b2)$, and $(b3)$ do not imply $(a)$, not even combined.

\begin{proposition} \label{STSL_prop1}
Let $d \in \mathbb{N}_{\ge 2}$ and $F_0\subset \mathbb{Z}$ with
$F_0^- = \{-d,-1\}$.
\begin{enumerate}
  \item The Structure Theorem holds for $\mathcal{B}(F_0 \cup
d\mathbb{N})$ if and only if it holds for $\mathcal{B}(F_0\cup \{d\})$.
More precisely, for each $L\in \mathcal{L}(F_0 \cup
d\mathbb{N})$, there exists some $y \in \mathbb{N}_0$ such that $-y
+ L \in \mathcal{L}(F_0 \cup \{d\})$.
\item  $\delta(F_0 \cup
d\mathbb{N}) = \delta(F_0 \cup \{d\})$
\item There is a map $\psi: \mathcal{B}(F_0 \cup
d\mathbb{N})\rightarrow \mathcal{B}(F_0 \cup
\{d\})$ such that, for each $B \in \mathcal{B}(F_0 \cup
d\mathbb{N})$ and adjacent lengths $k$ and $l$ of $\mathsf L (B)$, we have $\dd(\mathsf{Z}_k(B),\mathsf{Z}_l(B))\leq \dd(\mathsf{Z}_k(\psi(B)),\mathsf{Z}_l(\psi(B)))$ with $k$ and $l$ adjacent lengths of $\mathsf L(\psi(B))$.
\end{enumerate}
In particular, if $F_0$ is finite, then the Structure Theorem holds
for $\mathcal{B}(F_0 \cup d\mathbb{N})$, and $\delta(F_0 \cup
d\mathbb{N})$ and $\mathsf c_{\mon}(F_0 \cup d\mathbb{N})$ are
finite.
\end{proposition}

\begin{proof}
Let $G_0=F_0 \cup d \mathbb{N}$ and $G_1= F_0 \cup \{d\}$.

\smallskip
1. Since $G_1 \subset G_0$, one implication is clear and
it remains to show that if the Structure Theorem holds for
$\mathcal{B}(G_1)$, then it holds for $\mathcal{B}(G_0)$. Indeed,
the more precise assertion we establish shows that the Structure
Theorem holds with the same bound and the same set of differences.

Let $\psi \colon \mathcal{F}(G_0) \to \mathcal{F}(G_1)$ denote the
monoid homomorphism defined via $\psi(g)=g$ for $g \notin d
\mathbb{N}$ and $\psi(kd)=d^k$ for $kd\in d\mathbb{N}$. We note that
$\sigma(S)= \sigma(\psi(S))$ for each $S \in \mathcal{F}(G_0)$; thus
$\psi$ yields a homomorphism, and indeed an epimorphism, from
$\mathcal{B}(G_0)$ to $\mathcal{B}(G_1)$.

Moreover, we observe that if $A \in \mathcal{A}(G_0)$ with $kd \mid
A$, for some $k \in \mathbb{N}$, then $A^+ = kd$. This implies that,
for such an atom, $\psi(A)=d^{k}(-1)^{d\ell}(-d)^{k - \ell}$ and
$(d(-1)^d)^{\ell} \cdot (d(-d))^{k -\ell}\in \mathsf{Z}(\psi(A))$ is
the unique factorization of $\psi(A)$. We denote this factorization
by $\overline{\psi}(A)$ and we note that $|\overline{\psi}(A)| =
\sigma(A^+)/d$. Setting $\overline{\psi}(A) = A$ for each atom not
of this form, i.e., $A \in \mathcal{A}(G_0)$ with $\supp(A)\cap
d\mathbb{N}=\emptyset$, and extending this map to $\mathsf{Z}(G_0)$,
we get a homomorphism, indeed an epimorphism, $\overline{\psi}
\colon \mathsf{Z}(G_0) \to \mathsf{Z}(G_1)$.

Since $\pi(\overline{\psi}(z))=\psi(\pi(z))$, we see that
$\overline{\psi}(\mathsf{Z}(B))\subset \mathsf{Z}(\psi(B))$ for each
$B \in \mathcal{B}(G_0)$. Moreover, for $B \in \mathcal{B}(G_0)$ and
$z \in \mathsf{Z}(B)$, we have, denoting $F=\prod_{g \in
d\mathbb{N}}g^{\mathsf{v}_g(B)}$, that $|\overline{\psi}(z)|= |z| +
(\sigma(F)/d - |F|)$. In particular, the value of $|\overline{\psi}(z)|- |z|$ is
the same for each $z \in \mathsf{Z}(B)$.

Thus, to establish our claim on sets of lengths, it suffices to show
that $\overline{\psi}(\mathsf{Z}(B))= \mathsf{Z}(\psi(B))$ for each
$B \in \mathcal{B}(G_0)$. Let $B \in \mathcal{B}(G_0)$ and again let
$F = \prod_{g \in d\mathbb{N}}g^{\mathsf{v}_g(B)}=\prod_{i=1}^{|F|}(k_id)$, where $k_i\in \mathbb N$. Let $z'\in
\mathsf{Z}(\psi(B))$. There exists a unique decomposition
$z'=z_1'z_2'$ such that $z_1'$ is minimal with $d^{\sigma(F)/d}\mid
\pi(z_1')$ (note that $\mathsf{v}_d(\psi(B))=\sigma(F)/d$). We have
$|z_1'|= \sigma(F)/d$. Write  $z_1'= \prod_{i=1}^{|F|}y'_i$ such that each factor $y'_i\in \ZZ(\psi(B))$ contains exactly $|y'_i|=k_i$ atoms.
Then letting $A_i = (k_id)d^{-k_i}\pi(y'_i)$, we have  $A_i \in
\mathcal{A}(G_0)$, and so $z=A_1\cdot \ldots \cdot A_{|F|}z_2'$ is a
factorization of $B$ with $\overline{\psi}(z) =
\psi(A_1)\cdot \ldots \cdot \psi(A_{|F|})z_2'=y'_1\cdot \ldots \cdot
y'_sz_2'= z'$, establishing our claim.

\smallskip
2. Since $\delta(G_1)\le \delta(G_0)$ is obvious, we only have to
show that $\delta(G_0)\le \delta(G_1)$. We show the following
slightly stronger result. Let $B \in \mathcal{B}(G_0)$ and $z \in
\mathsf{Z}(B)$. Then $\delta(z)\le \delta(\overline{\psi}(z))$.

Let $F$ and $z= z_1z_2$ be defined as above, and let
$z_1=\prod_{i=1}^{|F|}A_i$ and let  $A_i^+
= k_id$, where $k_i\in \mathbb{N}$. Moreover, let $z' = \overline{\psi}(z)$ and let $z' =
z_1'z_2'$ with $z_1' = \overline{\psi}(z_1)$ and
$z_2'=\overline{\psi}(z_2)=z_2$. Additionally, let $y_i'=
\overline{\psi}(A_i)$ for each $i$. Let $j \in \mathbb{Z}$ be such that
$|z|$ and $|z|+j$ are adjacent lengths of $\mathsf L (B)$. By the
already established result for sets of lengths, it follows that
$|\overline{\psi}(z)|$ and $|\overline{\psi}(z)|+j$ are adjacent
lengths of $\mathsf L (\psi(B))$. Thus, by definition, there exists
some factorization $x'\in \mathsf{Z}(\psi(B))$ with $|x'|=
|\overline{\psi}(z)|+j$ and $d(x', \overline{\psi}(z))\le
\delta(\overline{\psi}(z))$. Let $x'=x_1'x_2'$ with $x_1'$ minimal such that  $d^{\sigma(F)/d}\mid \pi(x_1')$. We note that
\be\label{kio1}\dd(z',x')= \dd(z_1',x_1') + \dd(z_2',x_2').\ee Thus, by re-indexing appropriately, we find a  \be\label{kio2}t\leq \dd(z_1',x_1')\ee such that $\prod_{i=t+1}^{|F|}y_i'\mid
x_1'$.

Let $x_1''= x'_1\left(\prod_{i=t+1}^{|F|}y_i'\right)^{-1}$.
As we argued at the end of part 1, there exists, for $i\leq t$, factorizations
 $y''_i\in \ZZ(\psi(B))$, each containing exactly $|y'_i|=k_i$ atoms, such that  $\prod_{i=1}^t y_i''=x_1''$. For $i \leq t$,
let $A_i'= d^{-k_i}(k_id)\pi(y_i'')$, and for $i\in [t+1,|F|]$, let $A'_i=A_i$. Then, with $x_1=\prod_{i=1}^{|F|}A'_i$ and $x_2=x_2'$, we
have that  $x=x_1x_2$ is a factorization of $B$, and since
$\overline{\psi}(x) = x_1''(\prod_{i=t+1}^{|F|}y_i')x_2'=x_1'x_2'$, we get that $|x|-|z|=|\overline{\psi}(x)|-|\overline{\psi}(z)|=|x'|-|\overline{\psi}(z)|= j$. Finally, using \eqref{kio1} and \eqref{kio2}, we have $$\dd(z,x)\le \dd(z_1,x_1)+\dd(z_2,x_2)   \le t+ \dd(z_2,x_2)\le
\dd(z_1',x_1') + \dd(z_2',x_2')= \dd(z',x'),$$ establishing the
claim.

\smallskip

3. The argument is just a variation on the proof of parts 1 and 2.

\smallskip

We now address the additional statements. Suppose that $F_0$ is
finite. By Proposition \ref{nice_prop}, we know that the Structure
Theorem holds for $\mathcal{B}(F_0 \cup \{d\})$ and that
$\delta(F_0\cup \{d\})$ is finite. Thus, by parts 1 and 2, we get
that the Structure Theorem holds for $\mathcal{B}(F_0 \cup  d
\mathbb{N})$ and that $\delta(F_0\cup d \mathbb{N})$ is finite.
Since $\delta(F_0\cup d \mathbb{N})$ is finite, Proposition
\ref{6.4} implies that $\mathsf c_{\mon}(F_0 \cup d \mathbb{N})$ is
finite.
\end{proof}

The systems of sets of lengths of  $\mathcal{B}(F_0 \cup
d\mathbb{N})$ and $\mathcal{B}(F_0\cup \{d\})$ are very closely
related, but they are different in general. For finite $F_0$, the
elasticity of $\mathcal{B}(F_0\cup \{d\})$ is accepted (Proposition
\ref{nice_prop}), yet we see in Corollary \ref{STSL_cor} that this
is in general not the case for $\mathcal{B}(F_0 \cup d\mathbb{N})$.

\begin{proposition} \label{STSL_prop2}
Let $d \in \mathbb{N}_{\ge 2}$ and $G_0 = \{-d,
-1\} \cup (1 + d\mathbb{N}_0) \cup d \mathbb{N}_0$.
\begin{enumerate}
\item The Structure Theorem holds for $\mathcal B (G_0)$. More precisely, each $L \in
      \mathcal{L}(G_0)$ is an arithmetical progression with difference
      $d-1$.

\smallskip
\item For each $B \in \mathcal{B}(G_0)$ and adjacent lengths $k$ and $l$ of $\mathsf L (B)$,
      we have $\dd(\mathsf{Z}_k(B),\mathsf{Z}_l(B))=d+1$.

\smallskip
\item $\delta(G_0)= \infty$.

\smallskip
\item $\mathsf c_{\mon}(G_0) = d+1$.
\end{enumerate}
\end{proposition}

\begin{proof}
Before we start the argument for the individual parts, we start with
some general remarks. We begin by investigating $\mathcal{A}(G_0)$. Let $A \in
\mathcal{A}(G_0)$. If $kd \mid A$ for some $k \in \mathbb{N}_0$,
then $A=(kd)(-1)^{dl}(-d)^{k-l}$ for some $l \in [0,k]$. In
particular,  we thus have two atoms containing $d$, namely $U_1=
d(-1)^d$  and  $U_d = d(-d)$. Suppose $\supp(A) \cap d \mathbb{N}_0
= \emptyset$. Then $A^+ = \prod_{i=1}^{|A^+|}(1+k_id)$ with $k_i\in
\mathbb{N}_0$. It follows that $|A^+|\in \{1,d\}$. Moreover, if
$|A^+|=d$, then $-1 \nmid A$ and therefore
$A=A^+(-d)^{\sigma(A^+)/d}$. Thus, either $|A^+|=1$ or else  $|A^+|=d$ and
$A=A^+(-d)^{\sigma(A^+)/d}$. Conversely, each zero-sum sequence
$B\in \mathcal{B}(G_0\setminus \{0\})$ with  $B^+ =\prod_{i=1}^d
(1+k_id)$, $k_i \in \mathbb{N}_0$ and $-1 \notin \supp(B^-)$ is
an atom.

Let $B \in \mathcal{B}(G_0 \setminus \{0\})$ and let $z \in
\mathsf{Z}(B)$. In view of the considerations just made, there
exists a unique decomposition $z=z_1z_d$ such that for each $A\mid
z_1$ we have $|A^+|=1$ and for each $A \mid z_d$ we have $|A^+|=d$.
We denote $|z_d|$ by $t_d(z)$. Since $|B^+|=|z_1|+d |z_d|$, it
follows that
\begin{equation}
\label{STSL_eq2} |z|=|z_1|+|z_d|= |B^+| - (d-1)|z_d| = |B^+| -
(d-1)t_d(z),
\end{equation}
i.e., $|z|$ is determined by $B^+$ and $t_d(z)$.

By Proposition \ref{STSL_prop1} and since $0$ is a prime, it
suffices to consider the set $G_1 = \{-d, -1\} \cup (1 +
d\mathbb{N}_0) \cup \{d\}$ for the proof of parts 1 and 3.

\bigskip
1. Let $B \in \mathcal{B}(G_1)$. Let $z \in \mathsf{Z}(B)$ and let
$z=z_1z_d$ be defined as above. We observe that, since
$\mathsf{v}_{-1}(A)\ge 1$ for each $A$ that neither fulfils
$|A^+|=d$ nor equals $U_d$, we have
\begin{equation}
\label{STSL_eq} t_d(z)\ge \frac{\bigl(|B^+| - \mathsf{v}_{d}(B)\bigr) -
\mathsf{v}_{-1}(B)}{d}
\end{equation}
By \eqref{STSL_eq2}, we get that $\mathsf{L}(B)$ is
contained in an arithmetical progression with difference $(d-1)$.
In view of this, it suffices to establish the following claim.

\noindent \textbf{Claim 1:} If $|z|< \max \mathsf{L}(B)$, then there
exists some $z'\in \mathsf{Z}(B)$ with $|z'|= |z|+(d-1)$ and $\dd(z,z')=d+1$; in particular, $t_d(z')= t_d(z)-1$. Moreover, $\dd(z,z')\leq d+1$.

To prove this, we first investigate the case $|z|= \max
\mathsf{L}(B)$.

\noindent \textbf{Claim 2:} If $t_d(z)=0$ or $\mathsf{v}_{-1}(A)\le
1$ for each $A \mid z$, then $|z|= \max \mathsf{L}(B)$.

\noindent \emph{Proof of Claim 2.} If $t_d(z)=0$, the claim is clear
by \eqref{STSL_eq2}. Thus, assume $\mathsf{v}_{-1}(A)\le 1$ for each
$A \mid z$. In view of the characterization of atoms, it follows
that $z_1=z_1'U_d^{\mathsf{v}_d(B)}$ and $\mathsf{v}_{-1}(A)=1$ for
each atom $A \mid z_1'$. In particular, we have $|z_1|=
\mathsf{v}_{-1}(B)+ \mathsf{v}_{d}(B)$. In view of $d\cdot t_d(z)= |B^+|-
|z_1|$, this implies
\[ t_d(z)= \frac{\bigl(|B^+| - \mathsf{v}_{d}(B)\bigr) - \mathsf{v}_{-1}(B)}{d}.\]
Thus equality holds in \eqref{STSL_eq}, which by \eqref{STSL_eq2}
implies that $|z|$ is maximal.

\noindent \emph{Proof of Claim 1.} Suppose $|z|< \max
\mathsf{L}(B)$. By Claim 2, we know that $t_d(z)> 0$ and that there
exists some atom $C \mid z$ such that $\mathsf{v}_{-1}(C)> 1$. In
view of the characterization of atoms given above, we have
$\mathsf{v}_{-1}(C) \ge d$ and $|C^+|=1$. Since $t_d(z)>0$, there
exists some atom $A_d\mid z$ with $|A_d^+| = d$. Let $z= A_dCz_0$.
We consider the zero-sum sequences $B_1=(-d)^{-1}A_d(-1)^d$ and
$B_2= (-1)^{-d}C(-d)$. Clearly, $\pi(B_1B_2z_0)=B$. We note that
$B_2$ is an atom as $|B_2^+|=1$. Yet, since $|B_1^+|=d$ but
$\mathsf{v}_{-1}(B_1)\ge 1$, we get that $B_1$ is not an atom; more
precisely, $\max\mathsf{L}(B_1)=d$. Thus, replacing the two atoms $A_d$ and $C$ in $z$ by the atom $B_2$ and any factorization of length $d$ of $B_1$ completes the claim.

\smallskip
2. By Proposition \ref{STSL_prop1}.3 and since $0$ is a prime, it suffices to consider $G_1$ for finding an upper bound on $\dd(\mathsf{Z}_k(B),\mathsf{Z}_l(B))$. Thus, by Claim 2, we get that $\dd(\mathsf{Z}_k(B),\mathsf{Z}_l(B))\le d+1$. The converse
inequality follows by \eqref{E:Dist} in view of Proposition \ref{STSL_prop1}.3 and part 1.

\smallskip
3. We consider $B=(1+kd)^{d}d^{1+kd}(-d)^{1+kd}(-1)^{d(1+
kd)}$. We note that $\mathsf{L}(B)= \{2+kd, 1+d+kd\}$ and
$z=\bigl((1+kd)^{d}(-d)^{1+kd}\bigr)\cdot (d(-1)^d)^{1+kd}$ is its only
factorization of length $2+kd$. The factorization
$z'=\bigl((1+kd)(-1)^{1+kd}\bigr)^d\cdot \bigl(d(-d)\bigr)^{1+kd}$ has length $1+d+kd$
and $\dd(z',z)= |z'|= 1+d+kd$, implying that $\delta(B)\ge 1+d+kd$,
and the claim follows by letting $k\rightarrow \infty$.

\smallskip
4. By part 2 and since $0$ is prime, it is sufficient to show that, for any two
factorizations $z,\,y\in \ZZ(G_0\setminus\{0\})$ with $\pi(z)=\pi(y)$, we have
that: if $|z|=|y|$, then $z$ and $y$ can be concatenated by a
monotone $2$-chain. Clearly, in this case monotone means that each
factorization in this chain has length $|z|$, i.e., we claim that
$z$ and $y$ can be concatenated by a $2$-chain in
$\mathsf{Z}_{|z|}(\pi(z))$. We proceed by induction on $|z|$. Let
$z,\,y\in \ZZ(G_0)$ with $\pi(z)=\pi(y)$ and suppose that $|z|=|y|$.
If $|z|=1$, the statement is trivial. Thus, assume $|z|\ge 2$ and
that the statement is true for factorizations of length at most
$|z|-1$. We make the following claim.

\noindent \textbf{Claim 3:} There exist $z',\,y'\in \ZZ(\pi(z))$
with $|z'|=|y'|=|z|$ such that
 $z$ and $z'$, as well as $y$ and $y'$, can be concatenated by a $2$-chain in $\mathsf{Z}_{|z|}(\pi(z))$ and $\gcd\{z',y'\}\neq 1$.

We assume this claim is true and complete the argument. Let $z'$ and
$y'$ be factorizations with the claimed properties and let  $U \in
\mathcal{A}(G_0)$ with  $U \mid \gcd \{z',y'\}$. We set $z'' =
U^{-1}z'$ and $y'' = U^{-1}y'$. By induction hypothesis, there
exists a $2$-chain $z''=z_0'',z_1'', \dots, z_s''=y''$ in
$\mathsf{Z}_{|z''|}(\pi(U^{-1}z'))$. We note that $U\cdot z_i'' \in
\mathsf{Z}_{|z|}(\pi(z))$ for each $i \in [0,s]$. Thus, $z'$ and
$y'$ can be concatenated by a $2$-chain in
$\mathsf{Z}_{|z|}(\pi(z))$. Combining these three chains, the result
follows.

\noindent \emph{Proof of Claim 3.} If $0 \mid z$, then $0 \mid y$
and the claim is trivial. Thus, assume $0 \nmid z$.

Let $z=z_1z_d$ and $y=y_1y_d$ be as defined at the beginning of the
proof and recall that $|z|= |y|$ is equivalent to $t_d(z)=t_d(y)$.

Before starting the actual argument, we make three subclaims.

\noindent \textbf{Claim 3.1:} Let $h \mid \pi(z_1)$ and $g \mid
\pi(z_d)$ with $g,h \in 1+ d\mathbb{N}_0$ and $h \le g$. Then
there exists a factorization $x$ of $\pi(z)$ such that, with
$x=x_1x_d$ as above, $\pi(x_1)^+= \pi(z_1)^+gh^{-1}$ and $\pi(x_d)^+ =
\pi(z_d)^+hg^{-1}$ and $\mathsf{d}(z,x)\le 2$; in particular,
$|x|=|z|$.

To see this, let $A_h \mid z_1$ and $A_g\mid z_d$ with $h \mid
A_h$ and $g \mid A_g$. We set $A_h'= hA_g g^{-1}(-d)^{-(g-h)/d}$ and
$A_g'= g A_h h^{-1}(-d)^{(g-h)/d}$. Note that this process is
well-defined and that $A_g'$ and $A_h'$ are atoms by the above characterization of
atoms. Let $x=zA_g'A_h'A_g^{-1}A_h^{-1}$. Noting that $x_1=
A_g'A_h^{-1}z_1$ and $x_d= A_h'A_g^{-1}z_d$, the claim is
established.

\noindent \textbf{Claim 3.2:} Suppose that $t_d(z)=0$. Then $z$ and
$y$ can be concatenated by a $2$-chain in
$\mathsf{Z}_{|z|}(\pi(z))$.

Informally, each atom in $z$ and $y$ contains exactly one positive element, hence distinct atoms containing the same positive element
can only differ in the negative part. Successively exchanging $(-1)^d$ for $-d$ and vice versa, for suitable pairs of atoms, we can construct such a chain.

To give a formal argument, we use the independent material of Section 7 which follows. Note that, in this case, $|z|=|y|=|\pi(z)^+|$ and $\mathcal A(\mathcal E(G_P^-))=\{(-d,-d), (-1,-1), ((-1)^d,-d), (-d,(-1)^d)\}$. Thus $G'\cong \Z$ with $G_0'=\{0,1,-1\}$, where $G'$ and $G_0'$ are as defined before Theorem \ref{prop-chain}, whence $\mathsf D(\mathcal S(G_P^-),\mathcal E(G_P^-))=2$ by \eqref{teent}. Hence Theorem \ref{prop-chain}
shows that there is a $2$-chain concatenating $z$ and $y$.

\noindent \textbf{Claim 3.3:} Suppose that $t_d(z)=|z|$. Then $z$
and $y$ can be concatenated by a $2$-chain in
$\mathsf{Z}_{|z|}(\pi(z))$.

Informally, since in this case $\supp(\pi(z))=\{-d\}$, we can apply an argument similar to the one in Claim 3.1, without additional condition on the relative size of $g$ and $h$.

To get a formal argument, note that in this case $\pi(z) \in \mathcal{B}(G_0 \setminus \{-1\})$. By Lemma
\ref{transfer-to-finite}, we get that the block monoid associated to
$\mathcal{B}(G_0 \setminus \{-1\})$ is $\mathcal{B}(\{0 + d
\mathbb{Z},1 + d\mathbb{Z}\})\subset \mathcal B(\Z/d\Z)$. However, $\mathcal{B}(\{0 + d
\mathbb{Z},1 + d\mathbb{Z}\})$ is factorial, and thus its catenary degree is $0$; also note that the former monoid is
thus half-factorial. Since the catenary degree in the fibers of the
block homomorphism is $2$ (see Lemma \ref{3.3}), the claim follows.

Now, we give the actual proof of Claim 3. In view of Claim 3.2, we
may assume that $t_d(z)> 0$. Hence, let $S \mid \pi(z)$ be a subsequence
with $\supp(S) \subset 1 +d \mathbb{N}_0$ and $|S|=d$. Moreover,
assume that $\sigma(S)$ is minimal among all such subsequences of
$\pi(z)$. We assert that there exists some $x' \in
\mathsf{Z}_{|z|}(\pi(z))$ such that $S \mid \pi(x_d')$ and $z$ and
$x'$ can be concatenated by a $2$-chain in
$\mathsf{Z}_{|z|}(\pi(z))$. Let $x' \in \mathsf{Z}_{|z|}(\pi(z))$ be
a factorization such that  $z$ and $x'$ can be concatenated by a
$2$-chain in $\mathsf{Z}_{|z|}(\pi(z))$  and such that $S'=
\gcd\{\pi(x_d'), S\}$ is maximal. We show that $S'=S$. Assume to the
contrary that $S'\neq S$. Let $h \mid  \pi(x_1')$ with $hS'\mid
S$. We observe that there exists some $g \mid S'^{-1}\pi(x_d')$ with $g \in 1  + d \mathbb{N}_0$ and $g \ge h$; otherwise, the
sequence $gh^{-1}S$ would contradict the minimality of $\sigma(S)$.

We apply Claim 3.1 to $x'$ (with these elements $g$ and $h$) and
denote the resulting factorization by $x''$. Since it can be
concatenated to $z$ by a $2$-chain in $\mathsf{Z}_{|z|}(\pi(z))$ and
yet $hS'\mid \gcd\{\pi(x_d''), S\}$, its existence contradicts the
maximality of $S'$ for $x'$. Thus $S'=S$.

Since $S \mid \pi(x_d')$, we have that $U= S(-d)^{\sigma(S)/d}\mid
\pi(x'_d)$. Let $z_d'\in \mathsf{Z}(\pi(x'_d))$ with $U \mid z_d'$.
Since $t_d(\pi(x'_d))= |x'_d|$, Claim 3.3 applied to $x'_d$ yields
that $x_d'$ and $z_d'$  can be concatenated by a $2$-chain in
$\mathsf{Z}_{|x_d'|}(\pi(x_d'))$. We set $z'= z_d'x_1'$ and
observe that $x'$ and $z'$, and thus $z$ and $z'$, can be
concatenated by a $2$-chain in  $\mathsf{Z}_{|z|}(\pi(z))$ and $U \mid z'$

In the same way, noting that $S$ depends only on $\pi(z)$ and not on
$z$, we get a factorization $y' \in \mathsf{Z}_{|z|}(\pi(z))$ with
$U \mid y'$  such that $y$ and $y'$ can be   concatenated by a
$2$-chain in  $\mathsf{Z}_{|z|}(\pi(z))$. Since $U \mid \gcd\{z',
y'\}$, the claim is established.
\end{proof}

\medskip
\begin{proof}[{\bf Proof of Theorem \ref{STSL_thm}}]
By Lemma \ref{3.3}, it suffices to consider $\mathcal B (G_P)$. The
case $d=1$ is trivial. Suppose $d \ge 2$. One direction is merely
Proposition \ref{STSL_prop}. The other one follows, for the first
type of set, by Proposition \ref{STSL_prop1}, and for the second type
of set, by Proposition \ref{STSL_prop2}.
\end{proof}

By \cite{An-Ch-Sm94c}, it is known that Krull monoids with infinite
cyclic class group can have finite, non-accepted elasticity. The
following result shows that, even if the Structure Theorem holds,
the elasticity is not necessarily accepted.

\begin{corollary} \label{STSL_cor}
Let all assumptions be as in Theorem \ref{STSL_thm}.  Suppose that
the Structure Theorem holds for $H$. Then exactly one of the
following two statements holds{\rm \,:}
\begin{enumerate}
  \item[(a)] $H$ is half-factorial or $G_P$ is finite.
  \item[(b)] $\rho(H)=d$ and the elasticity is not accepted.
\end{enumerate}
\end{corollary}

\begin{proof}
Half-factorial monoids obviously have accepted elasticity and
monoids with $G_P$ finite also have accepted elasticity (Proposition
\ref{nice_prop}). Thus, we assume that $H$ is not half-factorial and
that $G_P$ is infinite, and show that under these assumptions
$\rho(H)=d$ and the elasticity is not accepted. Note that since $H$
is not half-factorial, we have $d \ge 2$.

We recall that if $A \in \mathcal{A}(G_P)$ with $(-1)\mid A$,
then $|A^+|=1$ (as explained in the proof of Proposition \ref{STSL_prop2}).

Let $B \in \mathcal{B}(G_P)$. We show that $\rho(B)< d$.
Assume to the contrary $\rho(B)\ge d$.
That is, there exist $z,z'\in\mathsf{Z}(B)$ such that
$|z'|/|z| \ge d$. By Lemma \ref{Lambert}, we know that $|A^+|\le d$ for each $A\in \mathcal{A}(G_P)$.
Thus, we get $|z| \ge \mathsf{v}_0(z) + |B^+|/d$, whereas clearly $|z'| \le \mathsf{v}_0(z') + |B^+|$.

Consequently, we have $\rho(B)\le d$, and $\rho(B)=d$ is equivalent to the following:
$|A^+|=d$ for each atom $A \mid z$ and $|A'^+|=1$ for each atom $A'
\mid z'$. It follows that $\mathsf{v}_{-1}(B)=0$, i.e., $B \in
\mathcal{B}(G_P\setminus \{-1\})$. By \cite{An-Ch-Sm94c}, or Lemma
\ref{transfer-to-finite} and \cite[Proposition 6.3.1]{Ge-HK06a}, we
get that $\rho (\mathcal{B}(G_P\setminus \{-1\}))\le
\rho(\mathbb{Z}/d \mathbb{Z}) = d/2 < d$, a contradiction.

It remains to show that $\rho(G_P) \ge d$. We may assume that $0
\notin G_P$. We note the existence of the two atoms $1(-1)$ and
$1^d(-d)$ in $\mathcal{A}(G_P)$. Thus, $1$ and $1^d$ are elements of
$\mathcal{A}(G_P)^+$. Thus, $\rho^{\kappa}(1^d)\ge d$, and the claim
follows by Lemma \ref{rel_lem_eq}.
\end{proof}

Our proofs that the Structure Theorem does not hold rely on the
existence of a single exceptional factorization, yet  the following
example illustrates that sets of lengths can deviate by more than a
single element (or a globally bounded number of elements) from being
an AAMP.

\begin{example}
Let $d, k,l \in \mathbb{N}$ and $e \in [1,d-1]$, and set $B = (e+ k d)( -e +
\ell d)1^{(k+\ell)d}(-1)^{(k+\ell)d}(-d)^{k +\ell}$.
Then
\[ \begin{split} \mathsf{L}(B)=  \{1+ k+ \ell +  (k+\ell)(d-1)\} &  \cup \{1 + e +k+ \ell + i(d-1) \mid i \in [ k, k+\ell-1]\} \\
&  \cup   \{ 2 -e + k + \ell+   i(d-1)   \mid i \in [\ell, \ell+ k]\} \\
& \cup \{ 2 + k + \ell + i(d-1) \mid i \in [0, k +\ell -1]  \}.
 \end{split} \]
\end{example}

\section{Chains of factorizations} \label{7}

In a  large class of monoids and domains satisfying natural
(algebraic) finiteness conditions,  the catenary degree is finite
(see \cite{Ge-HK06a} for an overview and \cite{C-G-L-P-R06,
Ge-Ha08b, C-G-L09, Ka10a} for some recent work). However, the
understanding of the structure of the concatenating chains is still
very limited. Whereas, on the one hand, the finiteness of the
monotone catenary degree  is a rare phenomenon (inside the class of
objects having finite catenary degree), the following two positive
phenomena have been observed. First, in a large class of monoids, all
problems with the monotonicity of concatenating chains occur only at
the beginning and the end of concatenating chains (\cite[Theorem
1.1]{Fo-Ge05}, \cite[Theorem 3.1]{Fo-Ha06b}). Second, in various
settings, there is a large subset consisting of `big' elements
having extremely nice concatenating chains (see \cite[Theorem
4.3]{Ge97d}, \cite[Theorems 7.6.9 and 9.4.11]{Ge-HK06a}).

Let $H$ be a Krull monoid with infinite cyclic class group and  $G_P \subset G$ as always. By Theorem
\ref{main-theorem-II}, it suffices to consider the situation where
$G_P^+$ is infinite and $2 \le |G_P^-| < \infty$. Our first result
points out that, in general, the monotone catenary degree is infinite.
In contrast to this, the main result (Corollary \ref{final-cor})
shows that there is a constant $M^*$ such that, for a large class of
elements $a$, any two factorizations $z$ and $y$ of $a$ with  $y$
having maximal length can be concatenated by a monotone
$M^*$-chain of factorizations and thus, for those factorizations $z$ and $y$ of $a$ neither of which need be of maximal length, there is an $M^*$-chain between $z$ and $y$ which `changes direction' at most once.

\begin{proposition} \label{7.1}
Let $H$ be a Krull monoid and $\varphi \colon H\to \mathcal{F}(P)$ a
cofinal divisor homomorphism into a free monoid such that the class
group $G= \mathcal{C}(\varphi)$ is an  infinite cyclic group that we
identify with $\mathbb Z$. Let $G_P \subset G$ denote the set of
classes containing prime divisors. Suppose that $-d_1,-d_2,d_1d_2\in
G_P$, where $3\leq d_1<d_2$, $\gcd(d_1,d_2)=1$ and $d_1-1\nmid
d_2-1$, and that $G_P$ contains infinitely many positive integers
congruent to $d_1+d_2$ modulo $d_1d_2$. Let $d=\gcd(d_1-1,d_2-1)$.
Then, for every $M,\,N\geq 0$, there exists $a \in H$ and $z,\,z'\in
\ZZ(a)$ such that \ber \label{z'-size}&&|z'|=|z|+d\leq
|z|+d_1-2,\\&&\label{notnearedges} |z|\in [\min \mathsf L (a) + N,
\, \max \mathsf L (a) - N], \mbox{ and}\\
\label{distance-thingy}&&\mathsf d \Bigl(z ,
\bigcup_{i=1}^{|Z|+d_1-2}\ZZ_i(a)\setminus \{z \} \Bigr) > M.\eer In
particular, $\mathsf c_{\mon} (H) = \infty$ and $\delta(H)=\infty$
\end{proposition}

\begin{proof}
That $\mathsf c_{\mon} (H)=\delta(H) = \infty$ follows from
\eqref{z'-size} and \eqref{distance-thingy}, so we need only show
\eqref{z'-size}, \eqref{notnearedges} and \eqref{distance-thingy}
hold. By Lemma \ref{3.3}, it suffices to prove the assertions for
$\mathcal B (G_P)$. We may also w.l.o.g. assume $$N\geq d_2-1\und
M\geq d_1,$$ as the theorem holding for large values of $M$ and $N$
implies it holding for all smaller values.

\smallskip
In view of the hypotheses, there exists $L\in G_P$ with \ber
\label{that}L&>& d_2M\geq d_1d_2,\\ \label{L-congruences} L\equiv
d_1\mod d_2&\und& L\equiv d_2 \mod d_1.\eer Let $B\in
\B(\{d_1d_2,-d_1,-d_2,L\})\subset \B(G_P)$ be the sequence
\[
B = L^{2d_1d_2N}(-d_2)^{2d_1LN}(-d_1)^{2d_2LN}(d_1d_2)^{2LN} \,.
\]
 Let $$A_1=L^{d_1}(-d_1)^L \und
A_{2}=L^{d_2}(-d_2)^L.$$ Since $\gcd(d_1,d_2)=1$, it follows, in
view of \eqref{L-congruences} and by reducing modulo $d_1$ and
$d_2$, respectively, that $A_1$ and $A_{2}$ are both atoms. Also
define
\[
B_1=(d_1d_2)(-d_1)^{d_2}\und B_{2}= (d_1d_2)(-d_2)^{d_1} \,,
\]
which, since they both contain exactly one positive integer, must
also be atoms. In view of \eqref{L-congruences}, define
\[
A_0=L(-d_2)^{\frac{L-d_1}{d_2}}(-d_1),
\]
which is an atom for the same reasons as those for the $B_i$.

Let  $z\in \mathsf Z (B)$ be given by
\[
z=A_1^{d_2N}  A_{2}^{d_1N}  B_1^{LN}  B_{2}^{LN} \,.
\]
Since $d=\gcd(d_1-1,d_2-1)$, it follows that there exists an integer
$l\in [1,d_2-1]$ such that $$l(d_2-d_1)\equiv -d\mod d_2-1.$$ Let
\be\label{lprime} l' = \frac{l(d_2-d_1)+d}{d_2-1}\in \mathbb{N}.\ee
Then, since $d=\gcd(d_1-1,d_2-1)\leq d_1-1$, it follows that $1\leq
l'\leq l\leq d_2-1$. Note that we have the
identities $$\pi(A_1^{d_2}B_2^L)=\pi(A_2^{d_1}B_1^L)\und
\pi(A_2B_1)=\pi(A_0^{d_2}B_2).$$ Thus, by considering the definition
of $z$ and recalling that $N\geq d_2-1\geq l\geq l'$, we see that
\[
z' = A_1^{d_2N-ld_2}  A_{2}^{d_1N+ld_1-l'}  A_0^{l'd_2}
B_1^{LN+lL-l'}  B_{2}^{LN-lL+l'}
\]
 is another factorization $z'\in
\ZZ(B)$ besides $z$.

Note that $|z'|-|z|=-l(d_2-d_1)+l'(d_2-1)=d$. Moreover, since
$d_1-1\nmid d_2-1$, $d_1<d_2$ and $\gcd(d_1-1,d_2-1)=d$, it follows
that $d<d_1-1$. Thus \eqref{z'-size} holds. Also, the factorizations
$$A_2^{2d_1N}  B_{1}^{2LN}\in \mathsf Z (B)\und
A_{1}^{2d_2N}  B_{2}^{2LN}\in \mathsf Z (B)$$ show that
\[
\min \mathsf L (B)+N \leq \min \mathsf L (B)+(d_2-d_1)N \leq |z|\leq
\max \mathsf L(B)-(d_2-d_1)N \leq \max \mathsf L(B)-N \,,
\]
whence \eqref{notnearedges} holds. It remains to establish
\eqref{distance-thingy}. We begin with the following claim.

\noindent
{\bf Claim 1:} If $A|B$ is an atom with $d_1d_2\in \supp(A)$, then $d_1d_2$ is the only positive element dividing $A$ and $\vp_{d_1d_2}(A)=1$.

Suppose instead that $a \t A (d_1d_2)^{-1}$ with $a\in
\{L,d_1d_2\}$. Then we must have $\vp_{-d_2}(A)<d_1$ and
$\vp_{-d_2}<d_1$, else $(d_1d_2)(-d_1)^{d_2}$ or
$(d_1d_2)(-d_2)^{d_1}$  would be a proper, nontrivial zero-sum
subsequence dividing $A$, contradicting that $A$ is an atom. But now
(in view of \eqref{that}) $$2d_1d_2>-\sigma (A^-)=\sigma (A^+)\geq
a+d_1d_2\geq \min\{L,d_1d_2\}+d_1d_2=2d_1d_2,$$ a contradiction. So
Claim 1 is established.

\bigskip

In view of Claim 1, we see that, in any factorization $y$ of $B$, there
will always be $2LN$ atoms $A$ having $A (d_1d_2)^{-1}$ consisting
entirely of negative terms. Thus the length of any factorization of
$B$ is determined entirely by the number of atoms containing an $L$.
Moreover, by considering sums modulo $d_i$, we find (in view of
\eqref{L-congruences} and $\gcd(d_1,d_2)=1$) that
$(d_1d_2)(-d_1)^{d_2}$ and $(d_1d_2)(-d_2)^{d_1}$ are the only atoms
dividing $B$ which contain $d_1d_2$. As a result, we in fact have
the factorization of $B$ completely determined by how the $2d_1d_2N$
terms equal to $L$ are factored (that is, if $y_L|y$ is the
subfactorization consisting of all atoms containing an $L$, then
$\pi(y_L^{-1}y)$ has a unique factorization, which will always have
length $2LN$). We continue with the next claim.

\noindent
{\bf Claim 2:} If $A|B$ is an atom with $L,-d_1,-d_2\in \supp(A)$, then $\vp_L(A)=1$.

Suppose instead that $L^2|A$. In view of \eqref{L-congruences} and \eqref{that}, both
$\frac{L-d_1}{d_2}$ and $\frac{L-d_2}{d_1}$ are positive integers.
Consequently,  we must have $\vp_{-d_1}(A)<\frac{L-d_2}{d_1}$ and
$\vp_{-d_2}<\frac{L-d_1}{d_2}$, else $$L(-d_1)^{(L-d_2)/d_1}(-d_2)
\; \mbox{ or }\; L(-d_2)^{(L-d_1)/d_2}(-d_1)$$  would be a proper,
nontrivial zero-sum subsequence dividing $A$, contradicting that $A$
is an atom. But now $$2L-d_1-d_2>-\sigma (A^-)=\sigma (A^+)\geq
2L,$$ a contradiction. So Claim 2 is established.

\bigskip

In view of \eqref{L-congruences}, $\gcd(d_1,d_2)=1$ and Claims 1 and
2, we see that if $A|B$ is an atom with $L\in \supp(A)$, then
either\begin{itemize}
\item[(a)] $A=A_1$ and $\vp_{-d_2}(A)=0$,
\item[(b)] $A=A_{2}$ and $\vp_{-d_1}(A)=0$, or
\item[(c)] $\vp_L(A)=1$ and $\vp_{d_1d_2}(A)=0$.\end{itemize}

Let $y\in \mathsf Z(B)$ be a factorization with $\dd(z,y)\leq M$ and
let $y_L|y$  and $z_L|z$ be the corresponding sub-factorizations
consisting of all atoms which contain an $L$. In view of the
definition of $z$, since $\dd(z,y)\leq M$ and $L>d_2M$ (by
\eqref{that}), and since $(d_1d_2)(-d_1)^{d_2}$ is the only atom
containing a $-d_1$ in $z_L^{-1}z$, it follows that
$$\vp_{-d_1}(\pi(y_L))\leq
\vp_{-d_1}(\pi(z_L))+Md_2=d_2NL+Md_2<d_2NL+L;$$ thus the
multiplicity $m_1$ of the atom $A_1$ in $y$ is at most $d_2N$ (since
each such atom $A_1$ requires $L$ terms equal to $-d_1$). Likewise,
$$\vp_{-d_2}(\pi(y_L))\leq \vp_{-d_2}(\pi(z_L))+Md_1=
d_1NL+Md_1<d_1NL+L,$$ whence the multiplicity $m_{2}$ of the
atom $A_{2}$ in $y$ is at most $d_1N$.

Let $m_0$ be the number of atoms dividing $y$ containing exactly one
term $L$. Hence, since all atoms containing an $L$ must be of one of
the three previously described forms, it follows that
\be\label{tist}d_1m_1+d_2m_{2}+m_0=\vp_{L}(B)=2d_1d_2N.\ee  Let
$m'_0$, $m'_1$ and $m'_2$ be analogously defined for $z$ instead of
$y$. Then $m'_0=0$, $m'_1=d_2N$ and $m'_2=d_1N$. In view of
\eqref{tist} and the comments after Claim 1, and since $m_1\leq
d_2N=m'_1$ and $m_{2}\leq d_1N=m'_2$, it follows that $$|y|=
|z|+(m'_1-m_1)(d_1-1)+(m'_2-m_2)(d_2-1)\geq |z|;$$ moreover, unless
$m_1=m'_1$ and $m_{2}=m'_2$, then $|y|\geq |z|+d_1-1$. On the other
hand, if $m_1=m'_1=d_2N$ and $m_{2}=m'_2=d_1N$, then $m_0=0$ (in
view of \eqref{tist}), whence $z_L=y_L$ (recalling that all atoms
containing an $L$ must be of one of the three previously described
forms), from which $z=y$ follows by the comments after the proof of
Claim 1. Consequently, we conclude that $\dd(z,y)\leq M$ implies
either $y=z$ or $|y|\geq |z|+d_1-1$, which establishes
\eqref{distance-thingy}, completing the proof.
\end{proof}

The following lemma helps describe when an atom can contain more than one positive term.

\begin{lemma}\label{breakapart-lemma}
Let $G_0\subset \Z$ be a condensed set such that  $G_0^-$ is
finite and nonempty. Let $M=|\min G_0|$, let $U\in \mathcal A(G_0)$ and let $R|U^-$ be
the subsequence consisting of all negative integers with
multiplicity at least $M-1$ in $U$. Suppose there is some $L\in \Sigma(U^+)\setminus \{\sigma(U^+)\}$ such that
\be\label{hyplem}|U^+|\geq 2,\quad L \geq (M-1)^2, \und
\sigma(U^+)\geq L+(M-1)^2.\ee Then the following statements hold:
\begin{enumerate}
\item There is some $a\in \supp(U)\cap G_0^-$ with $\vp_{a}(U)\geq M-1$,
      i.e., $R$ is nontrivial.

\smallskip
\item For any such $a\in \supp(R)$, we have $(-L+a\Z)\cap
\Sigma(U^-)=\emptyset$.

\smallskip
\item There exists a subsequence $R'|U^-$ with $R|R'$ such that
      $L\notin \langle \supp(R')\rangle=n\Z$ and $|{R'}^{-1}U^-|\leq n-2$; in particular, $\supp(R)\subset \supp(R')\subset n\Z$ does not generate $\Z$.
\end{enumerate}
\end{lemma}

\begin{proof}
1. Let $U_L|U^+$ be a proper subsequence with sum equal to $L$. Note
that $|G_0^-|\leq M$. Thus $\sigma(U^+)\geq L\geq(M-1)^2> (M-2)|G_0^-|$, whence the
pigeonhole principle implies that there is some $a\in \supp(U)\cap
G_0^-$ with $\vp_{a}(U)\geq M-1$.
\medskip

2. Let $a|U^-$ with $\vp_{a}(U)\geq M-1$ and let $\phi_a \colon \Z\rightarrow \Z/a\Z$ denote the natural
homomorphism. We say that a sequence $T$ is a zero-sum sequence (zero-sum free, resp.) modulo $a$ if  $\phi_a(T)\in \mathcal{F}(\Z/a\Z)$ has the respective property.
 Suppose $(-L+a\Z)\cap \Sigma(U^-)$ is nonempty and let $S$ be a
zero-sum free modulo $a$ subsequence  $S|U^-$ (possibly trivial)
with $\sigma(S)\equiv -L\mod a$. Note that any zero-sum free modulo
$a$ subsequence $T|U^-$ has length at most $\mathsf D(\Z/a\Z)-1=|a|-1$ \cite[Theorem 5.1.10]{Ge-HK06a},
and thus \be\label{zsf-bound}|\sigma(T)|\leq (|a|-1)\cdot |\min \bigl(
(\supp(U)\cap G_0^-)\setminus \{a\} \bigr)|\leq (M-1)^2\leq L;\ee in
particular, $|\sigma(S)|\leq (M-1)^2\leq L$.

Now factor $S^{-1}U^-=S_0S_1 \cdot \ldots \cdot
S_ta^{\vp_{a}(U^-)}$, where $S_0$ is zero-sum free modulo $a$
and each $S_i$, for $i\geq 1$, is an atom modulo $a$. In view of
$|\sigma(S_0)|\leq (M-1)^2$ (from \eqref{zsf-bound}) and the
hypothesis $\sigma(U^+)\geq L+(M-1)^2$, we have
\be\label{tlot}|\sigma(SS_1\cdot\ldots\cdot S_t a^{\vp_{a}(U^-)})|=|\sigma(S_0^{-1}U^-)|\geq L.\ee If $|\sigma(SS_1\cdot \ldots
\cdot S_t)|\leq L$, then it follows, in view of \eqref{tlot} and the definitions of $S$ and the $S_i$, that we can append on to $SS_1\cdot\ldots\cdot S_t$ a sufficient number of terms equal to $a$ so as to obtain a subsequence $B_L|S_0^{-1}U^-$
with $SS_1\cdot\ldots\cdot S_t|B_L$ and $\sigma(B_L)=-L$, and now $U_LB_L|U$ is a proper,
nontrivial zero-sum subsequence, contradicting that $U$ is an atom.
Therefore $|\sigma(SS_1\cdot \ldots \cdot S_t)|> L$, and let $t'<t$
be the maximal non-negative integer such that $|\sigma(SS_1\cdot
\ldots \cdot S_{t'})|\leq  L$, which exists in view of
$|\sigma(S)|\leq (M-1)^2\leq L$. By its maximality, we have
\be\label{luckstack}|\sigma(S_1\cdot \ldots \cdot S_{t'})|>
L-|\sigma(S)|-|\sigma(S_{t'+1})|\geq L-|\sigma(S)|-|a|M,\ee where the
second inequality follows by recalling that  $S_{t'+1}$ is an atom
modulo $a$ and thus has length at most $\mathsf D(\Z/a\Z)=|a|$. From the
definitions of all respective quantities, both the left and right
hand side of \eqref{luckstack} is divisible by $a$, whence
\[
|\sigma(S_1\cdot \ldots \cdot S_{t'})|\geq L-|\sigma(S)|-|a|(M-1) \,.
\]
But now we see, in view of $\vp_{a}(U)\geq M-1$ and the definition of $t'$, that we can append on to $SS_1\cdot\ldots\cdot S_{t'}$ a sufficient number of terms equal to $a$ so as to obtain a subsequence $B_L|S_0^{-1}U^-$ with $SS_1\cdot\ldots\cdot S_{t'}|B_L$ and
$\sigma(B_L)=-L$, once again contradicting that $U$ is an atom. So
we conclude that $(-L+a\Z)\cap \Sigma(U^-)$ is empty.

\medskip
3. In view of part 2, we see that \be\label{lizzard}-L\notin \langle
a\rangle +\Sigma(U^-).\ee Now, if $|a^{-\vp_a(U^-)}U^-|\leq |a|-2$,
then $\supp(R)=\{a\}$ (recall $|a|\leq M$ and $\vp_g(R)\geq M-1$ for
all $g\in \supp(R)$) and the final part of the lemma holds with
$R'=R$ in view of \eqref{lizzard}. Therefore we may assume
$y=|a^{-\vp_a(U^-)}U^-|\geq |a|-1$. Note that \eqref{lizzard}
implies that \[\phi_a(-L)\notin
\Sigma_y(\phi_a(a^{-\vp_a(U^-)}U^-)0^y)=\Sigma(\phi_a(U^-))\neq
\Z/a\Z.\] As a result, applying the Partition Theorem (see
\cite[Theorem 3]{Gr05b}) to $\phi_a(a^{-\vp_a(U^-)}U^-)0^y$, now
yields part 3 (to be more precise, we apply that result with
sequences $S = S' = \phi_a(a^{-\vp_a(U^-)}U^-)0^y$ and number of
summands $n=y$; also note that the resulting coset from the
Partition Theorem must be a subgroup in view of the high
multiplicity of $0$ and that $R|R'$ since $\vp_g(R)\geq M-1>|a|-2$
for all $g\in \supp(R)$).
\end{proof}

Before stating the next result, we need to first introduce some
notions. Let $G_0\subset \Z\setminus \{0\}$ be a condensed set
such that $G_0^-$ is finite and nonempty, and let $B\in \B(G_0)$. If
$z=A_1\cdot\ldots\cdot A_n\in \ZZ(B)$, with $A_i\in \mathcal
A(G_0)$, then we let
$$z^+=A_1^+\cdot\ldots\cdot A_n^+\in  \Fc(\mathcal
A(G_0)^+)$$ and $\ZZ(B)^+ = \{z^+ \mid z \in \ZZ(B) \}$. We can then define a partial order on
$\ZZ(B)^+$ by declaring, for $z^+,\, y^+ \in \ZZ(B)^+$, that $z^+ \le y^+$ when $z^+=A_1^+\cdot\ldots\cdot A_n^+\in
\ZZ(B)^+$, where $A_i\in \mathcal A(G_0)$,
\ber\nn y=(B_{1,1}\cdot
\ldots \cdot B_{1,k_1})\cdot (B_{2,1}\cdot \ldots \cdot
B_{2,k_2})\cdot \ldots \cdot (B_{n,1}\cdot\ldots \cdot
B_{n,k_n})\\ \;\mbox{ with }\;B_{j,i}\in \mathcal A(G_0)\;\mbox{
and }\;A_j^+ = B_{j,1}^+ \cdot \ldots \cdot B_{j,k_j}^+\; \mbox{ for
} \; j\in [1,n]\;\mbox{ and }\; i\in [1,k_j]. \nn \eer
We then define $\Upsilon(B)$ to be all those factorizations $z\in \ZZ(B)$
for which $z^+\in \ZZ(B)^+$ is maximal with respect to this partial
order.

Note that, if $z,\,y\in \ZZ(B)$ with $z^+\lneqq y^+$, then $|z|<|y|$.
Thus $\Upsilon(B)$ includes all factorizations $z\in \ZZ(B)$
of maximal length $|z|=\max \mathsf L(B)$, and  equality holds,
namely \be\label{Upsilon-equivalence}\Upsilon(B)=\{z\in \ZZ(B)\mid
|z|=|B^+|\},\ee when $\max \mathsf L(B)=|B^+|$.
If $H$ is a Krull monoid, $\varphi \colon H \to \mathcal F (P)$ a
cofinal divisor homomorphism and $a \in H$, then we define
\[
\Upsilon (a) = \{ z \in \mathsf Z (a) \mid \overline{\boldsymbol
\beta} (z) \in \Upsilon ( \boldsymbol \beta (a)) \} \,.
\]

For a pair of monoids $H\subset D$, we recall the definition of the
\emph{relative Davenport constant}, originally introduced in \cite
{Ge97c} and denoted $\mathsf D(H,D)$, which is the minimum $N\in
\mathbb N\cup \{\infty\}$ such that if $z\in \ZZ(D)=\mathcal
F(\mathcal A(D))$ with $\pi(z)\in H$, then there exists $z'|z$ with
$\pi(z')\in H$ and $|z'|\leq N$.

Next, we introduce two new monoids associated to $\mathcal F(G_0)$.
We assume that $\emptyset \neq G_0\subset \Z \setminus \{0\}$, yet here we do not assume that $G_0$ is condensed.
Consider the free monoid $\mathcal F(G_0)\times \mathcal F(G_0)$ and
let $$\mathcal E(G_0)=\{(S_1,S_2)\in \mathcal F(G_0)\times \mathcal
F(G_0)\mid \sigma(S_1)=\sigma(S_2)\}\subset \mathcal F(G_0)\times
\mathcal F(G_0)$$ the subset of pairs of sequences with equal sum and $$\mathcal S(G_0)=\{(S_1,S_2)\in \mathcal
F(G_0)\times \mathcal F(G_0)\mid S_1=S_2\}\subset \mathcal E(G_0)\subset
\mathcal F(G_0)\times \mathcal F(G_0)$$ the subset of symmetric pairs.
Note both $\mathcal E(G_0)$
and $\mathcal S(G_0)$ are monoids; furthermore, $\mathcal S(G_0)$ is
saturated  and cofinal in $\mathcal E(G_0)$, and $\mathcal E(G_0)$ is saturated
and cofinal in $\mathcal F(G_0)\times \mathcal F(G_0)$. Thus, if we let $G'$
denote the class group of the inclusion  $\mathcal
S(G_0)\hookrightarrow \mathcal E(G_0)$ and let $$G_0'=\{[u]\in G'\mid
u\in \mathcal A(\mathcal E(G_0))\}\subset G',$$ then \cite[Lemma
4.4]{Ge97c} shows that (recall that, due to the cofinality,
the definition of the class group in that paper is equivalent to the present one)
\be\label{teent}\mathsf D(\mathcal
S(G_0),\mathcal E(G_0))=\mathsf D(G_0').\ee Note that, if $(S_1,S_2)\in
\mathcal A(\mathcal E(G_0))$, then $S_1(-S_2)\in \mathcal A(G_0\cup
-G_0)$, whence $|S_1|+|S_2|\leq \mathsf D (G_0\cup -G_0)$; by
\cite[Theorem 3.4.2.1]{Ge-HK06a}, we know that, for a finite subset
$P$ of an abelian group, we have both $\mathsf D(P)$ and $\mathcal
A(P)$ finite. Consequently, if $G_0$ is finite, then $\mathsf
D(G_0\cup -G_0)$ is finite, whence $\mathcal A(\mathcal E(G_0))$ is
finite, which in turn implies $G_0'$, and hence also $\mathsf
D(G_0')$, is finite. Therefore, in view of \eqref{teent}, we conclude
that \be\label{relative-D-is-finite} \mathsf D(\mathcal
S(G_0),\mathcal E(G_0))<\infty\ee for $G_0$ finite.

\medskip
\begin{theorem}\label{prop-chain}
Let $H$ be a Krull monoid and $\varphi \colon H\to \mathcal{F}(P)$ a
cofinal divisor homomorphism into a free monoid such that the class
group $G= \mathcal{C}(\varphi)$ is an  infinite cyclic group that we
identify with $\mathbb Z$. Let $G_P \subset G$ denote the set of
classes containing prime divisors, and suppose that $G_P^-$ is
finite. Let $a \in H$ and $M=|\min (\supp( \boldsymbol \beta (a)))|$.
\begin{enumerate}
\item   For any factorization $z\in \ZZ(a)$, there exists a factorization $y \in \Upsilon (a)$  and a chain of
        factorizations $z=z_0,\ldots,z_r=y$ of $a$ such that
        \[
        |z|=|z_0|\leq \cdots\leq |z_r|=|y|\und \mathsf d(z_i,z_{i+1})\leq \max\{M\cdot  \mathsf D(\mathcal S(G_P^-),\mathcal E(G_P^-)),2\}<\infty
        \]
        for all $i\in [0,r-1]$; in fact $\overline{\boldsymbol \beta} (z_0)^+ \leq \overline{\boldsymbol \beta} (z_1)^+ \leq
        \ldots \leq  \overline{\boldsymbol \beta} (z_r)^+$, where $\leq$ is the partial order from the definition of
        $\Upsilon( \boldsymbol \beta (a))$.

\smallskip
\item For any two factorizations $z,\,y\in \Upsilon(a)$ with $\overline{\boldsymbol \beta} (z)^+ =
      \overline{\boldsymbol \beta} (y)^+$, there exists a
      chain of factorizations $z=z_0,\ldots,z_r=y$ of $a$ such that
      \[
      \overline{\boldsymbol \beta} (z)^+ = \overline{\boldsymbol \beta} (z_i)^+ = \overline{\boldsymbol \beta} (y)^+\und \mathsf{d}(z_i,z_{i+1})\leq \max \{\mathsf D(\mathcal S(G_P^-),\mathcal E(G_P^-)),2\}<\infty
      \]
      for all $i\in [0,r-1]$; in particular, $|z|=|z_i|=|y|$ for all $i\in [0,r]$.
\end{enumerate}
\end{theorem}

\begin{proof}
We set $B = \boldsymbol \beta (a)$. By Lemma \ref{3.3}, it suffices
to prove the assertion for $\mathcal B (G_P)$ and $B$. As $0$ is a
prime divisor of $\mathcal B(G_P)$, we may w.l.o.g. assume $0\notin
\supp(B)$.

Note $\mathsf D(\mathcal S(G_P^-),\mathcal E(G_P^-))<\infty$ follows from
\eqref{relative-D-is-finite}. Also, for $z_i,\,z_{i+1}\in \ZZ(S)$, we have
$|z_i|\leq |z_{i+1}|$ whenever $z_i^+\leq z_{i+1}^+$, and
$|z_i|=|z_{i+1}|$ whenever $z_i^+=z_{i+1}^+$ (where  $\leq$ is the
partial order from the definition of $\Upsilon(B)$). Let $z\in
\ZZ(B)$ and let $y\in \Upsilon(B)$ with $z^+\leq y^+$. We will
construct a chain of factorizations $z=z_0,\ldots,z_r$ of $B$ such
that $z_i^+\leq z_{i+1}^+$, either $z_r=y$ or $z^+<z_r^+$,  and \ber\label{ippi}\mathsf
d(z_i,z_{i+1})&\leq& M\cdot\mathsf D(\mathcal S(G_P^-),\mathcal E(G_P^-))<\infty\;\mbox{ (when } z_i^+<z_{i+1}^+\mbox{
)} \\ \mathsf d(z_i,z_{i+1})&\leq& \mathsf D(\mathcal S(G_P^-),\mathcal E(G_P^-))<\infty\;\mbox{ (when }z_i^+=z_{i+1}^+\mbox{
)},\label{ippy}\eer for $i\in [0,r-1]$. Since both parts of the
proposition follow by iterative application of this statement, the proof will be complete
once we show the existences of such a chain of factorizations
$z=z_0,\ldots,z_r=y$.

Since $z^+\leq y^+$, we have \ber\nn z&=&A_1\cdot \ldots \cdot A_n\\
\nn y&=&(B_{1,1}\cdot \ldots\cdot B_{1,k_1})\cdot (B_{2,1}\cdot
\ldots \cdot B_{2,k_2})\cdot \ldots \cdot (B_{n,1}\cdot\ldots \cdot
B_{n,k_n})\eer with $A_j,\,B_{j,i}\in \mathcal A(G_0)$ and
$A_j^+=B_{j,1}^+ \cdot \ldots \cdot B_{j,k_j}^+$, for $j\in [1,n]$
and $i\in [1,k_j]$. Then $A_j^+=B_{j,1}^+ \cdot \ldots \cdot B_{j,k_j}^+$ and $\sigma(A_j)=\sigma(B_{j,i})=0$, for all $j$ and $i$. Thus, for $j\in [1,n]$,  let
\[
T_j= (A^-_j , (B^-_{j,1}\cdot \ldots \cdot B^-_{j,k_j}))\in \mathcal E(G_P^-).
\]
For each $j\in [1,n]$, let
$$T_{j,1}\cdot \ldots\cdot T_{j,l_j}\in \ZZ(\mathcal E(G_P^-))$$ be a factorization of $T_j$ with each $T_{j,i}\in \mathcal A(\mathcal E(G_P^-))$. Now let \be\label{tiddybit}T=\prod_{j=1}^n \prod_{i=1}^{l_j}T_{j,i}\in
\ZZ(\mathcal E(G_P^-)).\ee However, since $z,\,y\in \ZZ(B)$
both factor the same element $B$, we in fact have $$\pi(T)\in
\mathcal S(G_P^-).$$ Let $T=T'T''$ where $T'|T$ is the maximal length sub-factorization with all atoms dividing $T'$ from $\mathcal S(G_P^-)$.

If $T''=1$, then
$A_j=\prod_{i=1}^{k_j}B_{j,i}$ for every $j\in [1,n]$. In view of $A_j,B_{j,i} \in \mathcal{A}(G_P)$, we get $k_j=1$ for every $j \in [1,n]$, that is $z=y$, and so there is nothing to
show.
Therefore we may assume $T''$ is nontrivial and proceed by induction
on $|z|$ and then $|T''|$, assuming \eqref{ippi} and \eqref{ippy} hold
for $z'\in \ZZ(B)$ when $z^+< {z'}^+$ or when $z^+={z'}^+$ and
$|R''|<|T''|$, where $R''$ is defined for $z'$ as $T''$ was for $z$.

Let $W=\prod_{j\in J} \prod_{i\in I_j}T_{j,i}$ be a nontrivial subsequence of $T''$, where $J\subset [1,n]$ and $I_j\subset [1,l_j]$ for $j\in J$, such that $\pi(W)\in \mathcal S(G_P^-)$. Note, since $\pi(T')\in \mathcal S(G_P^-)$ (by definition) and since $\pi(T)\in \mathcal S(G_P^-)$ (by \eqref{tiddybit}), we have $\pi(T'')\in \mathcal S(G_P^-)$, whence we may w.l.o.g. assume $|W|\leq \mathsf D(\mathcal S(G_P^-),\mathcal E(G_P^-))$ (in view of the definition of the relative Davenport constant). Write $W=\prod_{j\in J} W_j$ with each $W_j=\prod_{i\in I_j}T_{j,i}\in \ZZ(\mathcal E(G_P^-))$. Moreover, for $j\in J$, let $\pi(W_j)=(X_j,Y_j) \in \mathcal{E}(G_P^-)$.

Define a new factorization $z_1=z_1'\cdot\ldots\cdot z_n'\in
\ZZ(G_P^-)$ by letting $z_j'=A_j$ for $j\notin J$ and letting
$z_j'\in \ZZ(A_jX_j^{-1}Y_j)$ for $j\in J$---by construction $X_j$
is a subsequence of $A_j$, and since $(X_j, Y_j)\in \mathcal
E(G_P^-)$, we have $\sigma(X_j)=\sigma(Y_j)$, and thus
$\sigma(A_jX_j^{-1}Y_j)=\sigma(A_j)=0$ for all $j\in J$, so $z_1$ is
well defined.  Also, since $\pi(W)=\pi(\prod_{j\in J} W_j)\in
\mathcal S(G_P^-)$, it follows (by definition of $\mathcal
S(G_P^-)$) that $$\prod_{j\in J}X_j=\prod_{j\in J}Y_j,$$ and thus
$z_1\in \ZZ(B)$. Moreover, by construction, we have $z^+\leq z_1^+$,
and by Lemma \ref{Lambert}, we have $|B_j|\leq M$ for all $j$. Thus
\be\label{whash}\mathsf d(z,z_1)\leq M|J|\leq M|W|\leq M\cdot
\mathsf D(\mathcal S(G_P^-),\mathcal E(G_P^-)).\ee Additionally, if
$z\in \Upsilon(B)$, then $z^+\leq z_1^+$ implies that
$z^+=z_1^+=y^+$, whence $|z|=|z_1|$ and $|z_j'|=1$ for all $j$, in
which case  the estimate \eqref{whash} improves to $$\mathsf
d(z,z_1)\leq |J|\leq |W|\leq \mathsf D(\mathcal S(G_P^-),\mathcal
E(G_P^-)).$$ Finally, if $z^+=z_1^+$, then, by construction, the
sequence $R=R'R''$---whose role for $z_1$ is analogous to the role
of $T=T'T''$ for $z$---can be defined so that $R''=T''W^{-1}$, in
which case $|R''|<|T''|$. Consequently, applying the induction
hypothesis to $z_1$ completes the proof.
\end{proof}

\medskip
\begin{corollary} \label{final-cor}
Let $H$ be a Krull monoid and $\varphi \colon H\to \mathcal{F}(P)$ a
cofinal divisor homomorphism into a free monoid such that the class
group $G= \mathcal{C}(\varphi)$ is an  infinite cyclic group that we
identify with $\mathbb Z$. Let $G_P \subset G$ denote the set of
classes containing prime divisors, and suppose that $G_P^-$ is
finite.

Let $a \in H$ with $\max \mathsf L(a)=|\boldsymbol \beta (a)^+| +
\mathsf v_0 \bigl( \boldsymbol \beta (a) \bigr)$ and let $M=|\min
(\supp( \boldsymbol \beta (a)))|$. Then, for any factorization $z\in
\ZZ(a)$ and any factorization $y\in \ZZ(a)$ with $|y|=|\max \mathsf
L(a)|$, there exists a chain of factorizations $z=z_0,\ldots,z_r=y$
of $a$ such that $|z|=|z_0|\leq \cdots\leq |z_r|=|y|$ and
\[\mathsf d(z_i,z_{i+1})\leq \max\{M \cdot \mathsf D(\mathcal S(G_P^-),\mathcal E(G_P^-)), 2\} \leq \max\{|\min G_P|\cdot \mathsf D(\mathcal S(G_P^-),\mathcal E(G_P^-)),2\} <\infty\] for all $i\in [0,r-1]$.
\end{corollary}

\begin{proof}
This follows from directly from Theorem \ref{prop-chain} in view
of \eqref{Upsilon-equivalence}.
\end{proof}

\medskip
We end this section with a result showing that the assumption $\max
\mathsf L(a)=|\boldsymbol \beta (a)^+|  + \mathsf v_0 \bigl(
\boldsymbol \beta (a) \bigr)$ holds for a large class of $a\in H$. We formulate the result
in the setting of zero-sum sequences. Since $\B(G_P)$ is factorial
when $M=|\min G_P|\leq 1$, the assumption $M\geq 2$ below is purely
for avoiding distracting technical points in the statement and
proof.

\medskip
\begin{proposition}
Let $G_0\subset \Z \setminus \{0\}$ be a condensed set with $|G_0| \ge 2$.
Let $B\in \B(G_0)$ be such that, for  $M=|\min
(\supp(B))|$, we have $M\geq 2$ and $\min (\supp(B)^+)\geq M(M^2-1)$.
Then, at least one of the following statements holds{\rm \,:}
\begin{itemize}
\item[(a)]  There exists a subset $A\subset \supp(B^-)$ and a factorization $z\in \ZZ(B)$
such that $\langle \supp(B^+)\rangle \not\subset \langle A\rangle $
(in particular, $\langle A\rangle \neq \Z$) and every atom $U|z$ has
\be\label{thenumber}\vp_x(U)\leq 2M-2\quad \mbox{ for all }\quad
x\in \supp(B)\setminus A.\ee
\item[(b)]\begin{itemize}\item[(i)] $\max \mathsf L(B)=|B^+|$, and

\item[(ii)]  for any factorization $z\in \ZZ(B)$, there exists a chain of factorizations $z=z_0,\ldots,z_r$ of $B$ such that
            $$|z|=|z_0|<\cdots<|z_r|= |B^+|\und \mathsf d(z_i,z_{i+1})\leq M^2$$ for all $i\in [0,r-1]$.
\end{itemize}
\end{itemize}
\end{proposition}

\begin{proof}
We assume (a) fails and show that (b) follows. Note, by Lemma \ref{Lambert}, that $\vp_x(U)\leq
M\leq 2M-2$ holds for any atom $U\in \mathcal A(G_0)$ and $x\geq 0$,
whence \eqref{thenumber} can only fail for some $x\in G_0^-$. To
establish (i) and (ii), we need only show that, given an arbitrary
factorization $z\in \ZZ(B)$ with $|z|<|B^+|$, there is another
factorization $z'\in \ZZ(B)$ with $|z|<|z'|$ and $\mathsf
d(z,z')\leq M^2$. We proceed to do so.

Let $z\in \ZZ(B)$ with $|z|<|B^+|$. Then there must exists some atom
$U_0|z$ such that $|U_0^+|\geq 2$. Let $A\subset \supp(B)$ be all
those $a$ for which there exists some atom $V|z$ with $\vp_a(V)\geq
2M-1$. We must have \be\label{rex}\langle \supp(B^+)\rangle\subset \langle A\rangle ,\ee else (a) holds. Let
$a_1,\ldots,a_t\in A$ be those elements such that $\vp_a(U_0)\leq
M-2$, let $a_{t+1},\ldots,a_{|A|}$ be the remaining element of $A$
and, for all $i \in [1, t]$, let $U_i|z$ be an atom with
$\vp_{a_i}(U_i)\geq 2M-1$. Note that $U_i\neq U_0$ for $i\leq t$
since otherwise $$2M-1\leq \vp_{a_i}(U_i)=\vp_{a_i}(U_0)\leq M-2\leq
2M-2,$$ a contradiction. Also, $t<|A|\leq M$ since otherwise
$$2M(M^2-1)\leq 2\min( \supp(B^+))\leq\sigma(U_0^+)=-\sigma(U_0^-)\leq
M(2M-2),$$ a contradiction.

We proceed to describe a procedure to swap only negative integers
between the $U_i$ which results in new blocks
$U'_0,U'_1,\ldots,U'_t\in \B(G_0)$ with $U'_0U'_1 \cdot \ldots \cdot
U'_t = U_0U_1 \cdot \ldots \cdot U_t$,  with ${U'^+_i}=U^+_i$ for all
$i$, and with $U'_0$ not an atom. Once this is done, then, letting
$z_i\in \ZZ(U'_i)$, we can define $z'$ to be
$$z'=z_0z_1\cdot \ldots \cdot z_tU_0^{-1}U_1^{-1} \cdot \ldots \cdot U_t^{-1}z.$$ Then $|z'|>|z|$ in view of $U'_0$ not being an atom, while, in view of $t\leq |A|-1\leq M-1$ and Lemma \ref{Lambert}, we have $$\mathsf d(z,z')\leq \sum_{i=0}^{t}|U_i^+|\leq (t+1)M\leq M^2.$$ Thus the proof of (i) and (ii) will be complete once we show such a process exists.

Observe, for $i\in [1,t]$, that we can exchange $a_i^{c_{i,j}}|U_i$ for $c_{i,j}^{a_i}|U_0$
provided there is some term $c_{i,j}\in \supp(U^-_0)$ with
$\vp_{c_{i,j}}(U_0)\geq a_i$ and $\vp_{a_i}(U_i)\geq c_{i,j}$, and this
will result in two new zero-sum subsequences obtained by only
exchanging negative terms. The idea in general is to repeatedly and
simultaneously perform such swaps for the $a_i$ using disjoint
sequences
\be\label{whiff1}\prod_{i=1}^{t}\bigl(c_{i,1}^{a_i}c_{i,2}^{a_i}
\cdot \ldots \cdot c_{i,r_i}^{a_i}\bigr) \Bigm|U_0a_{t+1}^{-M+1} \cdot \ldots
\cdot a_{|A|}^{-M+1}\ee with
\be\label{whiff2}\sum_{j=1}^{r_i-1}|c_{i,j}|<M-1\quad\mbox{ but }\quad
\sum_{j=1}^{r_i}|c_{i,j}|\geq M-1\ee for all $i \in [1, t]$, and let
$U'_0,U'_1,\ldots,U'_t$ be the resulting zero-sum sequences. Then
$\vp_{a_i}(U'_0)\geq M-1$ for $i\geq t+1$ by construction, and
$\vp_{a_i}(U'_0)\geq \sum_{j=1}^{r_i}|c_{i,j}|\geq M-1$ for $i\leq t$;
consequently, in view of $\min(\supp(B^+))\geq M(M^2-1)\geq (M-1)^2$ and
$|{U'_0}^+|=|U_0^+|\geq 2$, we see that we can apply Lemma
\ref{breakapart-lemma} to $U'_0$, whence \eqref{rex} and
$\vp_a(U'_0)\geq M-1$ for $a\in A$ imply that $U'_0$ cannot be an
atom, and hence the $U'_i$ have the desired properties. Thus it
remains to show that a sequence satisfying \eqref{whiff1} and
\eqref{whiff2} exists and that each $a_i$, for all $i \in [1, t]$,
has sufficient multiplicity in $U_i$.

Note that \eqref{whiff2} and the definition of $a_i\in A$ imply
$$\sum_{j=1}^{r_i}|c_{i,j}|\leq \sum_{j=1}^{r_i-1}|c_{i,j}|+|c_{i,r_i}|\leq
M-2+M\leq \vp_{a_i}(U_i)$$ for all $i \in [1, t]$. Thus the
multiplicity of each $a_i$ in $U_i$ is large enough to perform such
simultaneous swaps. Also,
\be\label{stern}\big|\sigma \bigl(\prod_{i=1}^{t}(c_{i,1}^{a_i}c_{i,2}^{a_i}
\cdot \ldots \cdot c_{i,r_i}^{a_i})\bigr)\big|\leq
\sum_{i=1}^{t}(2M-2)|a_i|.\ee We turn our attention now to showing
\eqref{whiff1} and \eqref{whiff2} hold.

We can continue to remove subsequences
$c_{i,j}^{a_i}|U_0a_{t+1}^{-M+1} \cdot \ldots \cdot a_{|A|}^{-M+1}$
until the multiplicity of every term is less than $M$. But this
means a sequence satisfying \eqref{whiff1} and \eqref{whiff2} can be
found, in view of the estimate \eqref{stern}, provided
\[
|\sigma({U}^-_0)|-(M-1)\sum_{i=t+1}^{|A|}|a_i|-M(M-1)|\supp(B^-)|\geq \sum_{i=1}^{t}(2M-2)|a_i| \,.
\]
However, if this fails, then we have (since $|U_0^+|\geq 2$) \ber\nn 2M(M^2-1)&\leq& 2  \min (\supp(B^+))\leq \sigma(U_0^+)=-\sigma(U_0^-)=|\sigma(U_0^-)| \\
&< & \sum_{i=1}^{t}(2M-2)|a_i|+(M-1)\sum_{i=t+1}^{|A|}|a_i|+M(M-1)|\supp(B^-)|\nn\\
\nn &<& (2M-2)\sum_{i=1}^{|A|}|a_i| + M(M^2-1)\leq
(2M-2)\sum_{i=1}^{M}i+M(M^2-1)\\\nn &=& 2M(M^2-1),\eer a
contradiction, completing the proof.
\end{proof}


\providecommand{\bysame}{\leavevmode\hbox
to3em{\hrulefill}\thinspace}
\providecommand{\MR}{\relax\ifhmode\unskip\space\fi MR }
\providecommand{\MRhref}[2]{%
  \href{http://www.ams.org/mathscinet-getitem?mr=#1}{#2}
} \providecommand{\href}[2]{#2}

\end{document}